\documentclass{article}
\usepackage{tikz, hyperref}
\usepackage{amsmath,amscd}
\usepackage{amssymb}
\usepackage{upref}
\usepackage{multirow}
\usepackage{multicol}
\usepackage{url}
\usepackage[all]{xy}
\usepackage{kbordermatrix}

\usepackage{color}
\usepackage{makeidx}
\makeindex
\usepackage{theorem}
\newtheorem{theorem}{Theorem}
\newtheorem{lemma}[theorem]{Lemma}
\newtheorem{proposition}[theorem]{Proposition}
\newtheorem{corollary}[theorem]{Corollary}
{\theorembodyfont{\rmfamily}%
  \newtheorem{example}[theorem]{Example}
   }
\newenvironment{proof}{\noindent\textit{Proof.}}
{\QED\vskip\theorempostskipamount} 
\newenvironment{proofof}[1]{\noindent\textit{Proof
    \protect{#1}.}}
                       {\QED\vskip\theorempostskipamount}
\def\petitcarre{\vrule height4pt width 4pt depth0pt}
\def\QED{\relax\ifmmode\eqno{\hbox{\petitcarre}}\else{%
  \unskip\nobreak\hfil\penalty50\hskip2em\hbox{}\nobreak\hfil
  \petitcarre
  \parfillskip=0pt \finalhyphendemerits=0\par\smallskip}
  \fi}

\newcommand\A{\mathcal{A}}

\def\G{\mathcal{G}}

\newcommand\cL{\mathcal{L}}
\newcommand\RR{\mathcal{R}}

\newcommand{\Orb}{\mathcal O}

\newcommand{\XS}{\mathsf X}

\newcommand{\N}{\mathbb{N}}
\newcommand{\Z}{\mathbb{Z}}

\def\un(#1){\underline{#1}\,}
\newcommand{\edge}[1]{\stackrel{#1}{\rightarrow}}
\newcommand{\longedge}[1]{\stackrel{#1}{\longrightarrow}}

\newcommand{\dpower}[1]{\stackrel{\leftarrow}{#1}}

\newcommand{\odpower}{\overleftarrow}

\DeclareMathOperator{\mex}{mex}
\DeclareMathOperator{\Card}{Card}

\definecolor{ivoire}{rgb}{0.99,0.99,0.8}

\definecolor{light-gray}{gray}{0.7}
\usepackage{xspace}
\newcommand{\resp}{{resp.}\xspace}
\usepackage{calc}
\newcounter{hours}\newcounter{minutes}

\numberwithin{theorem}{section}
\numberwithin{equation}{section}
\numberwithin{figure}{section}
\numberwithin{table}{section}

\definecolor{lime}{HTML}{A6CE39}
\DeclareRobustCommand{\orcidicon}{%
	\begin{tikzpicture}
	\draw[lime, fill=lime] (0,0)
	circle [radius=0.16]
	node[white] {{\fontfamily{qag}\selectfont \tiny ID}};
	\draw[white, fill=white] (-0.0625,0.095)
	circle [radius=0.007];
	\end{tikzpicture}
	\hspace{-2mm}
}
\foreach \x in {A, ..., Z}{%
 \expandafter\xdef\csname orcid\x\endcsname{\noexpand%
 \href{https://orcid.org/\csname orcidauthor\x\endcsname}{\noexpand\orcidicon}}
}

\title{Decidable problems in substitution shifts}
\author{Marie-Pierre B\'eal,\orcidA{}%
  Dominique Perrin \\
   Universit\'e Gustave Eiffel, CNRS, LIGM, F-77454 Marne-la-Vall\'ee, France,\\
  and Antonio Restivo\\
   Universit\`a di Palermo}
\begin{document}

\maketitle

\begin{abstract}
  In this paper, we investigate the structure of the most general kind of substitution shifts, including non-minimal ones, and allowing erasing morphisms. We prove the decidability of many properties of these morphisms with respect to the shift space generated by iteration, such as aperiodicity, recognizability
  and (under an additional assumption) irreducibility, or minimality.
\end{abstract}
\tableofcontents
\section{Introduction}
Substitution shifts are an important class of symbolic dynamical systems
which has been studied extensively. However, most results have
been formulated for primitive substitution
shifts, which are generated by primitive morphisms or, slightly more
generally, for minimal substitution shifts. 

In this paper, we investigate the more general situation,
allowing in particular the morphisms to be erasing.
Our work is a continuation of the investigations
of \cite{Yuasa2007}
(who pleasantly writes in his introduction that
`the complement of primitive morphisms is uncharted territory')
and of \cite{Shimomura2018}. Our work is also related
with early results on the so-called D0L-systems,
initiated by Lindemayer (see~\cite{RozenbergSalomaa1980}).
It has also important connections with the study
of numeration systems and automatic sequences (see~\cite{AlloucheShallit2003}).

We investigate  the structure
of these shift spaces in terms of periodic points, fixed points
or subshifts. A striking feature is that all these properties
(such as, for example, the existence of periodic points)
are found to be decidable.

We show that, if the shift $\XS(\sigma)$ generated by an endomorphism
$\sigma$ is non-empty,
it contains  fixed
points of some power of $\sigma$
(Theorem~\ref{propositionFixedPointMinimal2})
and we describe them. We show that there is a finite and computable
number of orbits of fixed points.
We extend the notion of fixed point to that of quasi-fixed point,
which corresponds to the orbits stable under the action of the morphism.
We generalize to quasi-fixed points the results concerning
fixed points (Proposition \ref{propositionAdmissibleQuasiFP}). These
results play a role in the characterization
of periodic morphisms developed below.

We also investigate periodic points in a substitution shift.
We show that their existence is decidable
(Theorem~\ref{theoremDecidableAperiodic}). This
generalizes a result of Pansiot \cite{Pansiot1986} which implies
the same result for primitive morphisms.
We show that both the property of being aperiodic (no periodic points)
(Theorem~\ref{theoremDecidableAperiodic}), and
the property of being periodic (all
points periodic) (Theorem~\ref{theoremDecidablePeriodic})
are decidable.

Recognizability is an important property of morphisms.
A morphism $\sigma\colon A^*\to B^*$ is recognizable
in a shift space $X$ for a point $y\in B^\Z$ if
there is at most one pair $(k,x)$ with
$0\le k<|\sigma(x_0)|$ and $x\in X$ such that
$y=S^k(\sigma(x))$.
It is shown in \cite{BealPerrinRestivo2021} that
every endomorphism $\sigma$ is recognizable for aperiodic points
on $\XS(\sigma)$ (see Section~\ref{sectionRecognizability}
for the background of this result). We show here
that it is decidable whether $\sigma$ is fully recognizable in $\XS(\sigma)$
(Theorem~\ref{theoremDecidabilityRecognizable}).

Two basic notions concerning shift spaces (and, more generally
topological dynamical systems) are irreducibility and minimality.

Concerning irreducibility, we prove a partial result,
which implies its decidability under an additional hypothesis
(Theorem~\ref{theoremIrreducible}).

We  next prove that, under an additional assumption,
it is decidable whether a substitution shift
is minimal (Theorem~\ref{theoremDecidableMinimal}).
The proof uses a characterization of minimal substitution shifts
(Theorem~\ref{theoremCaractMinimal})
proved in~\cite{DamanikLenz2006}.
We also generalize to arbitrary morphisms several results
known for more restricted classes. In particular, we prove
that every minimal substitution shift is conjugate to
a primitive one (Theorem~\ref{theoremMaloneyRust}),
a result proved in~\cite{MaloneyRust2018} for non-erasing morphisms,
and also a particular case of a result of~\cite{DurandLeroy2020}.

We investigate in the last section the notion of quasi-minimal shifts
which are such that the number of subshifts is finite.
We prove, generalizing a result proved in~\cite{BertheSteinerThuswaldnerYassawi2019}
for non-erasing morphisms, that every substitution shift is quasi-minimal
(Proposition~\ref{lemmaLx}).

Our results show that all properties that we have considered are
decidable for a substitution shift. This raises two questions.
The first one is the possibility of natural undecidable problems
for substitution shifts. Actually, it is not known if the
equality or the conjugacy of two substitution shifts is decidable
(see \cite{DurandLeroy2020} for the proof of the decidability of conjugacy
for minimal substitution shifts).
The second question is whether there is a decidable logical language
in which these properties can be formulated. Such
a logical language exists for automatic sequences (which are
defined using fixed points of constant length morphisms).
It is an extension of the first-order logic of integers with
addition (also known as Presburger arithmetic)
with the function $V_p(x)$ giving the highest  power
of $p$ dividing $x$ (see \cite{BruyereHanselMichauxVillemaire1994}).
This logic is decidable and thus all properties which can
be expressed in this logic are decidable. This
has been applied in the software Walnut (see \cite{mousavi2016automatic}
and \cite{Shallit2021}).

Our paper is organized as follows. In the first section
(Section~\ref{sectionShiftSpaces}), we recall
the basic notions of symbolic dynamics.
In Section~\ref{sectionMorphisms}, we focus on substitution shifts.
We define the $k$-th higher block presentation of a non-erasing morphism
and prove that it defines a shift which is conjugate to the
original one (Proposition~\ref{propositionsigma_k}).
We prove several lemmas that are used in the sequel
(Lemma~\ref{lemmaGrowing} and \ref{lemmaSigma}).
We finally prove that the language of a
substitution shift
is decidable (Proposition~\ref{newPropositionDecidabilityL(X(sigma))}).

In Section~\ref{sectionGrowing}, we prove an important property
of substitution shifts, namely that there is a finite and computable
set of orbits of points formed of non-growing letters
(Proposition~\ref{lemmaNonGrowing}). This
result is used several times below.

In Section~\ref{sectionFixedPoints}, we prove several results
concerning fixed points of an endomorphism $\sigma$. We first describe
the one-sided fixed points in the one-sided shift associated to  $\XS(\sigma)$.
We prove that there is a finite and computable number of orbits
of these points (Proposition~\ref{lemmaOneSidedFixedPoint})
and we give a complete description of these fixed points
(Proposition~\ref{propositionAdmissibleOneSided}).
We next prove the analogous results for two-sided fixed points
(Theorem~\ref{propositionFixedPointMinimal2} and
Proposition~\ref{propositionAdmissibleTwoSided}).
We finally introduce the notion of quasi-fixed point,
which corresponds to the stability of orbits
under the action of the morphism $\sigma$. We characterize the
quasi-fixed points in the shift $\XS(\sigma)$ and
show, using a result from \cite{AlloucheShallit2003}, that there is a finite and computable number of orbits
of these points (Proposition~\ref{propositionQuasiFixedPoints}).

In Section~\ref{sectionPeriodic}, we investigate periodic points
in substitution shifts. We prove the decidability
of the existence of fixed points (Theorem~\ref{theoremDecidableAperiodic})
as well as that of the existence of non-periodic points, using the notion of quasi-fixed point (Theorem~\ref{theoremDecidablePeriodic}). This generalizes a result proved in \cite{Pansiot1986}
for primitive morphisms.

In Section~\ref{sectionErasable}, we prove a normalization result
concerning morphic sequences $x=\phi(\tau^\omega(u))$
image by a morphism $\phi$ of a fixed point of a morphism $\tau$.
We show that one may always replace
the pair $(\tau,\phi)$ by a pair $(\sigma,\theta)$ such that $\sigma$ is non-erasing
and $\theta$ is alphabetic (Theorem~\ref{theoremNormalization}).
This is essentially a result of Cobham \cite{Cobham1968}
(see also \cite{Pansiot1983}, \cite{CassaigneNicolas2003}).

In Section~\ref{sectionRecognizability}, we show that
it is decidable whether a morphism $\sigma$ is fully recognizable
in the shift $\XS(\sigma)$ (Theorem~\ref{theoremDecidabilityRecognizable}).

In Section~\ref{sectionIrreducible}, we characterize the
substitution shifts which are irreducible (Theorem~\ref{theoremIrreducible}),
under the additional assumption that $\cL(\XS(\sigma))=\cL(\sigma)$.

In Section \ref{sectionMinimal}, we show that, under an additional assumption,
it is decidable
whether a substitution shift is minimal (Theorem~\ref{theoremMaloneyRust}).
This uses a characterization of minimal substitution shifts
which had been proved in~\cite{DamanikLenz2006}.
We also prove, generalizing the result obtained in \cite{BertheSteinerThuswaldnerYassawi2019}
for non-erasing morphisms, that every substitution
shift has a finite number of subshifts (Proposition~\ref{lemmaLx}).
\paragraph{Acknowledgments} The authors thank Jeffrey Shallit
and Fabien Durand
for reading a draft of the manuscript and providing additional
references. They are also grateful to Herman Goulet-Ouellet for
reading the following version with great care and making many
useful comments. The referee is also thanked for helping
us to improve the final version.

This work was supported by the Agence Nationale
de la Recherche (ANR-22-CE40-0011).
\section{Shift spaces}\label{sectionShiftSpaces}

We refer to the classical sources, such as \cite{LindMarcus1995}
for the basic definitions concerning symbolic dynamics.
Concerning words, we refer to \cite{Lothaire1983}.

Given a finite alphabet $A$,
we let $A^*$ denote the set of words on $A$.
For a word $w\in A^*$, we let $|w|$ denote its length and,
for $a\in A$, by $|w|_a$ the number of occurrences of $a$ in $w$.
The alphabet $A$ is assumed to be finite.
We let $\varepsilon$ denote the empty word
and $A^+$ the set of nonempty words on $A$.
A nonempty word is \emph{primitive} if it is not a power of a shorter word.
We let $w^*$ denote the submonoid generated by a word $w$.

For a two-sided infinite sequence
$x\in A^\Z$, let $x=\cdots x_{-1} x_0x_1\cdots$ with $x_i\in A$.
For $i<j$, we denote $x_{[i,j]}=x_i\cdots x_{j}$
and $x_{[i,j)}=x_i\cdots x_{j-1}$. Similarly, $x_{(-\infty, j]}=\cdots x_{j-1}x_j$
with all obvious variants.

The set $A^\Z$   is a metric space
for the distance defined for $x,y\in A^\Z$ by $d(x,y)=2^{-r(x,y)}$
where $r(x,y)=\min\{n\ge 1\mid x_{[-n,n]}\ne y_{[-n,n]}\}$.
Since $A$ is finite, the topological space $A^\Z$ is compact.

  We let $S_A$ (or simply $S$) denote the shift on $A^\Z$. It is defined
  by $y=S(x)$ if
  \begin{equation}
    y_n=x_{n+1}\label{eqShift}
  \end{equation}
  for all $n\in \Z$. It is a homeomorphism
  from $A^\Z$ onto itself. We also let $S_A$ denote the shift transformation
  defined on $A^\N$ by \eqref{eqShift} for all $n\ge 0$.
For $x\in X$, we write sometimes $Sx$ instead of $S(x)$.

A \emph{shift space} on $A$ is a subset $X$ of $A^\Z$
which is closed for the topology
and invariant under the shift. An \emph{isomorphism}
(also called a conjugacy)
from  a shift space $X$ on $A$ onto a shift space $Y$
on $B$ is a homeomorphism $\varphi$ from $X$ onto $Y$
such that $\varphi\circ S_A=S_B\circ\varphi$.

For $x\in A^\Z$, we let
$x^+\in A^\N$ denote the right-infinite sequence $x_0x_1\cdots$
and by $x^-\in A^{-\N}$ the left-infinite sequence $\cdots x_{-2}x_{-1}$.
Thus $x^+=(x_n)_{n\ge 0}$ and $x^-=(x_{n-1})_{n\le 0}$.
Conversely, given $y\in A^{-\N}$ and $z\in A^{\N}$, we let
$x=y\cdot z$ denote the two-sided infinite sequence such that
$x^-=y$ and $x^+=z$.

 For a shift space $X$, we denote
 $X^+=\{x^+\mid x\in X\}$. A set $Y\subset A^\N$ \footnote{We use $\subset$
   everywhere instead of $\subseteq$}is called a \emph{one-sided shift space} if $Y=X^+$ for some shift space $X$.
 Equivalently, $Y$ is a one-sided shift space if it is closed
 and such that $S(Y)=Y$.

 Note that $X^+$ determines $X$ since
\begin{equation}
  X=\{x\in A^\Z\mid x_nx_{n+1}\cdots\in X^+\mbox{ for every $n\in \Z$}\}.
\end{equation}

The \emph{orbit} of a point $x$ in a shift space $X$
is the set
\begin{displaymath}
  \Orb(x)=\{S^n(x)\mid n\in\Z\}.
\end{displaymath}
When
$X$ is a one-sided shift space,
the orbit of $x$ is $\{y\in X\mid S^n(x)=S^m(y) \mbox{ for some $m,n\ge 0$}\}$.
 The  \emph{positive orbit} of a point $x$
is the set $\Orb^+(x)=\{S^n(x)\mid n\ge 0\}$.

A sequence $x\in A^\Z$ is \emph{periodic} if there is an $n\ge 1$ such that $S^n(x)=x$. The integer $n$ is called a \emph{period} of $x$. For a nonempty
word $w$, we let $w^\omega$ denote the right-infinite
sequence $ww\cdots$ and by $^\omega w$ the left-infinite
sequence $\cdots ww$. We denote $w^\infty={}^\omega w\cdot w^\omega$.
A periodic
sequence is of the form $w^\infty$ 
for a unique primitive word $w$. Its minimal period is then $|w|$.

\begin{example}
  Let $Y$ be the orbit of $^\omega a\cdot b^\omega$ and let
  $X$ be the topological
  closure of $Y$.
  Then $X=Y\cup \{a^\infty\}\cup\{b^\infty\}$ and
  $X^+=\{a^nb^\omega\mid n\ge 0\}\cup \{a^\omega\}$.
  \end{example}

A shift space is \emph{periodic} if it is formed of periodic sequences.
It is \emph{aperiodic} if it contains no periodic sequence.

We let $\cL(X)$ denote the \emph{language} of a shift space $X$, which is the
set of all factors of the sequences $x\in X$. We let
$\cL_n(X)$ (resp. $\cL_{\ge n}(X)$) denote the set of words of length $n$
(resp. $\ge n$) in $\cL(X)$. We also let $\cL(x)$ denote the set
of factors of a sequence $x$. Thus $\cL(X)=\cup_{x\in X}\cL(x)$.

A language $L$ is \emph{factorial} if it contains the factors of its
elements. A word $u\in L$ is \emph{extendable} in $L$
if there  are letters $a,b$ such that $aub\in L$.
A language $L$ is \emph{extendable} if every word in $L$ is extendable.
For every shift space $X$,
the language $\cL(X)$ is factorial and it is extendable. Conversely,
for every factorial extendable language $L$, there is a unique
shift space $X$ such that $\cL(X)=L$, which is
the set of sequences $x$ such that $\cL(x)\subset L$. In particular,
$\cL(X)$ determines $X$.

A word $w\in\cL(X)$ is called \emph{right-special} if there is
more than one letter $a\in A$ such that $wa\in\cL(X)$.
Symmetrically, $w\in\cL(X)$ is \emph{left-special} if
there is
more than one letter $a\in A$ such that $aw\in\cL(X)$.

A nonempty shift space $X$ is \emph{irreducible} if for every $u,v\in \cL(X)$
there is some $w$ such that $uwv\in\cL(X)$.

The following is classical (see~\cite[Proposition 2.1.3]{DurandPerrin2021}).
\begin{proposition}\label{propositionTopologicalTransitive}
  A shift space $X$ is irreducible if and only if there is some $x\in X$
  with a dense positive orbit.
\end{proposition}
A nonempty shift space is \emph{minimal} if it does not contain properly
a nonempty shift space. A shift space is minimal if and only
if the orbit of every point is dense.

A minimal shift is either aperiodic or periodic (and in this case formed
of one periodic orbit). It is aperiodic if and only if the set
of right-special (\resp left-special) words is infinite.

A nonempty shift space is \emph{uniformly recurrent}
if for every $w\in \cL(X)$, there is an $n\ge 1$ such that
every word in $\cL_n(X)$ contains $w$.
The following result is classical (see \cite[Proposition 5.2]{Queffelec2010}).
\begin{proposition}\label{UR=Minimal}
A shift space is minimal if and only if it is uniformly recurrent.
\end{proposition}
Every minimal shift $X$ is irreducible. Indeed, if $X$
is minimal, let $u,v\in\cL(X)$. There is an $n\ge 1$
such that every word in $\cL_n(X)$ contains $u$ and $v$.
Then every word in $\cL_{2n}(X)$ contains a word of the form $uwv$.
The converse is not true since, for example, the full shift
is irreducible but not minimal.

Given a shift space $X$ and $w\in\cL(X)$, a \emph{return word}
to $w$ is a nonempty word $u$ such that $wu$ is in $\cL(X)$ and has exactly
two occurrences of $w$, one as a prefix and the other one as a suffix.
An irreducible shift space is minimal if and only if the set of return words
to $w$ is  finite for every $w\in\cL(X)$.
Indeed, if this condition is satisfied, the shift $X$ is clearly
uniformly recurrent.

Let $X$ be a shift on $A$ and let $k\ge 1$ be an integer.
Consider an alphabet $A_k$ in one-to-one correspondence with
$\cL_k(X)$ via a bijection $f_k:\cL_k(X)\to A_k$.
The map $\gamma_k:X\rightarrow A_k^\Z$ defined for $x\in X$ by $y=\gamma_k(x)$
if for every $n\in\Z$
\begin{displaymath}
y_n=f_k(x_n\cdots x_{n+k-1})
\end{displaymath}
is the $k$-th \emph{higher block code}
 on $X$ (see Figure~\ref{figureSlidingBlock}).
The set $X^{(k)}=\gamma_k(X)$
 is a shift space on $A_k$, called the $k$-th
\emph{higher block presentation}
of $X$ (one also uses the term of coding by \emph{overlapping blocks} of length $k$).
\begin{figure}[hbt]
\centering

\tikzset{node/.style={draw,minimum size=0.4cm,inner sep=0pt}}
	\tikzset{title/.style={minimum size=0.5cm,inner sep=0pt}}
\begin{tikzpicture}
  \node[node](k) at (-0.2,1){$\quad x_n\quad $};
  \node[node](k+1) at (1,1){$\quad x_{n+1}\ \ $};
  \node[node](...) at (2,1){$\ \cdots\ $};
  \node[node](n+k-1) at (3,1){$\ x_{n+k-1}\ $};
  \node[node](y) at (-0.2,0){$\quad y_n\quad$};
  \draw[ ->,>=stealth] (k) edge node {} (y);
  \node[title] at (0,0.5){$f_k$};
\end{tikzpicture}
\caption{The $k$-th higher block code.}\label{figureSlidingBlock}
\end{figure}
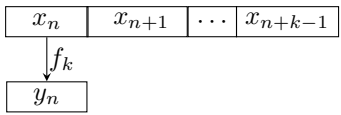

The following result follows easily from the definitions.
\begin{proposition}\label{propositionPi_k}
The higher block code is an isomorphism of shift spaces and the inverse
of $\gamma_k$ is
is the map $y \mapsto x$ such that, for all $n$, $x_n$ is the first letter of the
word $u$ such that $y_n = f_k(u)$.
\end{proposition}

\begin{example}
  Let $X=\{a,b\}^\Z$. Set $A_2=\{r,s,t,u\}$ with $f_2\colon r\mapsto aa, s\mapsto ab, t\mapsto ba, u\mapsto bb$. Then $X^{(2)}$ is the set of labels of
    bi-infinite paths in the graph of Figure~\ref{figureX2}.
    \begin{figure}[hbt]
      \centering
\tikzset{node/.style={circle,draw,minimum size=0.4cm,inner sep=0pt}}
\tikzset{title/.style={minimum size=0.5cm,inner sep=0pt}}
\tikzstyle{every loop}=[->,shorten >=1pt,looseness=8]
\tikzstyle{loop left}=[in=220,out=140,loop]
\tikzstyle{loop right}=[in=-40,out=40,loop]
        \begin{tikzpicture}
          \node[node](1)at(0,0){$1$};
          \node[node](2)at(2,0){$2$};
          \draw[->,left,loop left](1)edge node{$r$}(1);
          \draw[->,above,bend left](1)edge node{$s$}(2);
          \draw[->,below,bend left](2)edge node{$t$}(1);
          \draw[->,right,loop right](2)edge node{$u$}(2);
          \end{tikzpicture}
        \caption{The 2nd-higher block presentation of $\{a,b\}^\Z$.}\label{figureX2}
        \end{figure}
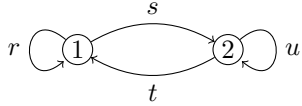
  \end{example}

\section{Morphisms}\label{sectionMorphisms}
A \emph{morphism} $\sigma\colon A^*\to B^*$ is a monoid morphism
from the free monoid $A^*$ to the free monoid $B^*$.
We allow  a letter $a\in A$ to be erased, that
is, to have $\sigma(a)=\varepsilon$.
We say that $\sigma$ is \emph{non-erasing} if $\sigma(a)\ne\varepsilon$
or all $a\in A$.

For $x=x_0x_1\cdots\in A^\N$, set
$\sigma(x)=\sigma(x_0)\sigma(x_1)\cdots$ if the right hand side
is infinite. Since
an infinity of the $x_i$ can possibly be erased, it
is a finite or right-infinite
sequence. Similarly, if $x=\cdots x_{-1}x_0\in A^{-\N}$,
$\sigma(x)=\cdots \sigma(x_{-1})\sigma(x_0)$ is finite or left-infinite. Next, for $x\in A^\Z$,
$\sigma(x)=\sigma(x^-)\cdot \sigma(x^+)$ is either finite,
left-infinite, right-infinite, or two-sided infinite.

Thus, a  morphism $\sigma\colon A^*\to B^*$
extends
to a map from $A^\N$ to $B^*\cup B^\N$. If $\sigma(a_n)=\varepsilon$
for $n\ge k$, then $\sigma(a_0a_1\cdots)=\sigma(a_0\cdots a_{k-1})$.
Otherwise $\sigma(a_0a_1\cdots)$ is the right-infinite sequence
$\sigma(a_0)\sigma(a_1)\cdots$.
Similarly, it extends also to a map from $A^\Z$ to
$B^*\cup B^{-\N}\cup B^\N\cup B^\Z$. For $x\in A^\Z$,
if $\sigma(x)$ is in $B^\Z$
(in particular if $\sigma$ is non-erasing),
  one has $\sigma(x)=\sigma(x^-)\cdot\sigma(x^+)$.

If $\sigma\colon A^*\to B^*$ is a morphism, we denote $|\sigma| = \sum_{a \in A} |\sigma(a)|$
and  define the \emph{size} of $\sigma$
as $|\sigma| + \Card(A)$.

Let $\sigma\colon A^*\to B^*$ be a morphism. A \emph{$\sigma$-representation}
of $y\in B^\Z$ is a pair $(x,k)$ of a sequence $x\in A^\Z$
and an integer $k$  such that
\begin{equation}
  y=S^k(\sigma(x)).\label{eqsigmaRep}
\end{equation}
The $\sigma$-representation $(x,k)$ is
\emph{centered}
if $0\le k<|\sigma(x_0)|$.

Note that, in particular, a centered $\sigma$-representation $(x,k)$ is
such that $\sigma(x_0)\ne\varepsilon$. 

The notion of centered representation is a normalization. Indeed, if $y=S^k(\sigma(x))$,
there is a unique centered $\sigma$-representation $(x',k')$ of $y$
such that $x'$ is a shift of $x$, namely $(S^nx,k)$
with $k=|\sigma(x_0\cdots x_{n-1})|+k'$.

\paragraph{Endomorphisms}
We now define notions specific to an \emph{endomorphism}
of $A^*$, that is, a morphism $\sigma$ from $A^*$ into itself.
In this case, one may iterate $\sigma$.

Given an endomorphism $\sigma\colon A^*\to A^*$, we let
$\cL(\sigma)$ denote the \emph{language} of $\sigma$.
It is formed of all factors of the words $\sigma^n(a)$
for $n\ge 0$ and $a\in A$. We let $\cL_n(\sigma)$
(resp. $\cL_{\ge n}(\sigma)$) denote the set
of words of length $n$ (resp. $\ge n$)
in $\cL(\sigma)$.

The shift $\XS(\sigma)$ \emph{associated}
to $\sigma$ is the set $\XS(\sigma)$ of two-sided infinite sequences $x\in A^\Z$
with all their factors in $\cL(\sigma)$. Such a shift space
is called a \emph{substitution shift}.

\begin{example}
  The morphism $\sigma\colon a\mapsto ab, b\mapsto a$ is
  called the \emph{Fibonacci morphism}. The associated shift
  $\XS(\sigma)$ is called the \emph{Fibonacci shift}. 
  \end{example}

We can also associate to $\sigma$ the set $\XS^+(\sigma)$
of right-infinite sequences having all their factors in $\cL(\sigma)$.

Note that it is not true in general that $\XS^+(\sigma)=\XS(\sigma)^+$,
because $X^+(\sigma)$ is not a one-sided shift,
as shown in the following example.
\begin{example}
  Let $\sigma\colon a\mapsto ab,b\mapsto b$. Then
  $\XS^+(\sigma)=ab^\omega\cup b^\omega$
  but, since $\XS(\sigma)=b^\infty$, we have $\XS(\sigma)^+=b^\omega$.
  \end{example}
We have the inclusion $\cL(\XS(\sigma))\subset \cL(\sigma)$
but the converse is not always true. This (unpleasant) phenomenon
occurs, for example, when $\sigma$ is the identity, since
then $\cL(\sigma)=A$ but $\XS(\sigma)$
(and thus $\cL(\XS(\sigma))$) is empty.

A letter $a\in A$ is \emph{erasable} if there is some $n\ge 1$
such that $\sigma^n(a)=\varepsilon$ (the term \emph{mortal}
is also used, see~\cite{Vitanyi1974} or \cite{AlloucheShallit2003}). A word is erasable if
all its letters are erasable. 

The \emph{mortality exponent}
of an erasable word $w$, denoted $\mex(w)$
 is the least integer $n$ such that $\sigma^n(w)=\varepsilon$. Note
 that $\mex(w)\le\Card(A)$, that is
 $\sigma^{\Card(A)}(w)=\varepsilon$.
 The \emph{mortality exponent}
  of $\sigma$,
  denoted $\mex(\sigma)$,
  is the maximal value of the mortality exponents of erasable letters.
  
A morphism is non-erasing
if no letter is erasable.

A word $w\in A^*$ is \emph{growing} if the sequence $|\sigma^n(w)|$ is unbounded.
A word is growing if some of its letters is growing. The
morphism $\sigma$ itself is said to be \emph{growing} if all letters
are growing.

An endomorphism $\sigma\colon A^*\to A^*$ is \emph{primitive} if
 there is an $n\ge 1$ such that for every $a,b\in A$, the
 letter $b$ appears in $\sigma^n(a)$.  

 For a primitive morphism $\sigma$, except the trivial case $A=\{a\}$
 and $\sigma(a)=a$, every letter is growing
 and $\cL(\sigma)=\cL(\XS(\sigma))$ (see~\cite{DurandPerrin2021} for example).

 An endomorphism $\sigma\colon A^*\rightarrow A^*$
 is \emph{minimal} if $\XS(\sigma)$ is a minimal shift space.
 The following
 statement is well known.
 \begin{proposition}\label{propositionPrimitiveisMinimal}
   Every primitive endomorphism not reduced to the identity
   on a one-letter alphabet is minimal.
 \end{proposition}
 For a proof, see~\cite[Proposition 5.5]{Queffelec2010}
 or \cite[Proposition 2.4.16]{DurandPerrin2021}.

An endomorphism $\sigma$ is
\emph{aperiodic} (resp. \emph{periodic}) if the shift $\XS(\sigma)$ is aperiodic
(resp. periodic).

\begin{example}\label{exampleFibonacci}
  The Fibonacci morphism $\sigma\colon a\mapsto ab, b\mapsto a$
  is primitive. It is also aperiodic. Indeed, it is easy to verify
  by induction on $n$ that $\sigma^n(a)$ is left-special and thus
  that the number of left-special words is infinite.
  Since $X(\sigma)$ is minimal, this implies that it is aperiodic.
\end{example}
\begin{example}\label{exampleMorse}
  The \emph{Thue-Morse} morphism $\sigma\colon a\mapsto ab,b\mapsto ba$
  is also primitive and aperiodic (see \cite{DurandPerrin2021} for example).
  The shift $\XS(\sigma)$ is called the \emph{Thue-Morse shift}.
  It is well known that the language $\cL(\sigma)$ does not contain
  cubes, that is words of the form $www$, with $w$ nonempty.
  \end{example}
  \begin{example}\label{example(a,bab)}
    The morphism $\sigma\colon a\mapsto a,b\mapsto bab$ is not primitive.
    It is periodic and minimal
    because $\sigma(ab)=abab$ and thus $\XS(\sigma)=\{(ab)^\infty,(ba)^\infty\}$.
    \end{example}

  We now present a classical construction (see~\cite{Queffelec2010})
  which, for every $k\ge 1$  allows
one to replace a  non-erasing morphism
$\sigma$ by a morphism $\sigma_k$ acting on $k$-blocks.

Let $\sigma\colon A^*\to A^*$ be a non-erasing endomorphism.
For every integer $k\ge 1$, let $ u\in \cL_k(\sigma)\mapsto
\langle u \rangle \in A_k$
be a bijection from $\cL_k(\sigma)$ onto an alphabet $A_k$.


We define an endomorphism
$\sigma_k:A_k^*\rightarrow A_k^*$ as follows. Let
$u\in \cL_k(\sigma)$ and let $a$ be the first letter of $u$.
Set $s=|\sigma(a)|$.
To compute $\sigma_k(\langle u\rangle)$, we first compute
the word $\sigma(u)=b_1b_2\cdots b_\ell$.
Note that, since $\sigma$ is non-erasing,
we have $\ell\ge |\sigma(a)|+k-1=s+k-1$.
 We define
\begin{equation}
\sigma_k(\langle u \rangle )=\langle b_1b_2\cdots b_k\rangle\langle b_2b_3\cdots b_{k+1}\rangle\cdots
\langle b_s\cdots b_{s+k-1}\rangle.\label{equationsigma_k}
\end{equation}
The morphism $\sigma_k$
is called the \emph{$k$-th higher block presentation of $\sigma$}.

Let $\pi_k:A_k^*\rightarrow A^*$ be the morphism defined by
$\pi_k(\langle u\rangle)=a$ where
$a$ is  the first letter of $u$. 
Then we have for each $n\ge 1$ the following commutative diagram
which expresses the fact that $\sigma_k^n$ is the
counterpart of $\sigma^n$ for $k$-blocks. 

\begin{equation}
\begin{CD}
A_k^* @>{\sigma_k^n}>> A_k^*\\
@VV{\pi_k}V               @VV{\pi_k}V\\
A^* @>{\sigma^n}>> A^*
\end{CD}\label{DiagramExtension}
\end{equation}

Indeed, for every $\langle u \rangle\in A_k$, let $a=\pi_k(\langle u \rangle)$ and let
$s=|\sigma(a)|$. Set $\sigma(u)=b_1b_2\cdots b_\ell$.
We have by definition of $\sigma_k$,
\begin{displaymath}
  \pi_k\circ\sigma_k(\langle u \rangle)=b_1b_2\cdots b_s=\sigma(a)=\sigma\circ\pi_k(\langle u \rangle).
  \end{displaymath}
Since $\pi_k\circ\sigma_k$ and $\sigma\circ\pi_k$ are
morphisms, this proves that the diagram \eqref{DiagramExtension}
is commutative for $n=1$ and thus for all $n\ge 1$.

\begin{proposition}\label{propositionsigma_k}
  Let $\sigma$ be a non-erasing endomorphism and $k$ a positive integer.
  Then $\XS(\sigma_k)=\XS(\sigma)^{(k)}$.
\end{proposition}
\begin{proof}
  Since the diagram \eqref{DiagramExtension} is commutative,
  $\pi_k$ is a conjugacy from $\XS(\sigma_k)$ onto $\XS(\sigma)$.
  On the other hand, by Proposition~\ref{propositionPi_k},
  it is a conjugacy from $\XS(\sigma)^{(k)}$
  onto $\XS(\sigma)$.
\end{proof}
\begin{example}\label{exampleBlockPresentation}
Let $\sigma\colon a\mapsto ab,b\mapsto a$ be the Fibonacci morphism. We have
$\cL_2(X)=\{aa,ab,ba\}$. Set $A_2=\{x,y,z\}$ and let
$\langle aa \rangle=x, \langle ab \rangle= y, \langle ba \rangle= z$.  We have
$\sigma_2(x)=\sigma_2(\langle aa \rangle)= \langle ab \rangle \langle
ba \rangle=yz$. Similarly, we have $\sigma_2(y)=yz$ and
$\sigma_2(z)=x$.
\end{example}

\paragraph{Ordered trees}
An \emph{ordered tree} is a graph $T$ in which
the vertices are labeled.
The end $t$ of an edge starting in a vertex $s$ is called a \emph{child}
of $s$ and
$s$ is the \emph{parent}
of $t$. Each vertex has a label and the set
of children of each vertex is totally ordered.
There is a particular vertex called the \emph{root}
of $T$
and there is a unique path from the root to each vertex.
The \emph{level}
of a vertex $s$ is the length of the path from the root to $s$.
A \emph{leaf} of a tree is a vertex without
children.

Let $\sigma\colon A^*\to A^*$ be a morphism.
For each $a\in A$ and $n\ge 0$, we define
an ordered tree $T_\sigma(a,n)$ labeled by $A$,
called the \emph{derivation tree} of $a$
at order $n$. 
 The root of $T_\sigma(a,n)$ is labeled by $a$ and
 has no child if $n=0$. If $n\ge 1$, set
 $\sigma(a)=b_1b_2\cdots b_k$ with $b_i\in A$. The root of $T(a,n)$ has
  $k$ children  $r_1,r_2,\ldots,r_k$ which are the roots of  trees $T_i$ with
  $1\le i\le k$ and each $T_i$ is isomorphic to $T(b_i,n-1)$.
  The order on the children is $r_1<r_2<\ldots <r_k$.

  \begin{example}
    Let $\sigma\colon a\mapsto ab,b\mapsto a$ be the Fibonacci
    morphism. The tree $T_\sigma(a,2)$ is represented in Figure~\ref{figureT(a,2)}.
    \begin{figure}[hbt]
      \centering
\tikzset{node/.style={draw,minimum size=0.4cm,inner sep=0pt}}
	\tikzset{title/.style={minimum size=0.5cm,inner sep=0pt}}
\begin{tikzpicture}
  \node[node](a0)at(3.25,3){$a$};
  \node[node](a1)at(2.5,2){$a$};\node[node](b1)at(4,2){$b$};
  \node[node](a2)at(2,1){$a$};\node[node](b2)at(3,1){$b$};\node[node](a22)at(4,1){$a$};

  \draw[->](a0.south)--node{}(2.5,2.2);\draw[->](a0.south)--node{}(b1.north);
  \draw[->](a1.south)--node{}(a2.north);\draw[->](2.5,1.8)--node{}(3,1.2);
  \draw[->](b1.south)--node{}(a22.north);
\end{tikzpicture}
\caption{The tree $T_\sigma(a,2)$.}\label{figureT(a,2)}
      \end{figure}
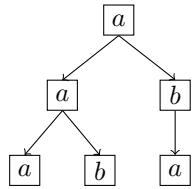
    \end{example}

\paragraph{Some useful lemmas}
The following three lemmas deal with questions that
concern non-growing or erasable letters.

We associate to an endomorphism $\sigma\colon A^*\to A^*$ the multigraph $G(\sigma)$ on $A$
with $|\sigma(a)|_b$ edges from $a$ to $b$. We let $M(\sigma)$ denote
the adjacency matrix of $G(\sigma)$, that is, the matrix $M$
such that $M_{a,b}=|\sigma(a)|_b$.

It is decidable in linear time
whether a letter is erasable or growing. To see this, we build the multigraph
$G(\sigma)$. A strongly connected component
of $G(\sigma)$ is \emph{trivial} if it has one vertex and no edges.

A letter $a$ is erasable if all paths going out of $a$ end only in
trivial strongly connected components of $G(\sigma)$.

A letter $a$ is growing if and only if either there is a path from $a$ in
$G(\sigma)$
to a non-trivial strongly connected component that is not reduced to
a cycle or there is a path from $a$ to a non-trivial strongly
connected component from which there is also a path to another non-trivial strongly
connected component.

\begin{example}
  Let $\sigma\colon a\mapsto ab, b\mapsto a$ be the Fibonacci morphism.
  The graph $G(\sigma)$ is shown in Figure~\ref{figureG(sigma)Fibo}.
    \begin{figure}[hbt]
    \centering
\tikzset{node/.style={circle,draw,minimum size=0.4cm,inner sep=0pt}}
\tikzset{title/.style={minimum size=0.5cm,inner sep=0pt}}
\tikzstyle{every loop}=[->,shorten >=1pt,looseness=8]
\tikzstyle{loop left}=[in=220,out=140,loop]
\tikzstyle{loop right}=[in=-40,out=40,loop]
    \begin{tikzpicture}
      \node[node](a)at(0,0){$a$};
      \node[node](b)at(2,0){$b$};

      \draw[->,>= stealth](a)edge[loop left]node{}(a);
      \draw[->,>= stealth,bend left](a)edge node{}(b);
      \draw[->,>= stealth,bend left](b)edge node{}(a);
      \end{tikzpicture}
    \caption{The graph $G(\sigma)$.}\label{figureG(sigma)Fibo}
    Since it is strongly connected and not reduced to a cycle,
    $\sigma$ is growing.
    \end{figure}
    
        The matrix $M(\sigma)$ is
    \begin{displaymath}
M(\sigma)=\begin{bmatrix}1&1\\1&0\end{bmatrix}
      \end{displaymath}
\end{example}

\begin{lemma}\label{lemmaGrowing}
  Let $\sigma\colon A^*\to A^*$ be an endomorphism. If 
  $u\in A^*$ is  growing, then $\lim_{n\to+\infty}|\sigma^n(u)|=+\infty$.
\end{lemma}
\begin{proof}
  We may assume that $u=a$ is a growing letter.
  We argue by induction on $\Card(A)$. The statement is true if $\Card(A)=1$.
  If the strongly connected component in $G(\sigma)$
   of $a$ 
  contains strictly a cycle, then $\lim_{n\to+\infty}|\sigma^n(a)|=+\infty$. Next, if the strongly
  connected component $C(a)$ of $a$
  is reduced to a cycle, we have
  $\sigma^i(a)=vaw$ for some $i\ge 1$ and $v,w$ containing no letter
  of $C(a)$. Either $v$ or $w$, say $v$, has to be growing. Thus the result
  holds by induction hypothesis applied to $v$.
  
  Finally, if the strongly connected component  of $a$
  is trivial,
  the result holds by induction hypothesis.
\end{proof}

The following result  is
  proved in \cite[Lemma 2.2]{BealPerrinRestivo2021}.

\begin{proposition}\label{propositionErasing}
For every endomorphism $\sigma$ there is a finite
  number of words formed of erasable letters in $\cL(\sigma)$.
  \end{proposition}

\begin{lemma}\label{lemmaSigma}
  Let $\sigma\colon A^*\to A^*$ be a morphism. For every $x$ in $\XS(\sigma)$,
  the sequence $\sigma(x)$ is in $\XS(\sigma)$. The map
  $x\mapsto\sigma(x)$ is continuous on $\XS(\sigma)$.
\end{lemma}
\begin{proof}
 We have to prove that $\sigma(x)$ is two-sided infinite.
  By Proposition~\ref{propositionErasing}, $x$ has an infinite number of non-erasable letters on
  the left and on the right, and $\sigma(x)$ is two-sided infinite.
  By Proposition~\ref{propositionErasing} again, there
  is an integer $k\ge 1$ such that every word in $\cL_k(\sigma)$
  contains a non-erasable letter. This implies that $\sigma(\cL_{kn}(\sigma))$
  is contained in $\cL_{\geq n}(\sigma)$, whence the second statement.
\end{proof}
Note that, since $\XS(\sigma)$ is compact, the map $\sigma$ is actually
uniformly continuous on $\XS(\sigma)$.

We prove the following concerning $\sigma$-representations.
\begin{proposition}\label{propositionBKM}
  Let $\sigma\colon A^*\to A^*$ be a morphism. Every point $y$ in $\XS(\sigma)$
  has a $\sigma$-representation $y=S^k(\sigma(x))$ with $x$ in $\XS(\sigma)$.
\end{proposition}
\begin{proof}
Let $k=|\sigma|$ and let $y$ be in $\XS(\sigma)$. For every $n > 2k$, there is an integer
$m \geq 1$ such that $y_{[-n,n]}$ is a factor of $\sigma^m(a)$ for some letter $a
\in A$.

By a compactness argument, there is an integer $0 \leq i < k$ such that
for every $n > 2k$, there are words $u_n, v_n$ with $u_n v_n \in \cL(\sigma)$ such that
$y_{[-n +k, -i]}$  is a suffix of $\sigma(u_n)$ and $y_{[-i, n-k]}$
is a prefix of $\sigma(v_n)$. Further, $|u_n| \geq (n-k -i)/k$ and
$|v_n| \geq (n -k + i)/k$.

By a compactness argument again, we get that there is a point $x \in X(\sigma)$
such that $y = S^i(\sigma(x))$.
\end{proof}

\paragraph{The language $\cL(\sigma)$}
A language $L\subset A^*$ is \emph{recognizable} if
it can be recognized by a finite automaton.
The following result appears in \cite[Lemma 3]{Salo2017}
(see also~\cite{CartonThomas2002}).
We reproduce the proof for the sake of completeness.
\begin{lemma}\label{lemmaDecidabilityL(sigma)}
  Let $\sigma\colon A^*\to A^*$ be a morphism.
  It is decidable, for a given recognizable language $L$
  whether $L\cap \cL(\sigma)$ is empty. In particular,
  $\cL(\sigma)$ is decidable. If $L$ is
  moreover factorial, it is decidable whether
  $L\cap \cL(\sigma)$ is finite.
\end{lemma}
\begin{proof}
  Since $L$ is recognizable, so is $A^*LA^*$.
  Let $\A=(Q,I,T)$ be a finite automaton recognizing $A^*LA^*$.
  For $w\in A^*$, denote by $\varphi(w)$ the binary relation on $Q$
  defined by
  \begin{displaymath}
    \varphi(w)=\{(p,q)\in Q\times Q\mid p\longedge{w}q\}.
  \end{displaymath}
  The map $w\mapsto \varphi(w)$ is a morphism from $A^*$ into
  the finite monoid of binary relations on $Q$.
  Let $\psi_n=\varphi\circ \sigma^n$. Since $\sigma$ and $\varphi$
  are morphisms, so is each $\psi_n$. Since $\varphi(A^*)$ is a finite
  set, there is a finite number of morphisms
  $\psi_n$ and thus there are $n<m$ such that $\psi_n=\psi_m$. Note
  that $\psi_n=\psi_m$ implies $\psi_{n+r}=\psi_{m+r}$
  because $\psi_{n+r}=\psi_n\circ\sigma^{r}=\psi_m\circ\sigma^r=\psi_{m+r}$.
  Then
  $L\cap \cL(\sigma)\ne\emptyset$ if and only if there is
  some $k\le m$ such that $(i,t)\in \psi_k(A)$
  for some $i\in I$ and $t\in T$.
  
  Assume now that $L$ is factorial. We consider
  this time an automaton recognizing $L$,
  in which we may assume that $I=T=Q$.
  We define $\varphi$ and $\psi_n$ as above.
  Then $L \cap \cL(\sigma)$  is infinite if and only if
  there is a growing
  letter $a$ such that $\psi_k(a)\ne\emptyset$
  for every $k<m$.

  Finally, one can decide if $w$ is in $\cL(\sigma)$ since
  it is equivalent to $\{w\}\cap\cL(\sigma)\ne\emptyset$.
  \end{proof}

\paragraph{The language $\cL(\XS(\sigma))$}
The following results deal with substitution shifts that fail
to satisfy the condition $\cL(\XS(\sigma))=\cL(\sigma)$.
We have already seen that the equality holds
 for primitive morphisms with $\Card(A)\ge 2$.
More generally,  one has
$\cL(\XS(\sigma))=\cL(\sigma)$ if and only if
$\cL(\sigma)$ is extendable, which for a non-erasing morphism,
occurs if and only if every letter is extendable.

In the general case, let us first note the following property.
\begin{proposition}\label{propositionCaractSubstitution}
  Let $\sigma\colon A^*\to A^*$ be a morphism.
  One has $\cL(\sigma)=\cL(\XS(\sigma))$
  if and only if every letter $a\in A$ is in $\cL(\XS(\sigma))$.
  \end{proposition}
\begin{proof}
  Assume that the condition is satisfied. Every $w\in \cL(\sigma)$ is a factor
  of some $\sigma^n(a)$ for $a\in A$ and $n\ge 0$. Let $x\in \XS(\sigma)$
  be such that $a\in\cL(x)$. By Lemma~\ref{lemmaSigma}, we have
  $\sigma^n(x)\in \XS(\sigma)$. Since $w\in\cL(\sigma^n(x))$, we obtain
  $w\in \cL(\XS(\sigma))$. The converse is obvious.
  \end{proof}

The following lemma shows that, for a morphism $\sigma\colon A^*\to A^*$,
it is decidable whether a word is a factor of some $\sigma^n(a)$
for $a\in A$ for an infinite number of $n$.
 Its proof uses the technique of the
proof of Lemma  \ref{lemmaDecidabilityL(sigma)}.
\begin{lemma}\label{lemmaDecidabilityL(sigma)Plus}
  Let $\sigma\colon A^*\to A^*$ be a morphism and let $u\in A^*$.
  It is decidable whether there  an infinite number
  of integers $n\ge 0$ such  $u$ is a factor of some $\sigma^n(a)$
  with $a$ in $A$.
\end{lemma}
\begin{proof}
  Let $u\in A^*$ be a word.
  Let $\A=(Q,I,T)$ be a finite automaton recognizing $A^*uA^*$.
  For $w\in A^*$, denote by $\varphi(w)$ the binary relation on $Q$
  defined by
  \begin{displaymath}
    \varphi(w)=\{(p,q)\in Q\times Q\mid p\longedge{w}q\}.
  \end{displaymath}
  The map $w\mapsto \varphi(w)$ is a morphism from $A^*$ into
  the finite monoid of binary relations on $Q$.
  Let $\psi_n=\varphi\circ \sigma^n$. 
 
  Thus $\psi_n(w)=\{(p,q)\in Q\times Q\mid p\xrightarrow{\sigma^n(w)}q\}$.

  There are integers $m<m'$ such that $\psi_{m+r}=\psi_{m'+r}$ for any
  nonnegative integer $r$ (see the proof of Lemma \ref{lemmaDecidabilityL(sigma)}).
  
  Then $u$ is factor 
  of an infinite number of some $\sigma^n(a)$
  with $a$ in $A$ if and only there is
  some $m \leq k\leq m'$ such that $(i,t)\in \psi_k(A)$
  for some $i\in I$ and $t\in T$.
\end{proof}

\begin{lemma} \label{lemmaExtensible}
  Let $\sigma\colon A^*\to A^*$ be a morphism. There is a constant $K$
  such that a word $u$ belongs to $\cL(\XS(\sigma))$ if
  and only if there are words $w,z$ of length $K$
  such that $wuz$ is a factor of words $\sigma^n(a)$
  with $a$ in $A$  for an infinite number of $n\ge 0$.
\end{lemma}
\begin{proof}
  Define
  \begin{displaymath}
    N = \max \{|\sigma^n(a)| \mid a \text{ non-growing}\},\mbox{ and }
    M = \max \{|\sigma^n(a)| \mid a \text{ erasing}\}.
  \end{displaymath}
  
  Let $r = \max \{|\sigma^{i}| \mid 0 \leq i \leq \Card(A)^2\}$.
  We set $K = (N+M)r$.

 Let $u$ be a word for which there are words $w,z$ of length $K$
  such that $wuz$ is factor of words $\sigma^n(a)$
  with $a$ in $A$ for an infinite number of $n\ge 0$.
  We set $\sigma^n(a) = w^{(n)}wuzz^{(n)}$
  with $w^{(n)}$ as short  as possible. .

  Let $e_n$ (\resp $f_n$) be the first (\resp last)
  growing letter of $\sigma^n(a)$.
  For $n$ large enough, there are integers
  $0 \leq k  \leq k' \leq \Card(A)^2$ such
  that $e_k = e_{k'} = e$ and $f_k = f_{k'} = f$.
  This implies that $e_{k +i} = e_{k'+i}$ and $f_{k +i} = f_{k'+i}$
  for any integer $i$.
  
  Let $\sigma^k(a) = vesft$, $\sigma^{k'-k}(e) = v'ev''$ and
  $\sigma^{k'-k}(f) = t'ft''$ with $v, t, v', t''$ non-growing
  (see Figure~\ref{dessinDecidabilityL(X(sigma))}) or,
  in the case where $e,f$ coincide,
  $\sigma^k(a) = vet$, $\sigma^{k'-k}(e) = v'ev''$ with $v, t, v', v''$ non-growing.
  
  \begin{figure}[hbt]
\centering
    \tikzset{node/.style={draw,minimum size=.4cm,inner sep=0pt}}
    \tikzset{title/.style={minimum width=.4cm,minimum height=.4cm}}
    \begin{tikzpicture}
      \node[node](a)at(0,5){$a$};
      \node[node,text width=2cm](v)at(-2,4){\qquad $v$};
      \node[node](e)at(-.8,4){$e$};
      \node[node,text width=1.2cm](s)at(0,4){\quad $s$};
      \node[node](f)at(.8,4){$f$};
      \node[node,text width=2cm](t)at(2,4){\qquad $t$};
      \node[node,text width=1cm](v')at(-2,3){\quad $v'$};
      \node[node]at(-1.3,3){$e$};
      \node[node,text width=1cm](v'')at(-.6,3){\quad $v''$};
      \node[node,text width=1cm](t')at(.6,3){\quad $t'$};
      \node[node]at(1.3,3){$f$};
      \node[node,text width=1cm](t'')at(2,3){\quad $t''$};
      \node[node,text width=4cm]at(-3.5,1){\hspace{2cm} $w^{(n)}$};
      \node[node,text width=1cm]at(-1,1){\quad $w$};
      \node[node,text width=1cm]at(0,1){\quad $u$};
      \node[node,text width=1cm]at(1,1){\quad $z$};
      \node[node,text width=4cm]at(3.5,1){\hspace{2cm} $z^{(n)}$};

      \draw[->,above](-.2,4.8)-- node{$\sigma^k$}(-3,4.2);
      \draw[->,above](.2,4.8)-- node{}(3,4.2);
      \draw[->,left](-1,3.8)--node{$\sigma^{k'-k}$}(-2.5,3.2);
      \draw[->,left](-.6,3.8)--node{}(-.1,3.2);
      \draw[->](.6,3.8)--node{}(.1,3.2);
      \draw[->](1,3.8)--node{}(2.5,3.2);
      \draw[->,left](-3,3.8)--node{$\sigma^{n-k}$}(-5.5,1.2);
      \draw[->,left](3,3.8)--node{}(5.5,1.2);
    \end{tikzpicture}
    \caption{Decidability of $\cL(\XS(\sigma))$}
    \label{dessinDecidabilityL(X(sigma))}
  \end{figure}
  Since $v, t$ are non-growing, we have $|\sigma^{n-k}(v)| \leq Nr$ and $|\sigma^{n-k}(t)| \leq Nr$. We claim that there is an infinite
  sequence $(i_n)$ such that the lengths of $w^{(i_n)}$ and $z^{(i_n)}$
  tend to infinity. We first prove it for $w^{(n)}$.

  Case 1. 
  If $\sigma^{n-k}(ve)$ is a prefix of length at most
  $|w^{(n)}wu|$
  of $w^{(n)}wuzz^{(n)}$ for an infinite number of $n$,
  then,  since $e$ is growing,
  $\sigma^{n-k}(esft)$ (or $\sigma^{n-k}(et)$ in the case $e=f$) s equal to
  $w'^{(n)}uz'^{(n)}$ with $w'^{(n)}$ suffix of $w^{(n)}$,
$z'^{(n)}$ prefix of $zz^{(n)}$ and $|w'^{(n)}|$ going to the
  infinity.
  
  Case 2. Let us assume that $\sigma^{n-k}(ve)$ is a prefix $w^{(n)}wut^{(n)}$ of
  $w^{(n)}wuzz^{(n)}$ for an infinite number of $n$.
  
   Case 2(i). If $v'$ is non-erasable, then again $\sigma^{n-k}(esft)$ (or $\sigma^{n-k}(et)$) is equal to
   $w'^{(n)}uz'^{(n)}$ with  $w'^{(n)}$ suffix of $w^{(n)}$,
   $z'^{(n)}$ prefix of $zz^{(n)}$ and $|w'^{(n)}|$ going to the
   infinity.
   
   Case 2(ii). Let us assume that $v'$ and $v''$ are erasable. Since $|\sigma^{n-k}(v')| \leq Mr$
   and $|\sigma^{n-k}(v'')| \leq Mr$, this case is impossible.

   Case 2(iii). If $v'$ erasable and $v''$ is non-erasable, since
   $|\sigma^{n-k}(v')| \leq Mr$, $\sigma^{n-k}(e) =
   w'^{(n)}wut'^{(n)}$ with $w'^{(n)}w$ suffix of $w^{(n)}$, $t'^{(n)}$ prefix of $zz^{(n)}$
   and $|w'^{(n)}|$ going to the infinity.

   This proves that there is an
   infinite sequence of integers $(i_n)$ such that the size of $w^{(i_n)}$ tends to infinity.

Similarly there is a infinite subsequence $(j_n)$ of $(i_n)$ such that 
the size of $z^{(j_n)}$ tends to  infinity.
This proves the claim.
   It  implies that $u$ belongs to $\cL(\XS(\sigma))$.
\end{proof}

\begin{proposition} \label{newPropositionDecidabilityL(X(sigma))}
  Let $\sigma\colon A^*\to A^*$ be a morphism. The language
  $\cL(\XS(\sigma))$ is decidable.
\end{proposition}
\begin{proof}
  Let $K = (N+M)r$ be defined as in the proof of Lemma \ref{lemmaExtensible}.
  By Lemma \ref{lemmaExtensible}, a word $u$ belongs to $\cL(\XS(\sigma))$ if
  and only if there are words $w,z$ of length $K$
  such that there is an infinite number of $n$ for which
  $wuz$ is factor of some $\sigma^n(a)$  with $a$ in $A$.
  We thus may check whether there is a word $wuz$ of length $2K + |u|$
  having this property by Lemma \ref{lemmaDecidabilityL(sigma)Plus}.
\end{proof}

Proposition \ref{newPropositionDecidabilityL(X(sigma))} implies that it
is decidable whether a letter belongs to $\cL(\XS(\sigma))$. The
following proposition shows that this can be decided in linear time.

\begin{proposition}\label{lemmaDecidability}
Let $\sigma$ be a morphism. It is decidable in linear time in the size
of $\sigma$ whether a
letter belongs to $\cL(\XS(\sigma))$.
\end{proposition}
\begin{proof}
We use the graph $G(\sigma)$. We will define letters of type $1,2,3$ as follows.

The letters of type 1 are the letters accessible in $G(\sigma)$ from
non-trivial strongly connected components not reduced to a cycle.
Letters of type 1 belong to $\cL(\XS(\sigma))$. Indeed, if $a$ belongs
to a non-trivial strongly connected component of $G(\sigma)$
which is not reduced to a cycle,
there is an integer $k$ such that $\sigma^k(a)$ contains
at least two occurrences of $a$. It follows that $\XS(\sigma)$ contains
a sequence with an infinite number of $a$ on the right and on the left.
Letters accessible from $a$ have the same property.


The letters of type $2$ can be of type $2'$ or $2''$.
The letters of type $2'$ are the letters accessible from a non-trivial strongly
connected component $C$ reduced to a cycle such that there are a
letter $a$ in a trivial component, a letter $b$ in $C$, integers $k,p$ such that $\sigma^k(a)
= ubv$, $\sigma^p(b) =wbz$  with
$u,v,w,z$ satisfying the two following conditions.
\begin{enumerate}
\item[(i)] either $u$ is growing or $w$ is non-erasable
\item[(ii)] either $v$ is growing or $z$ is non-erasable.
\end{enumerate}
The letters of type $2''$ are the letters accessible from a non-trivial strongly
connected component $C$ reduced to a cycle such that there is
a letter $b$ in $C$, an integer $p$ such that $\sigma^p(b) =wbz$ and 
$w$ and $z$ non-erasable. 
The letters of type 2 belong to $\cL(\XS(\sigma))$. Indeed,
if the conditions for the type $2'$ are satisfied,
$\sigma^{k+np}(a) = \sigma^{np}(u) \sigma^{np}(w)\ldots wbv \ldots \sigma^{np}(v)\sigma^{np}(z)$,
making $b$ a letter of $\cL(\XS(\sigma))$.
Next, a letter of
type $2''$ is accessible from a letter $b$ which is in $\cL(\XS(\sigma))$
and it is therefore in $\cL(\XS(\sigma))$.

The letters of type 3 are the letters accessible from a non-trivial strongly
connected component $C'$ itself accessible from another strongly
component $C$ which is a cycle. The letters of type 3 belong to
$\cL(\XS(\sigma))$. Indeed, there is a letter $a$ in $C$, a letter
$b$ in $C'$ and an integer $p$ such that $\sigma^p(a) = uav$ and
$\sigma^p(b) = wbz$, where $b$ is
a letter of $u$ or of $v$.
Le us assume that $b$ is a letter of $v$.
Then $\sigma^{np}(a)= \sigma^{(n-1)p}(u) \cdots
\sigma^p(u)uav\sigma^p(v) \cdots \sigma^{(n-1)p}(v)$.
Each $\sigma^{rp}(v)$ contains the letter $b$. Thus there is a
infinite number of words $xby$ with $|x| = |y| = n$ in $\cL(\sigma)$
and $b$ belongs to some point in $\XS(\sigma)$. Note that letters may have
several types.

We now show that if a letter is neither a letter of type 1 to 3, then it is
not in $\cL(\XS(\sigma))$. Let $c$ be a letter which is neither of type
1 to 3. Assume that $c \in \cL(\XS(\sigma))$.
Let $x$ be a point of $\XS(\sigma)$ containing
the letter $c$. We may assume that $x_0 = c$.

Since $x \in \XS(\sigma)$, for each integer $n$, the word
$x_{[-n,n]}$ is a factor of some $\sigma^{k_n}(a)$ for some fixed
letter $a$. Let $T_\sigma(a,k_n)$ be the derivation tree at order $k_n$ rooted
at $a$. For $n$ large enough, there are integers $0 \leq m < m' \leq \Card(A)$ such that
the unique path $\pi$ going from $a$ to $x_0$ in $T_\sigma(a,k_n)$ goes through a
vertex labeled by
$b$ at the level $m$ and at the level $m'$. Thus $\sigma^m(a) = ubv$
and $\sigma^p(b) = wbz$ with $p = m'- m$.
Since $c$ is not of type 1 the letter $b$ belongs to a non-trivial
component $C$ reduced to a cycle.

If $a$ is in a trivial
component, since $c$ is not of type $2'$, we have either
$u$ non-growing and $w$ erasable, or
$v$ non-growing and $z$ erasable.

Assume that $u$ is non-growing and $w$ is erasable, the other case
being symmetrical.
There is a constant $K$ such that the lengths of all
$\sigma^n(u)$ are bounded by $K$ and a constant $M$ such that $\sigma^{M}(w)  = \varepsilon$.

Thus for a large enough $n$, there is are integers $r_n, i_n$ with $0 \leq i_n < p$ such that
$x_{[-n,0)}$ is a factor of $\sigma^{r_n}(u)\sigma^{i_n}
(\sigma^{Mp}(w) \ldots  \sigma^{2p}(w)\sigma^p(w)w$, a contradiction.

If $a$ is in a non-trivial
component, this component is $C$, since $c$ is not of type 3, and we may
assume that $a = b$. Since $c$ is not of type $2''$, we have either
$w$ or $z$ erasable. Assume that $w$ is erasable.
As above, $x_{[-n,0)}$ is a factor of $\sigma^{r_n}(u)\sigma^{i_n}
(\sigma^{Mp}(w) \ldots  \sigma^{2p}(w)\sigma^p(w)w$ for some integers
$r_n, i_n$ with $0 \leq
i_n < p$, a contradiction.

We now investigate the complexity
of testing whether a letter is in $\cL(\XS(\sigma))$.
The graph $G(\sigma)$ can be built in linear time in the size of
$\sigma$. Letters of type 1, 2, or 3 can be computed in linear time in
the size of $\sigma$. Indeed, it is clear for types 1 and 3.

For  type 2, one can compute for each component $C$ reduced to a
cycle whether there is an integer $p$ such that $\sigma^p(a) = uav$
with $a \in C$, and $u$ erasable (\resp $v$ erasable).
This can be done in linear time in the size of $\sigma$ for all components.
Thus the computation of letters of type $2''$ can be done in
linear time.

The growing letters are computed in linear time in the size of
$G(\sigma)$. Now finding letters $a$ in trivial components with a path
from $a$ to a component reduced to a cycle containing a letter $b$
such that $\sigma^k(a) = ubv$ with $u$ growing (\resp $v$ growing)
can be done in linear time. Thus the computation of letters of type $2'$
can be done in
linear time.
\end{proof}

\begin{example}
  Let $\sigma\colon a\mapsto ab, b\mapsto bc, c\mapsto cc$.
  The graph $G(\sigma)$ is represented in Figure~\ref{figureExample3.17}
  on the left.
    \begin{figure}[hbt]
    \centering
\tikzset{node/.style={circle,draw,minimum size=0.4cm,inner sep=0pt}}
\tikzset{title/.style={minimum size=0.5cm,inner sep=0pt}}
\tikzstyle{every loop}=[->,shorten >=1pt,looseness=8]
\tikzstyle{loop left}=[in=220,out=140,loop]
\tikzstyle{loop above}=[in=120,out=60,loop]
\tikzstyle{loop aboveleft}=[in=140,out=90,loop]
\tikzstyle{loop aboveright}=[in=90,out=40,loop]
\tikzstyle{loop right}=[in=-40,out=40,loop]
\begin{tikzpicture}
      \node[node](a)at(0,0){$a$};
      \node[node](b)at(1.5,0){$b$};
     \node[node](c)at(3,0){$c$};
     
      \draw[->,>= stealth](a)edge[loop left]node{}(a);
      \draw[->,>= stealth](a)edge node{}(b);
      \draw[->,>= stealth](b)edge node{}(c);
      \draw[->,>= stealth](b)edge[loop above]node{}(b);
      \draw[->,>= stealth](c)edge[loop above]node{}(c);
      \draw[->,>= stealth](c)edge[loop right]node{}(c);
      \node[node](a)at(5,0){$a$};
      \node[node](b)at(6.5,.5){$b$};
      \node[node](c)at(6.5,-.5){$c$};
      \draw[->,>= stealth](a)edge[loop left]node{}(a);
      \draw[->,>= stealth](a)edge node{}(b);
      \draw[->,>= stealth](a)edge node{}(c);
      \draw[->,>= stealth](c)edge[loop right]node{}(c);
      \draw[->,>= stealth](b)edge[loop right]node{}(b);
      \end{tikzpicture}
    \caption{The graph $G(\sigma)$.}\label{figureExample3.17}

    \end{figure}
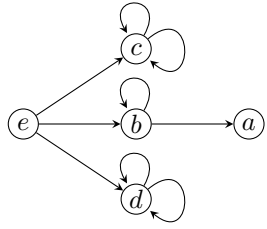  
  Then $c$ is a letter of type 1 and 3 (since it is accessible from
  $b$) and $b$ is a letter of type 3. The letter $a$ does not belong
  to $\cL(\XS(\sigma))$.
\end{example}
\begin{example}
  Let $\sigma\colon a\mapsto bac, b\mapsto b, c\mapsto c$.
  The graph $G(\sigma)$ is represented in Figure~\ref{figureExample3.17}
  on the right.
  Then $a, b, c $ are letters of type $2''$. The point $\cdots bbbb \cdot
  accccc \cdots$ belongs to $\XS(\sigma)$.
\end{example}

\begin{example} \label{example3.19}
  Let $\sigma\colon e\mapsto cbd, c\mapsto cc, d\mapsto dd, b \mapsto
  ba, a \mapsto \varepsilon$.
  The graph $G(\sigma)$ is represented in Figure~\ref{figureExample3.19}
  on the right.
      \begin{figure}[hbt]
    \centering
\tikzset{node/.style={circle,draw,minimum size=0.4cm,inner sep=0pt}}
\tikzset{title/.style={minimum size=0.5cm,inner sep=0pt}}
\tikzstyle{every loop}=[->,shorten >=1pt,looseness=8]
\tikzstyle{loop left}=[in=220,out=140,loop]
\tikzstyle{loop above}=[in=120,out=60,loop]
\tikzstyle{loop aboveleft}=[in=140,out=90,loop]
\tikzstyle{loop aboveright}=[in=90,out=40,loop]
\tikzstyle{loop right}=[in=-40,out=40,loop]
\begin{tikzpicture}
      \node[node](e)at(0,0){$e$};
      \node[node](b)at(1.5,0){$b$};
      \node[node](c)at(1.5,1){$c$};
      \node[node](d)at(1.5,-1){$d$};
      \node[node](a)at(3,0){$a$};
      \draw[->,>= stealth](e)edge node{}(b);
       \draw[->,>= stealth](e)edge node{}(c);
       \draw[->,>= stealth](e)edge node{}(d);
       \draw[->,>= stealth](b)edge node{}(a);
       \draw[->,>= stealth](b)edge[loop above]node{}(b);
      \draw[->,>= stealth](d)edge[loop above]node{}(d);
      \draw[->,>= stealth](d)edge[loop right]node{}(d);
     \draw[->,>= stealth](c)edge[loop above]node{}(c);
      \draw[->,>= stealth](c)edge[loop right]node{}(c);
      \end{tikzpicture}
    \caption{The graph $G(\sigma)$.}\label{figureExample3.19}

    \end{figure}  
  Then $a, b$ are letters of type $2'$ and $c,d$ are of type 1.
  The point $\cdots ccccc \cdot
  baddddd \cdots$ belongs to $\XS(\sigma)$.
\end{example}
Note that it is not possible to characterize the letters of
$\cL(\XS(\sigma))$ on the multigraph $G(\sigma)$. Indeed
If $\sigma$ is the morphism of Example \ref{example3.19}
and $\tau\colon e\mapsto bcd, c\mapsto cc, d\mapsto dd, b \mapsto
  ba, a \mapsto \varepsilon$, then $\sigma$ and $\tau$ have the same
  multigraph. But $b \in \cL(\XS(\sigma))$ while $b \notin \cL(\XS(\tau))$.

It follows from Proposition~\ref{lemmaDecidability} and Proposition \ref{propositionCaractSubstitution} that the
property $\cL(\sigma)=\cL(\XS(\sigma))$ is decidable in linear time.





\medskip

\section{Growing letters}\label{sectionGrowing}

Morphisms $\sigma$ having points in $\XS(\sigma)$ with only non-growing
letters are called \emph{wild} in \cite{MaloneyRust2018}.
The other ones are called \emph{tame}. A characterization
of tame morphisms is given in \cite[Theorem 2.9]{MaloneyRust2018}.

A \emph{finite fixed point} of a morphism $\sigma\colon A^*\to A^*$
is a word $w\in A^*$ such that $\sigma(w)=w$ (we shall come
back to fixed points in the next section).

\begin{proposition}\label{propositionFiniteFixedPoint}
  Let $\sigma\colon A^*\to A^*$ be a morphism. If, for some
  $a\in A$, one has $\sigma(a)=uav$ with $u,v$ erasable,
  then $w=\sigma^{\mex(\sigma)}(a)$ is a finite fixed point of $\sigma$.
\end{proposition}
\begin{proof}
  Since $u,v$ are erasable, we have
  $\sigma^{\mex(\sigma)}(u)=\sigma^{\mex(\sigma)}(v)=\varepsilon$.
  Since
  \begin{displaymath}
    w=\sigma^{\mex(\sigma)-1}(u)\cdots\sigma(u)uav\sigma(v)\cdots\sigma^{\mex(\sigma)-1}(v)
  \end{displaymath}
  we have $\sigma(w)=w$.
  \end{proof}

\begin{proposition}\label{propositionNonGrowingLetters}
  Let $\sigma\colon A^*\to A^*$ be a morphism.
  There are computable integers
  $i\ge 0$ and $p\ge 1$ depending only on $\sigma$ such that
  $\sigma^i(a)=\sigma^{i+p}(a)$ for every non-growing letter $a\in A$.
\end{proposition}
\begin{proof}
    Let $a\in A$ be a non-growing  and non-erasable
  letter. Let $i(a)$
  be the maximal length of a path from $a$ to a cycle in $G(\sigma)$.
  Let $p(a)$ be the least common multiple of the lengths
  of cycles accessible from $a$. Then
  $\sigma^{i(a)+p(a)}=u\sigma^{i(a)}(a)v$ with $u,v$ erasable.

  This shows that for every non-growing and non-erasable letter $a$,
  by definition of $i(a)$,
  one has $\sigma^{i(a)}(a)=u_0a_0u_1\cdots a_ku_k$
  with the component of $a_i$ on a cycle
  and $u_i$ erasable. Each $a_i$ is such that $\sigma^{p(a)}(a_i)=v_ia_iw_i$
  with $v_i,w_i$ erasable. By Proposition~\ref{propositionFiniteFixedPoint}
  applied to $\sigma^{p(a)}$, this implies that $\sigma^{p(a)\mex(\sigma)}(a_i)$
  is a finite fixed point of $\sigma^{p(a)}$. Then
  \begin{eqnarray*}
 \sigma^{i(a)+p(a)+p(a)\mex(\sigma)}(a)&=&\sigma^{p(a)+p(a)\mex(\sigma)}(u_0a_1u_1\cdots a_ku_k)\\
 &=&\sigma^{p(a)\mex(\sigma)}(u_0a_1u_1\cdots a_ku_k)\\
 &=&\sigma^{i(a)+p(a)\mex(\sigma)}(a)
  \end{eqnarray*}
  The statement follows, choosing
  \begin{displaymath}
    i=\max\{i(a)+p(a)\mex(\sigma)\mid  \mbox{$a$ non-growing and non-erasable}\}+\mex(\sigma)
    \end{displaymath}
  and $p$ equal to the least common multiple of the $p(a)$
  for $a$ non-growing.
  \end{proof}

The following statement describes a special type
of fixed points of morphisms that arises only
when there are non-growing letters. It bears
some similarity with the result of \cite{MaloneyRust2018}
but proves additionally the existence of periodic points
without growing letters.

\begin{proposition}\label{lemmaNonGrowing}
Let $\sigma\colon A^*\to A^*$ be a morphism. If $x$ is a point of $\XS(\sigma)$
which has only non-growing letters, then there is an integer $k$ such that
$x = S^k( ^\omega u \cdot w v^\omega)$,
where $u,v,w$ are finite words of lengths bounded by a
computable integer depending
only on $\sigma$.
One can further choose the words $u, v, w$ such that
they are fixed points of some power of $\sigma$.
The finite set of orbits
of these points is effectively computable.
\end{proposition}
\begin{proof}
  Let $x\in \XS(\sigma)$ be a point without growing letters.
  Let $i, p$ be the constants of Lemma \ref{propositionNonGrowingLetters}.
  Let $p'$ be a multiple of $p$ larger than $i$.
  Hence for each non-growing word $u$, $\sigma^{kp'}(u)=\sigma^{p'}(u)$
  for every  $k\ge 1$.

  For every $n\ge 1$, there is some $N \ge 1$
  and $a\in A$ such that $x_{[-n,n]}$ is a factor of $\sigma^{N}(a)$.
  We may assume that $a$ is growing.

  We distinguish three cases.
  We choose $n$ large enough so that $N > p'\Card(A)^2$. 
  (the exponent 2 will
  be needed only in Case 3).

 Case 1. We have, for $N$ large enough , $\sigma^N(a)=qbr$
  with $q$ non-growing $b\in A$ growing and $x_{[-n,n]}$ a factor of $q$.
  
  We can write $N=N_1+kN_2+N_3$ with $N_1,N_2,N_3\le p'\Card(A)$,
  $N_1,N_2$ multiple of $p'$, $N_2>0$ and
  \begin{equation}
    \sigma^{N_1}(a)=p_1cq_1,\quad \sigma^{N_2}(c)=p_2cq_2,\quad
    \sigma^{N_3}(c)=p_3bq_3,\label{equationN1N2N3}
  \end{equation}
  with $p_1,p_2,p_3$ non-growing.
  Note that, since $b$ is growing, $c$ is growing
  and is the first growing letter of
  $\sigma^{N_1}(a)$ and of $\sigma^{N_2}(c)$. Thus it
  is uniquely determined.
  \begin{figure}[hbt]
    
    \centering
    \tikzset{node/.style={draw,minimum size=.4cm,inner sep=0pt}}
    \tikzset{title/.style={minimum width=.4cm,minimum height=.4cm}}
    \begin{tikzpicture}
      
      \node[node](e)at(3,3){$a$};
      \node[node,text width=1cm](p1)at(2.3,2){};\node[title](pp1)at(2.3,2){$p_1$};
      \node[node](ch)at(3,2){$c$};
      \node[node,text width=1cm](q1)at(3.7,2){};\node[title](qq1)at(3.7,2){$q_1$};

      \node[node,text width=1cm](p1b)at(.3,1){};\node[title](pp1b)at(.3,1){$p''_2$};
      \node[node,text width=2cm](p2)at(1.8,1){};\node[title](uu2)at(1.8,1){${p'_2}^{k-1}p_2$};
      \node[node](cb)at(3,1){$c$};
      \node[node,text width=2cm](v2)at(4.2,1){};\node[title](vv2)at(4.2,1){$q'_2$};

      \node[node,text width=1cm](p1bb)at(-.7,0){};\node[title]at(-.7,0){$p''_3$};
      \node[node,text width=2cm](p3)at(.8,0){};\node[title](uu3)at(.8,0){$p'_3$};
      \node[node,text width=1cm](u3)at(2.3,0){};\node[title](uu3)at(2.3,0){$p_3$};
      \node[node](b)at(3,0){$b$};
      \node[node,text width=1cm](q3)at(3.7,0){$q_3$};

      \draw[->,above](2.8,2.8)--node{$\sigma^{N_1}$}(1.8,2.2);
      \draw[->,above](3.2,2.8)--node{}(4.2,2.2);

      \draw[->,above](1.8,1.8)--node{$\sigma^{kN_2}$}(-.2,1.2);
      \draw[->,left](2.8,1.8)--node{}(.8,1.2);
      \draw[->,left](3.2,1.8)--node{}(5.2,1.2);

      \draw[->,left](-.2,.8)--node{$\sigma^{N_3}$}(-1.2,.2);
      \draw[->,left](.8,.8)--node{}(-.2,.2);
      \draw[->,left](2.8,.8)--node{}(1.8,.2);
      \draw[->,left](3.2,.8)--node{}(4.2,.2);
     \end{tikzpicture}
    \caption{Case 1.}\label{figureLemma42Case1}
  \end{figure}

  Since $p_1,p_2$ are non-growing and since $N_2$
  is  multiple of $p'$,
  by Lemma~\ref{propositionNonGrowingLetters}, we have
  \begin{displaymath}
    \sigma^{kN_2}(p_1)=p''_2
  \end{displaymath}
  with $p''_2=\sigma^{N_2}(p_1)$ and 
  \begin{displaymath}
    \sigma^{kN_2}(c)={p'_2}^{k-1}p_2cq'_2
  \end{displaymath}
  with $p'_2=\sigma^{N_2}(p_2)$
  (see Figure~\ref{figureLemma42Case1}).
  
  Thus  the word $x_{[-n,n]}$ is a factor of $p=p''_3p'_3p_3$
  with $p''_3=\sigma^{N_3}(p''_2)$ and $p'_3=\sigma^{N_3}({p'_2}^{k-1}p_2)$.
  This implies that $x_{[-n,n]}$ is, up to a prefix and a suffix of bounded length,
  a word of period $\sigma^{N_3}(p'_2)$. Thus $x$ is a periodic
  point of the form $x=w^\infty$ where $w$ is a fixed point of a power
  of $\sigma$ (actually of $\sigma^{N_2}$) of bounded length.

  Case 2.  We have, for $N$ large enough, $\sigma^N(a)=pbq$
  with $q$ non-growing, $b\in A$ growing and $x_{[-n,n]}$ a factor of $q$.
  This case is symmetric to Case 1. We find that $x_{[-n,n]}$ is a factor
  of $q=q_3q'_3q''_3$ with $q'_2=\sigma^{N_2}(q_1)$, $q''_2=\sigma^{N_2}(q_2)$,
    $q'_3=\sigma^{N_3}(q_2(q'_2)^{k-1})$ and $q''_3=\sigma^{N_3}(q''_2)$
  (see Figure~\ref{figureLemma42Case2}). This implies, as in Case 1,
  that $x$ is periodic point of the form $x=w^\infty$ where $w$ is a fixed point of a power
  of $\sigma$ (actually of $\sigma^{N_2}$) of bounded length.
 \begin{figure}[hbt]
    
    \centering
    \tikzset{node/.style={draw,minimum size=.4cm,inner sep=0pt}}
    \tikzset{title/.style={minimum width=.4cm,minimum height=.4cm}}
    \begin{tikzpicture}
      
      \node[node](e)at(3,3){$a$};
      \node[node,text width=1cm](p1)at(2.3,2){};\node[title](pp1)at(2.3,2){$p_1$};
      \node[node](ch)at(3,2){$c$};
      \node[node,text width=1cm](q1)at(3.7,2){};\node[title](qq1)at(3.7,2){$q_1$};

      \node[node,text width=2cm](p2)at(1.8,1){$\quad p'_2$};
      \node[node](cb)at(3,1){$c$};
      \node[node,text width=2cm](v2)at(4.2,1){$\quad q_2(q'_2)^{k-1}$};
      \node[node,text width=1cm](v2)at(5.7,1){$\quad q''_2$};
      
      \node[node,text width=2cm](p3)at(.8,0){};\node[title](uu3)at(.8,0){$p'_3$};
      \node[node,text width=1cm](u3)at(2.3,0){};\node[title](uu3)at(2.3,0){$p_3$};
      \node[node](b)at(3,0){$b$};
      \node[node,text width=1cm](q3)at(3.7,0){};\node[title](qq3)at(3.7,0){$q_3$};
      \node[node,text width=2cm](q'3)at(5.2,0){$\quad q'_3$};
      \node[node,text width=1cm](q3)at(6.7,0){$\quad q''_3$};

      \draw[->,above](2.8,2.8)--node{$\sigma^{N_1}$}(1.8,2.2);
      \draw[->,above](3.2,2.8)--node{$\sigma^{N_1}$}(4.2,2.2);

      \draw[->,left](2.8,1.8)--node{$\sigma^{kN_2}$}(.8,1.2);
      \draw[->,left](3.2,1.8)--node{}(5.2,1.2);
      \draw[->,left](4.2,1.8)--node{}(6.2,1.2);
     
      \draw[->,left](.8,.8)--node{$\sigma^{N_3}$}(-.2,.2);
      \draw[->,left](2.8,.8)--node{}(1.8,.2);
      \draw[->,left](3.2,.8)--node{}(4.2,.2);
      \draw[->,left](5.2,.8)--node{}(6.2,.2);
      \draw[->,left](6.2,.8)--node{}(7.2,.2);
     \end{tikzpicture}
    \caption{Case 2.}\label{figureLemma42Case2}
 \end{figure}
 
  Case 3. We have, for $N$ large enough, $\sigma^N(a)=pbqcr$
  with $q$ non-growing, $b,c\in A$ growing and $x_{[-n,n]}$ a factor of $q$.

  \begin{figure}[hbt]
    
    \centering
    \tikzset{node/.style={draw,minimum size=.4cm,inner sep=0pt}}
    \tikzset{title/.style={minimum width=.4cm,minimum height=.4cm}}
    \begin{tikzpicture}

      \node[node](e)at(3,3){$a$};
      \node[node,text width=1cm](p1)at(1.6,2){};\node[title](pp1)at(1.6,2){$p_1$};
      \node[node](ch)at(2.3,2){$d$};
      \node[node,text width=1cm](q1)at(3,2){};\node[title](qq1)at(3,2){$q_1$};
      \node[node](dh)at(3.7,2){$e$};
      \node[node,text width=1cm](r1)at(4.4,2){};\node[title](rr1)at(4.4,2){$r_1$};
      \node[node,text width=1cm](p2)at(0,1){};
      \node[node](cb)at(.7,1){$d$};
      \node[node,text width=1.6cm](q2)at(1.7,1){$\ q_2(q'_2)^{k-1}$};
      \node[node,text width=1cm](q'2)at(3,1){};\node[title](qq'2)at(3,1){$s_2$};
      \node[node,text width=1.6cm](q''2)at(4.3,1){$(q'''_2)^{k-1}q''_2$};
      \node[node](db)at(5.3,1){$e$};
      \node[node,text width=1cm](r2)at(6,1){};
      \node[node,text width=1cm](p3)at(-.4,0){};\node[title](pp2)at(-.4,0){$p_3$};
      \node[node](a)at(.3,0){$b$};
      \node[node,text width=1cm](q3)at(1,0){};\node[title](qq2)at(1,0){$q_3$};
      \node[node,text width=3cm](q'3)at(3,0){};\node[title](qq'0)at(3,0){$q'_3$};
      \node[node,text width=1cm](q''3)at(5,0){};\node[title](qq''3)at(5,0){$q''_3$};
      \node[node](b)at(5.7,0){$c$};
      \node[node,text width=1cm](r3)at(6.4,0){};\node[title](rr3)at(6.4,0){$r_3$};

      \draw[->,above](2.8,2.8)--node{$\sigma^{N_1}$}(1.1,2.2);
      \draw[->,above](3.2,2.8)--node{}(4.9,2.2);
      \draw[->,left](2.1,1.8)--node{$\sigma^{kN_2}$}(-.5,1.2);
      \draw[->,left](2.5,1.8)--node{}(2.5,1.2);
      \draw[->,left](3.5,1.8)--node{}(3.5,1.2);
      \draw[->,left](3.9,1.8)--node{}(6.5,1.2);
      \draw[->,left](.5,.8)--node{$\sigma^{N_3}$}(-.9,.2);
      \draw[->,left](.9,.8)--node{}(1.5,.2);
      \draw[->,left](5.1,.8)--node{}(4.5,.2);
      \draw[->,left](5.5,.8)--node{}(6.9,.2);
    \end{tikzpicture}
    \caption{Case 3.}\label{figureLemma42Case3}
  \end{figure}
  For $N$ large enough, we have $N=N_1+kN_2+N_3$ with $N_1,N_2,N_3\le p'\Card(A)^2$,
  $N_1,N_2$ multiple of $p'$, $N_2>0$, and
  
  \begin{eqnarray*}
    \sigma^{N_1}(a)&=&p_1dq_1er_1,\\
    \sigma^{N_2}(d)=p_2dq_2,\quad \sigma^{N_2}(q_1)&=&s_2,\quad \sigma^{N_2}(e)=q''_2er_2,
  \end{eqnarray*}
  where $d,e\in A$ are growing letters. The words $q_1,q_2,s_2,q''_2$ 
  are non-growing.

  We have  (see Figure~\ref{figureLemma42Case3})
  
  \begin{eqnarray*}
        \sigma^{kN_2}(d)&=&p'_2dq_2(q'_2)^{k-1},\quad \sigma^{kN_2}(q_1)=s_2,\quad \sigma^{kN_2}(e)=(q'''_2)^{k-1}q''_2er'_2,\\
    \sigma^{N_3}(d)&=&p_3bq_3,\quad \sigma^{N_3}(q_2(q'_2)^{k-1}s_2(q'''_2)^{k-1}q''_2)=q'_3,\quad \sigma^{N_3}(e)=q''_3cr_3
  \end{eqnarray*}
 with $q'_2=\sigma^{N_2}(q_2)$, $q'''_2=\sigma^{N_2}(q''_2)$. The word $q_3,q'_3,q''_3$ are non-growing.
  
  Thus $x_{[-n,n]}$ is a factor of $q=q_3q'_3q''_3$.
  This implies that $x$ is a shift of $^\omega u\cdot vw^\omega$
  with $u=\sigma^{N_3}(q'_2)$, $v=\sigma^{N_3}(q''_2)$ and $w=\sigma^{N_3}(q'''_2)$.
  Each of these words is a fixed point of $\sigma^{N_2}$.

\end{proof}

\begin{example}
  Let $\sigma\colon a\mapsto abb,b\mapsto b$ (see Example~\ref{example(abb,b)}).
  We have $\XS(\sigma)=\{b^\infty\}$ and $b^\infty$ is a periodic point
  having only non-growing letters.
\end{example}

The following corollary of Proposition~\ref{lemmaNonGrowing} appears
in~\cite[Proposition 5.5]{BezuglyiKwiatkowskiMedynets2009} (and also in \cite[Corollary 2.11]{MaloneyRust2018})
in the case of non-erasing morphisms. 
\begin{corollary}
  If a morphism $\sigma$ is aperiodic, every point in $\XS(\sigma)$
  contains an infinite number of growing letters on the left and on
  the right.
\end{corollary}
We will also use the following result describing the points with
a finite number of growing letters on the left or on the right.
\begin{lemma}\label{lemmaNonGrowingBis}
Let $\sigma\colon A^*\to A^*$ be a morphism. 
If $x$ is a point of $\XS(\sigma)$
which has a growing letter at some position $i$ such that all letters
at smaller positions are non-growing, then there is an integer $k$
such that $x_{(-\infty,i)}= {}^\omega u v w^{k}z$ where
$u,v,w,z$ are finite words of lengths bounded by a value depending
only on $\sigma$. 
\end{lemma}
\begin{proof}
  
  For every $n\ge 1$, there is some $N\ge 1$
  and $a\in A$ such that $x_{[-n,i)}$ is a factor of $\sigma^N(a)$.
  We distinguish three cases as in the proof of Proposition~\ref{lemmaNonGrowing}.

  Case 1. We have, for $N$ large enough , $\sigma^N(a)=pbq$
  with $p,q$ non-growing $b\in A$ growing and $x_{[-n,i)}$ a factor
    of $p$. We define $N_k$, $p_k,q_k,p'_k,q'_k$ as in Case 1
    of the proof of Proposition~\ref{lemmaNonGrowing}.
 Thus  $x_{[-n,i)}$ is a factor of $\sigma^{N_3}(p''_2)\sigma^{N_3}((p'_2)^{k-1}p_2)$.
Since there is a finite number of possible $p_2, p'_2, p''_2, N_3$ for
all $n$, there is an infinite number of $n$ such that $x_{[-n,i)}$ is
a factor of $\sigma^{N_3}(p''_2)\sigma^{N_3}((p'_2)^{k-1}p_2)$
for some fixed $p_2, p'_2, p''_2, N_3$.
This implies that $\sigma^{N_3}(p_2)$ is non-empty and
$x_{(-\infty,i)} = {}^\omega\sigma^{N_3}(p'_2)s_2$ where $s_2$ is some prefix
of $\sigma^{N_3}(p_2)p_3$. 

Case 2. We have, for $N$ large enough , $\sigma^N(a)=pbq$
  with $p,q$ non-growing $b\in A$ growing and $x_{[-n,i)}$ a factor
  of $q$. We define $N_k$, $p_k,q_k,p'_k,q'_k$ as in Case 2.
  Thus  $x_{[-n,i)}$ is a factor of $q_3\sigma^{N_3}(q_2)\sigma^{N_3}(q'_2)^{k-1}q''_3$.
This implies that $\sigma^{N_3}(q'_2)$ is non-empty and
$x_{(-\infty,i)} = {}^\omega\sigma^{N_3}(q'_2)r_2$ where $r_2$ is some factor
of $\sigma^{N_3}(q'_2)q''_3$.

Case 3. We have, for $N$ large enough, $\sigma^N(a)=pbqcr$
  with $q$ non-growing, $b,c\in A$ growing and $x_{[-n,i)}$ a factor of $q$.
  We define $N_1, N_2, N_3$, $p_k,q_k, p'_k, q'_k$, $q''_k, r_k, r'_k$ as
  in Case 3. We get that $x_{[-n,i)}$ is a factor of
    \begin{displaymath}
      q=\sigma^{N_3}(q_2(q'_2)^{k-1} q'_1 (q'''_2)^{k-1}q''_2)
    \end{displaymath}
    for all $n$ 
 with a finite number of possible $q'_1, q_2, q'_2,q''_2,q'''_2, N_3$. 
 Then $x_{[-n,i)}$ is a
 factor of $q$
 for some fixed $N_3$, $q'_1, q_2, q'_2,q''_2,q'''_2$ for an infinite number of $n$.
 This implies that either $x_{[-n,i)}= {}^\omega
 \sigma^{N_3}(q'''_2)s_2$ with $s_2$ a prefix of $q''_2$, or
$x_{(-\infty,i)}= {}^\omega
\sigma^{N_3}(q'_2)r_2$ with $r_2$ a prefix of
$\sigma^{N_3} (q'_1(q'''_2)^{k-1}q''_2)$ for some non-negative
integer~$k$.
\end{proof}
A symmetric version of Lemma~\ref{lemmaNonGrowingBis} holds
for points with no growing letter after some position $i$.
The following example illustrates Lemma~\ref{lemmaNonGrowingBis}
\begin{example}
  Let $\sigma\colon a\mapsto bac, b\mapsto d,d\mapsto b,c\mapsto c$.
  Then $x={}^\omega(db)\cdot ac^\omega$ is in $\cL(\XS(\sigma))$ and
  has no growing letter at negative indices. Thus
  $x^-$ has the form indicated in Lemma~\ref{lemmaNonGrowingBis}
  with $u=db$ and $v=w=\varepsilon$. We have no example with
  $|vw^kz|$ unbounded.
  \end{example}
\section{Fixed points}\label{sectionFixedPoints}

We investigate fixed points of morphisms.  Beginning with
finite fixed points, we will next turn to one-sided infinite
and then to two-sided infinite ones.

\paragraph{Finite fixed points}
Let $\sigma\colon A^*\to A^*$ be a morphism. As
we have already seen before, a finite fixed point
of $\sigma$ is a word $w\in A^*$ such that $\sigma(w)=w$.

Set
\begin{displaymath}
  A(\sigma)=\{a\in A\mid \sigma(a)=uav\mbox{ with $u,v$ erasable}\}.
  \end{displaymath}
The following is from \cite{Shallit1999}
(see also \cite[Theorem 7.2.3]{AlloucheShallit2003}).
\begin{proposition}\label{propositionASTheorem723}
  The set of finite fixed points of $\sigma\colon A^*\to A^*$ is
  a submonoid of $A^*$ generated by the finite set
  \begin{displaymath}
    F(\sigma)=\{\sigma^{\Card(A)}(a)\mid a\in A(\sigma)\}.
    \end{displaymath}
\end{proposition}

\paragraph{Infinite fixed points}
Let $\sigma\colon A^*\to A^*$ be a morphism.
A \emph{right-infinite fixed point} of $\sigma$ is a right-infinite sequence $x\in A^\N$ such that $\sigma(x)=x$.
Symmetrically, a \emph{left-infinite fixed point}
of $\sigma$ is a  left-infinite sequence $x\in A^{-\N}$
such that $\sigma(x)=x$. A \emph{two-sided infinite fixed point}
is a two-sided infinite sequence $x\in A^\Z$
such that $\sigma(x)=x$.

A morphism $\sigma\colon A^*\to A^*$ is \emph{right-prolongable}
on $u\in A^+$ if $\sigma(u)$ begins with $u$ and $u$ is growing.
In this case, there is a unique right-infinite sequence $x\in A^\N$
which has each $\sigma^n(u)$ as a prefix. We denote $x=\sigma^\omega(u)$.

Symmetrically, $\sigma$ is \emph{left-prolongable}
on $v\in A^+$ if $\sigma(v)$ ends with $v$ and $v$ is growing.
In this case, there is a unique left-infinite sequence $y\in A^{-\N}$
which has all $\sigma^n(v)$ as a suffix. We denote $y=\sigma^{\tilde{\omega}}(v)$.

We begin with an elementary result concerning periodic fixed points.
\begin{lemma}\label{lemmaPeriodicFixedPoints}
For every  $u\in A^+$, the following conditions are
equivalent.
\begin{enumerate}
  \item[\rm(i)] The right-infinite
 sequence $u^\omega$ is a right-infinite fixed point
of $\sigma$.
\item[\rm(ii)] The left-infinite
 sequence ${}^\omega u$ is a left-infinite fixed point
of $\sigma$.
\item[\rm(iii)] The  two-sided infinite
 sequence $u^\infty$ is a two-sided infinite fixed point
of $\sigma$.
\item[\rm(iv)] One has $\sigma(u)=u^n$ for some $n\ge 1$.
\end{enumerate}
\end{lemma}
\begin{proof}
Assume that (i) holds. We may assume that $u$ is primitive.
Set $u'=\sigma(u)$. Then $u^\omega=u'^\omega$, which implies
$u'=u^n$ since $u$ is primitive. Thus (i) implies (iv).
Symmetrically, (ii) implies (iv).
The other implications are clear.
\end{proof}
We shall come back
to periodic fixed points in the next section.

We now describe more precisely the words $u$ such that
a morphism is right-prolongable on $u$.

\begin{lemma}\label{lemmaRightProl}
  If some power $\sigma^k$
  of a morphism $\sigma\colon A^*\to A^*$ is right-prolongable on $u\in A^+$,
  there is a growing letter $a$
  such that $u=praq$ with $\sigma^k(p)=p$, the word
  $r$ erasable, $\sigma^k$ right-prolongable on $ra$
  and $\sigma^{k\omega}(u)=\sigma^{k\omega}(pra)$.
  Moreover, if $u\in\cL(\sigma)$,
  such an integer $k$ can be bounded in terms of $\sigma$.
  \end{lemma}
\begin{proof}
  Let $a$ be the first growing letter of $u$ and set $u=saq$. Then $a$
  is also the first growing letter of $\sigma^k(a)$. Set
  $\sigma^k(a)=vaw$. Since we have
  $s=\sigma^{nk}(s)\sigma^{(n-1)k}(v)\cdots\sigma(v)v$
  for all $n\ge 1$, the word $v$ is erasable
  and $r=\sigma^{k(\mex(\sigma)-1)}(v)\cdots\sigma^k(v)v$ is such that $\sigma^k$
  is right-prolongable on $ra$. Set
  $t=w\sigma^k(w)\cdots \sigma^{k(\mex(\sigma)-1)}(w)$ and $p=\sigma^{k\mex(\sigma)}(s)$.
  Then
  \begin{equation*}
   \sigma^{k\mex(\sigma)}(u)=prat\sigma^{k\mex(\sigma)}(q).
  \end{equation*}
  Thus $s=pr$ and $\sigma^{k\omega}(u)=\sigma^{k\omega}(pra)$.

  There remains to show that $k$ can be bounded. We prove separately
  that one can find  a bounded integer $k$ such that $\sigma^k(p)=p$
  and, provided $u\in\cL(\sigma)$,
  such that $\sigma^k$ is right-prolongable on $ra$. If $\sigma^k(p)=p$,
  then, by Proposition~\ref{propositionASTheorem723}, we have
  $p\in F(\sigma^k)^*$. Since $A(\sigma^k)\subset A(\sigma^{\mex(\sigma)!})$,
  we can choose $k\le\mex(\sigma)!$. Next, if $\sigma^k$ is
  right-prolongable on $ra$ with $r\in\cL(\sigma)$ erasable,
  the length of $r$ can be bounded in terms of $\sigma$
  by \cite[Lemma 2.2]{BealPerrinRestivo2021}. Since $a$
  is growing, there is a bounded integer $k$ such that $|\sigma^k(ra)|\ge |ra|$.
    Then $\sigma^k$ is right-prolongable on $ra$.
    \end{proof}

 The following 
 result is due to \cite{HeadLando1986} (see also \cite[Theorem 7.3.1]{AlloucheShallit2003}). Other proofs or
 variations are also
 given in \cite{ShallitWang2002} or \cite[Theorem 1]{DiekertKrieger2009}.

\begin{proposition}\label{theoremHeadLando}
  Every right-infinite fixed point of a morphism $\sigma\colon A^*\to A^*$
  is either in $F(\sigma)^\omega$ or of the form $\sigma^\omega(u)$
  where $\sigma$ is right-prolongable on $u\in A^+$.
  \end{proposition}
A right-infinite fixed point
of a morphism $\sigma$ is  \emph{admissible}  if it is
in $\XS(\sigma)^+$. 
\begin{proposition}\label{lemmaOneSidedFixedPoint}
  Let $\sigma\colon A^*\to A^*$ be a morphism. The number
  of orbits of right-infinite admissible fixed points of a power of $\sigma$
  is finite and effectively bounded in terms of $\sigma$. It is nonzero
  provided $\XS(\sigma)$
  is non-empty.
\end{proposition}

The proof of Proposition~\ref{lemmaOneSidedFixedPoint}
 uses properties of non-negative matrices.
To such a matrix $M$, we associate the multigraph $G$
such that $M$ is the adjacency matrix of $G$.
The matrix is \emph{irreducible} if $G$ is strongly connected.
The \emph{period} of a graph is the greatest common divisor of the lengths
of its cycles.
The matrix $M$ is \emph {primitive} if it is irreducible
and the period of $G$ is $1$. 
By a well-known bound due to Wielandt \cite{Wielandt1950},
if $M$ is a primitive
matrix of dimension $k$, then $M^n$ is positive for
$n=(k-1)^2+1$ and thus for all $n\ge (k-1)^2+1$
(see also \cite[Exercise 8.7.8]{AlloucheShallit2003}).

It follows that if $M$ is irreducible and $G$ has period $p$,
then $M^p$ is a block diagonal matrix formed of $p$
primitive diagonal blocks. Thus $M^{mp}$ 
is block diagonal with $p$ positive diagonal blocks
for $m=(k/p-1)^2+1$.

\begin{lemma}\label{lemmaABC}
Let $\sigma\colon A^*\to A^*$ be a morphism. If $\XS(\sigma)$
is non-empty, there are disjoint subsets $B,C,D,E$ of $A$
with $B$ non-empty
  and an integer $n\le \Card(A)!$
  such that
  \begin{eqnarray*}
    \sigma^n(B)&\subset& (B\cup C\cup D\cup E)^*,\\
    \sigma^n(C)&\subset& (D\cup E)^*,\\
    \sigma^n(d)&\subset& E^*dE^*\mbox{ for every $d\in D$},\\
    \sigma^n(E)&=& \{\varepsilon\},
    \end{eqnarray*} 
  and, moreover, such that every letter in $B$ is growing and appears in every
  word $\sigma^n(b)$ for  every $n\ge \Card(A)^2$ and 
   $b\in B$.
\end{lemma}
\begin{proof}
  We use the multigraph $G(\sigma)$.
  Consider the set $\cal S$
    of non-trivial strongly connected components of $G(\sigma)$.
    Let $\G$ be the graph on $\cal S$ with edges $(S,T)$
    if there is a path from a vertex of $S$ to a vertex of $T$.
    The \emph{leaves} of $\G$
     are the elements of $\cal S$ from which no other non-trivial
    component can be reached.
    We consider two cases.

    \noindent Case 1. All leaves of $\G$ are reduced to a cycle.
    Since $\XS(\sigma)$ is not empty, not all elements of $\cal S$
    can be leaves of $\cal G$. Considering a path of maximal length in $\G$
    ending in a leaf of $\G$,
    it follows that there is
    an element $H=(V_H,E_H)$ of $\cal S$ which is not a leaf in $\G$ and
    which can be followed only
    by leaves in $\G$.

    Let $p$ be the period of $H$ and let $B\subset V_H$ be the set
    of vertices of a strongly
  connected component of $G(\sigma^p)$.
  We choose for $D$ the set of vertices of all leaves  of $G$
  accessible from $B$.
We choose for $C$ the set of vertices of $G$ forming a
trivial component on a path from
$B$ to $D$ in $G$. Finally, we choose for $E$ the union of vertices
of trivial strongly connected
components not belonging to $C$ (and thus formed of erasable letters)
  accessible from $B$ in $G$.

  Set $n=\Card(A)!$.  Then $n$ is a multiple of the periods
  of all components $S$ and thus
  $M(\sigma^n)$ is the identity on the components $S$ in $D$. We have
  \begin{displaymath}
    ((\Card(B)-1)^2+1)p\le \Card(B)!p\le \Card(A)!.
  \end{displaymath}
  Thus, by Wielandt's bound,
  the matrix $M(\sigma^n)$
  is  positive on $B$. We also have
   $\sigma^n(E)=\{\varepsilon\}$
  and also  $\sigma^n(C)\subset(D\cup E)^*$.

  \noindent Case 2. There exist a  leaf $K=(V_K,E_K)$ of $\cal G$
   which is not reduced to a cycle. 
   Let $p$ be the period of $K$ and let $B$
   be the set of vertices in $V_K$ of  a strongly connected
  component of $G(\sigma^p)$. Since $((\Card(B)-1)^2+1)p\le \Card(A)!$
  as above, $M^n$ is positive on $B$.
   We choose $C=D=\emptyset$ and $E$
   is the set of erasable letters accessible from $B$.
   Since $n\ge\Card(E)$, we have also $\sigma^n(E)=\{\varepsilon\}$.
\end{proof}

Note that, as a reformulation
of the conditions of Lemma~\ref{lemmaABC}, the
restriction $M'$ to $B\cup C\cup D\cup E$
of the adjacency matrix of $G(\sigma^n)$ has the form
\begin{displaymath}
  M'=\kbordermatrix{&B&C&D&E\cr
    B&N&*&*&*\cr
    C&0&0&*&*\cr
    D&0&0&\begin{vmatrix}1&\cdots&0\\0&\ddots&0\\0&\cdots&1\end{vmatrix}&*\cr   
  E&0&0&0&0}
  \end{displaymath}
where $N>0$ is a primitive matrix with  $M_{b,d}\ne 0$
for at least one $b\in B$ and $d\in D$
(Case 1) or $N$ not the identity matrix of dimension $1$
and $C=D=\emptyset$ (Case 2).
We illustrate the first case of the proof in the following example.
  \begin{figure}[hbt]
    \centering
\tikzset{node/.style={circle,draw,minimum size=0.4cm,inner sep=0pt}}
\tikzset{title/.style={minimum size=0.5cm,inner sep=0pt}}
\tikzstyle{every loop}=[->,shorten >=1pt,looseness=8]
\tikzstyle{loop left}=[in=220,out=140,loop]
\tikzstyle{loop right}=[in=-40,out=40,loop]
    \begin{tikzpicture}
      \node[node](a)at(0,0){$a$};
      \node[node](b)at(2,0){$b$};
      \node[node](c)at(3,.5){$c$};
      \node[node](d)at(3,-.5){$d$};
     
      \draw[->,>= stealth](a)edge[loop left]node{}(a);
      \draw[->,>= stealth,bend left](a)edge node{}(b);
      \draw[->,>= stealth,bend right](a)edge node{}(b);
      \draw[->,>= stealth](b)edge node{}(c);
      \draw[->,>= stealth](b)edge node{}(d);
      \draw[->,>= stealth](c)edge[loop right]node{}(c);
      \draw[->,>= stealth](d)edge[loop right]node{}(d);
      \end{tikzpicture}
    \caption{The graph $G(\sigma)$.}\label{figureG(sigma)}
    
  \end{figure}
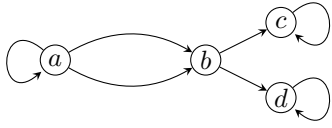
\begin{example}
  Let $\sigma\colon a\mapsto bab,b\mapsto cd,c\mapsto c,d\mapsto d$.
  The graph $G(\sigma)$ is represented in Figure~\ref{figureG(sigma)}.
  In this case, we choose $n=1$, $B=\{a\}$, $C=\{b\}$, $D=\{c,d\}$
  and $E=\emptyset$.
\end{example}
The following example illustrates the second case
in the proof of Lemma~\ref{lemmaABC}.
\begin{example}\label{example(baab,1)}
  Let $\sigma\colon a\mapsto baab,b\mapsto\varepsilon$. The
  graph $G(\sigma)$ is represented in Figure~\ref{figureG(sigma)2}.
  We choose this time $n=1$, $B=\{a\}$, $C=D=\emptyset$ and $E=\{b\}$.
    \begin{figure}[hbt]
    \centering
\tikzset{node/.style={circle,draw,minimum size=0.4cm,inner sep=0pt}}
\tikzset{title/.style={minimum size=0.5cm,inner sep=0pt}}
\tikzstyle{every loop}=[->,shorten >=1pt,looseness=8]
\tikzstyle{loop above}=[in=130,out=50,loop]
\tikzstyle{loop below}=[in=-130,out=-50,loop]
    \begin{tikzpicture}
      \node[node](a)at(0,0){$a$};
      \node[node](b)at(2,0){$b$};
 
A     
      \draw[->,>= stealth](a)edge[loop above]node{}(a);
      \draw[->,>= stealth](a)edge[loop below]node{}(a);
      \draw[->,>= stealth,bend right](a)edge node{}(b);
      \draw[->,>= stealth,bend left](a)edge node{}(b);
      \end{tikzpicture}
    \caption{The graph $G(\sigma)$.}\label{figureG(sigma)2}
    
  \end{figure}
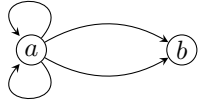
  \end{example}
  In the next lemma, we use the notation
  $\sigma^{\dpower{p}}(u)=\sigma^{p}(u)\cdots\sigma(u)u$.
\begin{lemma}\label{lemma2cases}
  Let $\sigma\colon A^*\to A^*$ be a morphism and let $a\in A$
  be a growing letter such that $\sigma(a)=uav$
  with $u$ non-growing. Let $i,p$ be such that
  $\sigma^{i+p}(u)=\sigma^i(u)$ with $p>0$. If $\sigma^i(u)=\varepsilon$,
   then $\sigma^i$ is right-prolongable
  on $\sigma^{\dpower{i}}(u)a$. Otherwise $w^\omega$,
  with $w=\sigma^{i+p-1}(u)\dots\sigma^{i+1}(u)\sigma^{i}(u)$  is
  an admissible right-infinite fixed point of $\sigma^p$.
\end{lemma}
\begin{proof}
If $\sigma^i(u)$ is empty, using the notation
   $w\ge w'$ if $w'$ is a prefix of $w$,
  we have $\sigma^i(a)\ge \sigma^{\dpower{i}}(u)a$.
  Moreover $\sigma^n(a)\ne\sigma^{\dpower{i}}(u)a$ because $a$ is growing.
  Set $t=\sigma^{\dpower{i}}(u)a$.
  Since
  $\sigma^i(t)\ge\sigma^i(a)\ge t$, this
  implies that $\sigma^i$
  is right-prolongable on $t$ and
  thus the conclusion in this case.

   Otherwise, we have for every $k\ge 1$, 
  \begin{eqnarray*}
    \sigma^{i+kp}(a)&\ge& \sigma^{\odpower{i+kp-1}}(u)a\\
    &\ge& (\sigma^{i+p-1}(u)\cdots\sigma^i(u))^{k-1}\sigma^{\odpower{i+p-1}}(u)a\\
    &\ge&w^{k-1}\sigma^{\odpower{i+p-1}}(u)a
      \end{eqnarray*}
 with $w=\sigma^{i+p-1}(u)\cdots\sigma^i(u)$.
 This shows that $w^\omega$ is in $\XS(\sigma)^+$ and is a right-infinite
 fixed point
  of $\sigma^p$.
\end{proof}
There is a symmetric version of Lemma~\ref{lemma2cases}. Assume that
that  $\sigma(a)=uav$ with $v$ non-growing. Let $i,p$ be such
that $\sigma^{i+p}(v)=\sigma^i(v)$ with $p>0$.
If $\sigma^i(v)=\varepsilon$, then $\sigma^i$ is left-prolongable
on $av\sigma(v)\cdots\sigma^{i-1}(v)$. Otherwise, $^\omega w$
with $w=\sigma^i(v)\sigma^{i+1}(v)\cdots\sigma^{i+p-1}(v)$
is a left-infinite fixed point of $\sigma$.

\quad

\begin{proofof}{of Proposition~\ref{lemmaOneSidedFixedPoint}}
  We first show that the number of orbits is nonzero.
 Let $B,C,D,E$ and $n$ be as in Lemma~\ref{lemmaABC}. 
 Taking the restriction of $\sigma^n$ to $B$ and changing $\sigma$
  for $\sigma^n$, we may assume that $A=B\cup C\cup D\cup E$,
  with $\sigma(C)\subset (D\cup E)^*$,
  $\sigma(D)\subset E^*DE^*$ and $\sigma(E)=\{\varepsilon\}$
  and every letter $a\in B$ appears in every $\sigma(b)$
  for $b\in B$. 

   Changing again $\sigma$ for some power, we may assume
  that there is a letter $a\in B$ such that $\sigma(a)=uav$
  with $u\in (C\cup D\cup E)^*$. If $u$ is non-erasable,
  the elements of $C,D,E$ are non-growing and
  there is a periodic admissible fixed point by Lemma~\ref{lemma2cases}.
  If $u$ is erasable, some power $\sigma^n$ of $\sigma$ is right-prolongable
  on $w=\sigma^{\dpower{n}}(u)a$ and thus $\sigma^{n\omega}(w)$
  is a right-infinite fixed point of $\sigma^n$
  (recall that $\sigma^{n\omega}=(\sigma^n)^\omega$). If $v$ is growing,
  it contains a letter of $B$ and thus $a$ is in $\cL(\XS(\sigma))$,
  which implies that the right-infinite fixed point $\sigma^{n\omega}(w)$
  is admissible. Otherwise, since $v$ cannot be erasable,
  we obtain by a symmetric version of Lemma~\ref{lemma2cases}
  a periodic admissible left-infinite fixed point,
  and thus also a periodic right-infinite fixed point.

  There remains to show that the number of orbits of
  admissible right-infinite fixed points
  of a power of $\sigma$ is finite.
  Let $x$ be such a fixed point. By Proposition \ref{theoremHeadLando}, either
  $x$ is in $F(\sigma^n)^\omega$ or of the form $\sigma^{n\omega}(u)$ where $\sigma^n$
  is right-prolongable on $u\in A^+$. If $x$ is in $F(\sigma^n)^\omega$, 
  it is formed of non-growing letters and there is
  a finite number of orbits of such points by Proposition~\ref{lemmaNonGrowing}.
  Assume that  $x=\sigma^{n\omega}(u)$ with $\sigma^n$ right-prolongable
  on $u$. By Lemma~\ref{lemmaRightProl}, there is a
  growing letter $a$ such that
  $u=praq$ with $\sigma(p)=p$, the word $r$ erasable, $\sigma^n$
  right-prolongable on $ra$ and $x=\sigma^{n\omega}(pra)$.
  Then $x$ belongs to the orbit of the admissible right-infinite fixed
  point $y=\sigma^{n\omega}(ra)$. By Proposition~\ref{propositionErasing},
  there is a finite number of erasable words $r$ and thus
  a finite number of points $y$.
  
 \end{proofof}

\begin{example}\label{example(baab,1)2}
  Let $\sigma\colon a\mapsto baab,b\mapsto\varepsilon$ as in Example~\ref{example(baab,1)}. Then
  $\sigma(baab)=(baab)^2$. Thus
  $\sigma$ is right-prolongable on $ba$ and
  $\sigma^\omega(baab)=(baab)^\omega$ is an admissible one-sided fixed point of $\sigma$.
\end{example}

\begin{example}\label{example(abb,b)}
  Let $\sigma\colon a\mapsto abb,b\mapsto b$. Then
  $\XS(\sigma)=\{b^\infty\}$ and $b^\omega$ is an admissible right-infinite fixed
  point of $\sigma$. The sequence $ab^\omega$ is also a right-infinite
  fixed point but it is not admissible.
\end{example}
The following example shows that the number of admissible right-infinite
fixed points of $\sigma$ need not be finite.
\begin{example}\label{exampleInfiniteNbFP}
  Let $\sigma\colon a\mapsto bc, b\mapsto bd, c\mapsto ec,
  d\mapsto d, e\mapsto e$. Then $d^ne^\omega$ is an admissible
  right-infinite fixed point of $\sigma$ for every $n\ge 1$.
\end{example}
We now add a complement to Proposition~\ref{theoremHeadLando}
which characterizes right  infinite admissible fixed points.
\begin{proposition}\label{propositionAdmissibleOneSided}
  Let $\sigma\colon A^*\to A^*$ be a morphism. A right-infinite
  sequence $x\in A^\N$
  is an admissible right-infinite fixed point of $\sigma$ if and only if one the two
  following conditions is satisfied.
  \begin{enumerate}
  \item[\rm(i)] $x=uv^\omega$ with $\sigma(u)=u$, $\sigma(v)=v$
    and $uv^*\subset \cL(\XS(\sigma))$.
  \item[\rm(ii)] $x=\sigma^\omega(u)$ with $u\in\cL(\XS(\sigma))$  and
    $\sigma$  right-prolongable on $u$.
    \end{enumerate}
\end{proposition}
\begin{proof}
  The conditions are clearly sufficient.
  
  Conversely, if $x$ contains a growing letter,
  let $u$ be the shortest prefix of $x$
  containing a growing letter. Then $u$ is growing and $\sigma$
  is right-prolongable on $u$. Thus (ii) is satisfied.
  Otherwise, by Lemma~\ref{lemmaNonGrowing}, $x$ is eventually periodic.
  On the other hand, by Proposition~\ref{theoremHeadLando},
  we have $x\in F(\sigma)^\omega$. Thus there
  are $u,v\in F(\sigma)^*$ such that $x=uv^\omega$. Thus condition (i)
  is satisfied.
  \end{proof}
\paragraph{Two-sided fixed points}
Let us now consider two-sided infinite fixed points.
Let $\sigma\colon A^*\to A^*$ be a morphism.
Recall that a
two-sided infinite fixed point of $\sigma$ is a sequence $x\in A^\Z$
such that $\sigma(x)=x$.

Assume that $\sigma$ is right-prolongable on $u \in A^+$ and
let $x=\sigma^\omega(u)$. Assume next that $\sigma$ is
left-prolongable on $v\in A^+$.
Denote by $y=\sigma^{\tilde{\omega}}(v)$  the left-infinite sequence
having all $\sigma^n(v)$ as suffixes. Let $z\in A^\Z$ be the two-sided infinite
sequence such that $x=z^+$ and $y=z^-$. Then $z$
is a two-sided infinite
fixed point of $\sigma$ denoted $\sigma^\omega(v\cdot u)$.
It can happen that $\sigma^{ \omega}(v \cdot u)$
does not belong to $X (\sigma )$. 
In fact, $\sigma^{ \omega}(v \cdot u)$ belongs to $X (\sigma )$
if and only if $vu$ belongs to $\mathcal{L}(\sigma)$.

The description of two-sided infinite fixed points follows directly
from Proposition~\ref{theoremHeadLando}. Indeed, $x\in A^\Z$
is a fixed point of $\sigma\colon A^*\to A^*$ if
and only if $x^-$ and $x^+$ are one-sided infinite fixed points of $\sigma$.
Thus, there are four cases according to whether none
of $x^-,x^+$ is in $^\omega F(\sigma)\cup F(\sigma)^\omega$ 
or one of them or both.

An {\em admissible}
two-sided infinite fixed point of $\sigma$ is a sequence $x\in \XS(\sigma)$
which is a two-sided fixed point. Thus a fixed point of the
form $\sigma^\omega(v\cdot u)$ with $vu\in\cL(\sigma)$
is admissible.

In the following result, the existence of admissible two-sided
infinite fixed points is a modification of
\cite[Proposition 1.4.13]{DurandPerrin2021} where
it is stated for a non-erasing morphism $\sigma$
such that $\XS(\sigma)$ is minimal.

\begin{theorem}\label{propositionFixedPointMinimal2}
  A morphism $\sigma\colon A^*\to A^*$ such that $\XS(\sigma)$ is non-empty
  has a power $\sigma^k$ with an admissible two-sided infinite fixed point.
  The set of orbits of these fixed points
  is a finite computable set and the exponent $k$ can be bounded
  in terms of $\sigma$.
\end{theorem}
The proof uses the following lemma.
\begin{lemma}\label{lemmaRightExtendable}
  Let $\sigma\colon A^*\to A^*$ be a morphism. There
  is a finite number of words $auvb\in\cL(\XS(\sigma))$
  with $u,v$ non-growing, $a,b$ growing, such that
  for some $k\ge 1$,
  $\sigma^k$ is left-prolongable on $au$ and
  right-prolongable on $vb$. Their length is bounded
  effectively in terms of $\sigma$.
\end{lemma}
\begin{proof}
  Consider $a,u,v,b$ such that for some integer $k$
  the conditions are satisfied. By Lemma~\ref{lemmaRightProl},
  the least $k$ such that this is true
  is bounded in terms of $\sigma$. Thus we may replace $\sigma^k$
  by $\sigma$.
  
  Since $auvb\in\cL(\XS(\sigma))$, there is, for every large enough 
  integer $n$, a letter $e\in A$ such that $auvb$ is a factor of $\sigma^n(e)$.
  For each $e\in A$ and $n\ge 0$, consider
  the derivation tree $T_\sigma(e,n)$ of $e$ at order $n$.
  In this tree, there is  a unique path from the root
  to each of the two vertices labeled $a$ and $b$ at level
  $n$. Let us say that
  the tree $T_\sigma(e,n)$ is \emph{special} if the nodes
   at level $1$ on these paths are distinct.

  We distinguish two cases.

  Case 1. The integers $n$ such that the tree $T_\sigma(e,n)$ is special
    are bounded by $\Card(A)^2$.
   Then, the length of $auvb$ is bounded by
   $|\sigma|^{\Card(A)^2}$ and thus the conclusion holds.

   Case 2. There is a special tree $T_\sigma(e,n)$ with $n>\Card(A)^2$. 
  Since $a$
  and $b$ are growing, all vertices on these paths are labeled
  by a growing letter. On the path from the root to $a$,
  the successor of a vertex labeled by $c$ distinct of the root is the last
  growing letter $f$ of $\sigma(c)$ and is thus uniquely determined.
  Similarly, on the path from the root to $b$,
  the successor of a vertex labeled by $d$ distinct of the root is the first
  growing letter $g$ of $\sigma(d)$ and is thus uniquely determined
  (see Figure~\ref{figureFiniteaub} on the left).

  Since $n>\Card(A)^2$,
  we can write $n=n_1+kn_2+n_3$ with $n_1,n_2,n_3\le \Card(A)^2$, $n_2>0$ and $k\ge 1$,
  such that there is
  a path of length $n_1$ from the root to a node labeled by $c$ (resp. $d$),
  $k$ paths of length $n_2$ from $c$ to $c$ (resp. $d$ to $d$)
  and a path of length $n_3$
  from $c$ to $a$ (resp. from $d$ to $b$)
  (see Figure~\ref{figureFiniteaub}).
\begin{figure}[hbt]
    
    \centering
    \tikzset{node/.style={draw,minimum size=.4cm,inner sep=0pt}}
    \tikzset{title/.style={minimum width=.4cm,minimum height=.4cm}}
    \begin{tikzpicture}
      \node[node](e)at(-7,3){$e$};
      \node[node](ch)at(-8,2){};
      \node[node](dh)at(-6,2){};
      \node[node](c)at(-8,1){$c$};
      \node[node](d)at(-6,1){$d$};
      \node[node](f)at(-8,0){$f$};
      \node[node](g)at(-6,0){$g$};
      \node[node](a)at(-8,-1){$a$};
      \node[node](b)at(-6,-1){$b$};

      \draw[->](-7,2.8)--node{}(-8,2.2);
      \draw[->](-7,2.8)--node{}(-6,2.2);
      \draw[dashed,->](-8,1.8)--node{}(-8,1.2);
      \draw[dashed,->](-6,1.8)--node{}(-6,1.2);
      \draw[->](-8,.8)--node{}(-8,.2);
      \draw[->](-6,.8)--node{}(-6,.2);
      \draw[dashed,->](-8,-.2)--node{}(-8,-.8);
      \draw[dashed,->](-6,-.2)--node{}(-6,-.8);
      \node[node](e)at(-3,3){$e$};
      \node[node](ch)at(-4,2){$c$};
      \node[node](dh)at(-2,2){$d$};
      \node[node](cb)at(-4,1){$c$};
      \node[node](db)at(-2,1){$d$};
      \node[node](a)at(-4,0){$a$};
      \node[node](b)at(-2,0){$b$};
      
      \draw[->,above](-3,2.8)--node{$\sigma^{n_1}$}(-4,2.2);
      \draw[->](-3,2.8)--node{}(-2,2.2);
      \draw[->,left](-4,1.8)--node{$\sigma^{kn_2}$}(-4,1.2);
      \draw[->](-2,1.8)--node{}(-2,1.2);
      \draw[->,left](-4,.8)--node{$\sigma^{n_3}$}(-4,.2);
      \draw[->](-2,.8)--node{}(-2,.2);
    \end{tikzpicture}
    \caption{The pair $(c,d)$.}\label{figureFiniteaub}
    \end{figure}
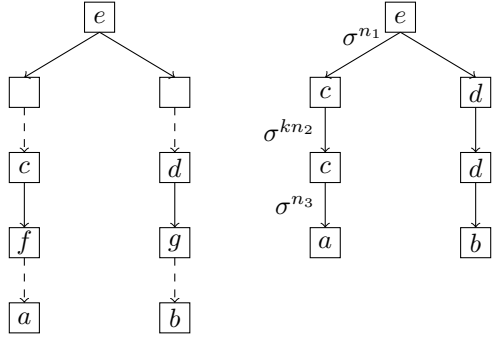

  Since $T_\sigma(e,n)$ is special, the nodes labeled by $c$ and $d$
  at level $n_1$ are distinct.  Since $\sigma$
  is left-prolongable on $au$, the letter $a$ is the last
  growing letter of $\sigma(a)$. This forces $c=a$. Similarly
  $d=b$. Consequently, we can choose $n_2=1$ and $n_3=0$.

\begin{figure}[hbt]
    
    \centering
    \tikzset{node/.style={draw,minimum size=.4cm,inner sep=0pt}}
    \tikzset{title/.style={minimum width=.4cm,minimum height=.4cm}}
    \begin{tikzpicture}

      \node[node](e)at(3,3){$e$};
      \node[node,text width=1cm](p1)at(1.6,2){};\node[title](pp1)at(1.6,2){$p_1$};
      \node[node](ch)at(2.3,2){$a$};
      \node[node,text width=1cm](q1)at(3,2){};\node[title](qq1)at(3,2){$q_1$};
      \node[node](dh)at(3.7,2){$b$};
      \node[node,text width=1cm](r1)at(4.4,2){};\node[title](rr1)at(4.4,2){$r_1$};
      \node[node,text width=1cm](p2)at(.6,1){};\node[title](pp2)at(.6,1){$p_2$};
      \node[node](cb)at(1.3,1){$a$};
      \node[node,text width=1cm](q2)at(2,1){};\node[title](qq2)at(2,1){$q_2$};
      \node[node,text width=1cm](q'2)at(3,1){};\node[title](qq'2)at(3,1){$q'_2$};
      \node[node,text width=1cm](q''2)at(4,1){};\node[title](qq''2)at(4,1){$q''_2$};
      \node[node](db)at(4.7,1){$b$};
      \node[node,text width=1cm](r2)at(5.4,1){};\node[title](rr2)at(5.4,1){$r_2$};
      \node[node,text width=1.5cm,](u)at(2.25,0){\qquad$u$};
      \node[node,text width=1.5cm](v)at(3.75,0){\qquad$v$};

      \draw[->,above](2.8,2.8)--node{$\sigma^{n_1}$}(1.1,2.2);
      \draw[->,above](3.2,2.8)--node{}(4.9,2.2);
      \draw[->,above](2.1,1.8)--node{}(.1,1.2);
      \draw[->,left](2.5,1.8)--node{$\sigma^{k}$}(2.5,1.2);
      \draw[->,left](3.5,1.8)--node{}(3.5,1.2);
      \draw[->,left](3.9,1.8)--node{}(5.9,1.2);
      
    \end{tikzpicture}
    \caption{The derivation tree $T_\sigma(e,n)$.}\label{figureFiniteaubCase1}
\end{figure}
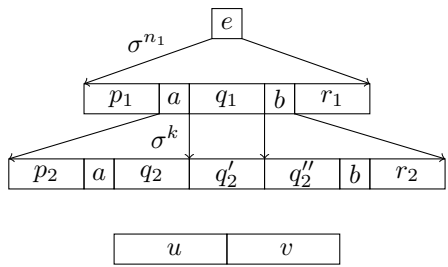
  Thus we have 
  \begin{align*}
    \sigma^{n_1}(e)&=p_1aq_1br_1,\\
    \sigma^{k}(a)=p_2aq_2,\quad \sigma^{k}(q_1)&=q'_2,\quad \sigma^{k}(b)=q''_2br_2
  \end{align*}
  and $uv=q_2q'_2q''_2$ (see Figure~\ref{figureFiniteaubCase1}).
   If $q_2$ is non-erasable, then $\sigma$ is not
  left-prolongable on $au$,
  a contradiction. Thus $q_2$ is erasable and similarly $q''_2$ is erasable.
  This shows that the length of $uv$ is bounded.
  One has actually
  \begin{eqnarray*}
|uv|&=&|q_2q'_2q''_2|=|q_2|+|\sigma^{k}(q_1)|+|q''_2|.
  \end{eqnarray*}
   Set
   $K=\max\{|\sigma^m(a)|\mid m\ge 1,\mbox{ and $a\in A$ non-growing}\}$.
   Let also $L$ be the maximal length of erasable words in $\cL(\sigma)$
   (see Proposition \ref{propositionErasing}). We have
   $|\sigma^k(q_1)|\le K|\sigma|^{n_1}\le K|\sigma|^{\Card(A)^2}$
   and since $q_2,q''_2$ are erasable, $|q_2|,|q''_2|\le L$. Thus
  \begin{displaymath}
    |uv|\le K|\sigma|^{\Card(A)^2}+2L
    \end{displaymath}
showing that $|uv|$ is bounded in terms of $\sigma$.
  
\end{proof}

\begin{proofof}{of Theorem~\ref{propositionFixedPointMinimal2}}
  Let $B,C,D,E$ and $n$ be as in Lemma~\ref{lemmaABC}. We change $\sigma$
  for $\sigma^n$. Suppose first that the
  graph induced by $G(\sigma)$ on $B$ is reduced to a loop on $a\in A$
  and thus that $\sigma(a)=uav$ with $u,v$ non-growing. Since
  $a$ is growing, we cannot have $u$ and $v$ erasable.
  We may assume that $u$ is non-erasable. Then, by Lemma~\ref{lemma2cases},
  some power $\sigma^n$ of $\sigma$ has an admissible periodic right-infinite
  fixed point $w^\omega$ and
  thus the admissible two-sided infinite fixed
  point $w^\infty$.

  Assume next that graph induced on $B$ is not reduced to a loop
  on one vertex. Then for every $a\in B$, considering the tree
  $T_\sigma(a,n)$ for $n$ large enough (as in Figure~\ref{figureFiniteaubCase1}),
  we can then find $m\ge 1$ (at most equal to the least integer $k$
  such that $M(\sigma)^k$ restricted to $B\times B$ is positive) and $b,c\in B$
  such that
  \begin{eqnarray*}
    \sigma^m(a)&=&pbqcr,\\
    \sigma^m(b)&=&sbt,\\
    \sigma^m(c)&=&ucv
    \end{eqnarray*}
  with $q,t,u\in (C\cup D\cup E)^*$.
  If $u$ or $t$ is non-erasable,
  then by Lemma~\ref{lemma2cases} or its symmetric version,
  some power of $\sigma$
  has a periodic admissible one-sided
  fixed point and thus a periodic admissible two-sided fixed point. Otherwise,
  assume that $\sigma^n(u)=\sigma^n(t)=\varepsilon$.
  Then, by Lemma~\ref{lemma2cases}, $\sigma^n$
  is right-prolongable on $w=\sigma^{\dpower{n}}(u)c$
  and, symmetrically, it is left-prolongable
  on $z=bt\sigma(t)\cdots\sigma^{n-1}(t)$.
  Since $q$ is non-growing, we may assume
  that $\sigma^{2n}(q)=\sigma^n(q)$.
  We can change $\sigma$ for $\sigma^n$.
   Then $\sigma$
  is left-prolongable on $z\sigma(q)$, right-prolongable
  on $w$ 
  and $z\sigma(q)w$ is in $\cL(\sigma)$. Therefore,
  $\sigma^{\omega}(z\sigma(q)\cdot w)$ is an admissible
  two-sided fixed point of $\sigma$.

  This proves the existence of the fixed point. Let now
  $x$ be an arbitrary admissible two-sided fixed point  of $\sigma^n$.
  Then $x^+$ and $x^-$ are one-sided fixed points. If $x^+$
  is in $F(\sigma^n)^\omega$
  and  $x^-$
   in $^\omega F(\sigma^n)$, then $x$ is formed of non-growing
  letters and there is a finite number of orbits of such points
  by Proposition~\ref{lemmaNonGrowing}. If $x^-$ is formed of
  non-growing letters and $x^+$ contains growing letters,
  then by Proposition~\ref{theoremHeadLando} we have $x^+=\sigma^{n\omega}(u)$
  where $\sigma^n$ is right-prolongable on $u$.
  By Proposition~\ref{lemmaOneSidedFixedPoint}, there is a finite number of
  orbits of points of this type which are admissible fixed points of a power
  of $\sigma$. The case where $x^-=\sigma^{n\omega}(v)$ and $x^+$
  is formed of non-growing letters is symmetrical. Finally,
  assume that $x=\sigma^{n\omega}(u\cdot v)$,
  where $\sigma^n$ is left-prolongable on $u$
  and right-prolongable on $v$. Set $u=paq$ and $v=rbs$
  with $q,r$ non-growing. Then $\sigma^n$ is left-prolongable
  on $aq$ and right-prolongable on $rb$.
  Then $x=\sigma^{n\omega}(aq\cdot rb)$ and 
by Lemma~\ref{lemmaRightExtendable},
  there is a finite number of points of this form.
\end{proofof}
\begin{example}
  Let $A=\{a,b\}$ and let $\sigma$ be morphism
  $a\mapsto a,b\mapsto baabab$. Then $\sigma^\omega(ba\cdot b)$
  is an admissible two-sided fixed point. 
\end{example}
Note that, as for one-sided fixed points, the number of admissible two-sided
fixed points need not be finite, as shown in the following example.
\begin{example}
  Let $\sigma\colon a\mapsto bc, b\mapsto bd, c\mapsto ec, d\mapsto d,
  e\mapsto e$, as in Example~\ref{exampleInfiniteNbFP}. Then, for any
  $n\ge 0$, the sequences $^\omega d e^n\cdot e^\omega$ and
  $^\omega d\cdot d^ne^\omega$ are admissible
  two-sided fixed points of $\sigma$.
\end{example}
In the same way as we did for one-sided fixed points
(Proposition~\ref{propositionAdmissibleOneSided}), we now
give the following characterization of two-sided admissible
fixed points.
\begin{proposition}\label{propositionAdmissibleTwoSided}
  Let $\sigma\colon A^*\to A^*$ be a morphism. A sequence
  $x\in \XS(\sigma)$ is a two-sided infinite fixed point of $\sigma$
  if and only if one of the following conditions is satisfied.
  \begin{enumerate}
  \item[\rm(i)] $x={}^\omega r s \cdot t u^\omega$ with $r,s,t,u\in F(\sigma)^*$
    and $r^*stu^*\subset\cL(\XS(\sigma))$.
  \item[\rm(ii)] $x={}^\omega r s \cdot \sigma^\omega(t)$
    with $r,s\in F(\sigma)^*$,   $\sigma$  right-prolongable on $t$ and $r^*st\subset \cL(\XS(\sigma))$.
  \item[\rm(iii)] $x=\sigma^{\tilde{\omega}}( r) \cdot s t^\omega$ with
     $\sigma$ left-prolongable on $r$,
    $s,t\in F(\sigma)^*$ and $rst^*\subset\cL(\XS(\sigma))$.
  \item[\rm(iv)] $x=\sigma^{\tilde{\omega}}(r)\cdot\sigma^\omega(s)$
    with  $\sigma$ left-prolongable on $r$ and right-prolongable on $s$ and $rs\in\cL(\sigma)$.
    \end{enumerate}
\end{proposition}
\begin{proof}
  The result follows applying Proposition \ref{propositionAdmissibleOneSided}
  to $x^-$ and $x^+$.
\end{proof}

\paragraph{Quasi-fixed points}
Let us say that a two-sided infinite
sequence $x$ is a \emph{quasi-fixed point} if its
orbit is stable by $\sigma$, that is  $\sigma(\Orb(x))\subset \Orb(x)$.
Thus $x$ is a quasi-fixed point
if and only if $\sigma(x)=S^kx$ for some $k\in\Z$.
A  quasi-fixed point $x$ is \emph{admissible} if 
$x$ is in $\XS(\sigma)$.
For two-sided infinite sequences  $x,y$,
we write $x\sim y$ if $x=S^ky$ with $k\in\Z$, that
is, if $x,y$ are in the same orbit.
If $a$ is a letter such that $\sigma(a)=uav$ with $u,v$ non-erasable
and $i=|ua|$, we denote
\begin{displaymath}
  \sigma^{\omega,i}(a)=\cdots \sigma^2(u)\sigma(u)u\cdot av\sigma(v)\sigma^2(v)\cdots
  \end{displaymath}
Note that $\sigma^{\omega,i}$ is defined for every $a\in A$
such that $\sigma(a)=uav$ with $u,v$ non-erasable and $i=|ua|$.
Otherwise, it is undefined.
The following result is from \cite{ShallitWang2002}
(see also \cite[Theorem 7.4.3]{AlloucheShallit2003}).
\begin{proposition}\label{propositionQuasiFixedPoints}
  Let $\sigma\colon A^*\to A^*$ be a morphism.
  A two-sided sequence $x\in A^\Z$ is a quasi-fixed point of $\sigma$ if and
  only if one of the following conditions is satisfied.
  \begin{enumerate}
  \item[\rm(i)] $x$ is a shift of a two-sided infinite fixed point of $\sigma$.
  \item[\rm(ii)] $x\sim\sigma^{\omega,i}(a)$ for some $a\in A$
    such that $\sigma(a)=uav$ with $u,v$ non-erasable and $i=|ua|$.
  \item[\rm(iii)] $x=(uv)^\infty$ for some non-empty words $u,v$
    such that $\sigma(uv)=vu$.
    \end{enumerate}
\end{proposition}
The following example illustrates case (ii).
\begin{example}
  Let $\sigma\colon a\mapsto bab,b\mapsto b$. Then
  $\sigma^{\omega,2}(a)={}^\omega b\cdot a b^\omega$ is a quasi-fixed point of $\sigma$.
\end{example}
Let us now characterize admissible quasi-fixed points.
\begin{proposition}\label{propositionAdmissibleQuasiFP}
 Let $\sigma\colon A^*\to A^*$ be a morphism.
 A two-sided infinite sequence $x\in \XS(\sigma)$
 is a quasi-fixed point of $\sigma$ if and
 only if one of the following conditions is satisfied.
  \begin{enumerate}
  \item[\rm(i')] $x$ is a shift of an admissible two-sided infinite
    fixed point of $\sigma$.
  \item[\rm(ii')] $x\sim\sigma^{\omega,i}(a)$ for some $a\in A$
    such that $\sigma(a)=uav$ with $u,v$ non-erasable and $i=|ua|$.
  \item[\rm(iii')] $x=(uv)^\infty$ for some non-empty words $u,v$
    such that $\sigma(uv)=vu$.
  \end{enumerate}
  There is a finite and computable
  number of orbits of admissible quasi-fixed points
  of powers of $\sigma$.
\end{proposition}
\begin{proof}
  The conditions are clearly sufficient.

  Conversely, let $x\in \XS(\sigma)$ be  a quasi-fixed point of $\sigma$.
  By Proposition \ref{propositionQuasiFixedPoints}, one
  of the three conditions (i), (ii) or (iii) holds.
  If condition (i) holds, we have $x\sim y$ where $\sigma(y)=y$.
  Since $\XS(\sigma)$ is closed under the shift, we have also
  $y\in \XS(\sigma)$ and thus $y$ is admissible. Thus (i') holds.
  Next (ii) is the same as (ii') and (iii) is the same as (iii').

  Finally, by Theorem \ref{propositionFixedPointMinimal2}, there
  is a finite and computable number of orbits of quasi-fixed points
  of type (i'). The number of orbits of quasi-fixed points of type
  (ii') is at most $\Card(A)$. Finally, the quasi-fixed points
  of type (iii') are made of non-growing letters and there
  is a finite number of orbits of points of this type by Proposition~\ref{lemmaNonGrowing}.
  \end{proof}
\section{Periodic points}\label{sectionPeriodic}
We now investigate the periodic points in substitution shifts.

We will need the following technical lemma.
\begin{proposition}\label{propositionnewX(sigman)}
  Let $\sigma\colon A^*\to A^*$ be a morphism. There is an
  alphabet $B$ containing $A$ and a morphism $\tau\colon B^*\to B^*$
  such that $\XS(\sigma)=\XS(\tau)$ and $\XS(\tau^n)=\XS(\tau)$
  for every $n\ge 1$. Moreover $\sigma^n$ and $\tau^n$
  have the same finite and infinite fixed points for every $n\ge 1$.
\end{proposition}
\begin{proof}
  Let $p$ be the least common multiple of the periods of
  the strongly connected components of the graph $G(\sigma)$.
  Let $B=A\cup A\times\{1,\ldots p-1\}$. For every $a\in A$,
  set $\tau(a)=\sigma(a)$,
  $\tau(a,1)=a$ and $\tau(a,i)=(a,i-1)$
  for every $2\le i\le p-1$.

  We have clearly $\XS(\tau)=\XS(\sigma)$. Indeed, $\cL(X(\sigma))\subset \cL(X(\tau))$. Conversely, if $w\in \cL(X(\tau))$, then
  $w$ is in $A^*$ and thus $w\in\cL(X(\sigma))$.

  Let us show that $\XS(\tau^n)=\XS(\tau)$. Consider a word $w\in \cL(\XS(\tau))$.
  There are arbitrary large integers $k$ such that $w$ is a factor
  of $\tau^k(a)$ for some $a\in A$. If $k$ is large enough, we have
  $k=k_1+k_2+k_3$ and $b\in A$ such that $b$ appears in $\tau^{k_1}(a)$
  and in $\tau^{k_2}(b)$ and that $w$ is a factor of $\tau^{k_3}(b)$.
  By the definition of $p$, this implies that there exists $\ell$
  such that $w$ is a factor of $\tau^{\ell+qp}(a)$ for every $q\ge 0$.
  Let $t$ be such that $tn\ge \ell$ and set $tn-\ell=qp+r$
  with $q\ge 0$ and $0\le r<p$. Then $w$ is a factor of
  $\tau^{tn}(a,r)=\tau^{\ell+qp+r}(a,r)=\tau^{\ell+qp}(a)$ and
  thus $w$ is in $\cL(\tau^n)$.

  It is clear that $\sigma^n$ and $\tau^n$ have the same finite and
  infinite fixed points for $n\ge 1$.
  \end{proof}

The following result relates periodic points
to fixed points.
It is a complement to Proposition~\ref{lemmaNonGrowing}.
\begin{proposition}\label{lemmaNouveau}
  Let $\sigma\colon A^*\to A^*$ be a morphism. Every periodic point
  $x\in \XS(\sigma)$ is a  shift of some fixed point of a power of $\sigma$.
\end{proposition}
\begin{proof}
We have seen in Proposition~\ref{lemmaNonGrowing} that
  the result is true if $x$ has only non-growing letters.
  Let us assume that $x$ contains growing letters.
  Let $p$ be the minimal period of $x$. Set $y^{(0)}=x$.
  Since $x$ is in $\XS(\sigma)$,
  there is a sequence  $(y^{(n)},k_n)_{n\ge 1}$ such that
  $(y^{(n+1)},k_{n+1})$ is a $\sigma$-representation of $y^{(n)}$
  for $n\ge 0$ with $y^{(n)}\in \XS(\sigma)$
  (by Proposition \ref{propositionBKM}). Thus
  \begin{equation}
    y^{(n)}=S^{k_{n+1}}\sigma(y^{(n+1)}).
  \end{equation}
  
  Let $A_n$ be the set of letters appearing in $y^{(n)})$.
  By the pigeonhole principle, there is an infinity of $n$
  such that $A_n=B$. 
  This forces the sequence $(A_n)$ to be periodic.
  Changing $\sigma$ for some power, we may assume that
  $A_0=B$ and $\sigma(B)\subset B^*$.
  Then $\sigma(B)\subset B^*$ and every letter of $B$ appears
  in some $\sigma(b)$ with $b\in B$.  Let
  $\tau$ be the restriction of $\sigma$ to $B^*$.

  Since every
  letter of $B$ appears in every $y^{(n)}$, every word of
  $\cL(\tau)$ appears in $x$. Thus, since $x$ is periodic,
  every growing letter $a$ of $x$
  appears with bounded gaps and every letter $b\in B$
  appears in some $\sigma^n(a)$. 
  This implies that the shift $\XS(\tau)$ is minimal. Since
  $x$ is in $\XS(\tau)$, the shift $\XS(\tau)$ is formed
  of the shifts of $x$, and thus $x$ is a shift of a fixed point of a power
  of $\tau$.
\end{proof}

A morphism $\sigma\colon A^*\to C^*$ is \emph{elementary}
if for every decomposition $\sigma=\alpha\circ\beta$ with
$\beta\colon A^*\to B^*$ and $\alpha\colon B^*\to C^*$,
one has $\Card(B)\ge\Card(A)$. An elementary morphism
is clearly non-erasing.

This notion has been
introduced in \cite{EhrenfeuchtRozenberg1978}
where it is shown that, if $\sigma$ is elementary, then
$\sigma$ defines an injective map from $A^\N$
to $C^\N$.

Note that the property
of being elementary is decidable. Indeed, if
$\sigma\colon A^*\to C^*$ is a morphism
consider the finite family $\cal F$ of sets $U\subset C^*$
such that $\sigma(A)\subset U^*\subset C^*$
with every $u\in U$ being a factor of some $\sigma(a)$ for $a\in A$.
Then $\sigma$ is elementary if and only if
$\Card(U)\ge \Card(A)$ for every $U\in\cal F$.

The proof of the following result follows
an argument from \cite[Lemma 2]{Pansiot1986}
(where the result is actually stated for admissible fixed points
of morphisms),
see also \cite[Exercise 2.38]{DurandPerrin2021}.
\begin{lemma}\label{lemmaPansiotPeriodicity}
  Let $\sigma\colon A^*\to A^*$ be an elementary morphism
  having a periodic point $x\in \XS(\sigma)$. 
  The minimal period of $x$ can be effectively bounded in terms of
  the morphism $\sigma$.
\end{lemma}
\begin{proof}
Let $x\in \XS(\sigma)$ be a periodic point. By Proposition~\ref{lemmaNouveau},
  it is a shift of some fixed point of a power of $\sigma$.

  If $x$ has no growing letter, its period can be bounded
  by Proposition~\ref{lemmaNonGrowing}. Assume that $x$
  contains at least one growing letter. 
  
  We first note that since $x$ is periodic,
  there are no right-special words in $\cL(x)$
  of length larger than the period of $x$.
    We claim that there is no  word in $\cL(x)$
  which is right-special with respect to $\cL(x)$
   and which begins with a growing letter.

  Assume the contrary. Let $u=u_0\in\cL(x)$ be a right-special factor of $x$
  beginning with a growing letter $a$.
  Let $b,c\in A$ be distinct letters such that $ub,uc\in\cL(x)$.
  Since $\sigma$ is elementary, it is injective on $A^{\N}$
  and thus there exist $v,w\in\cL(x)$ such that
  $ubv,ucw\in\cL(x)$ with $\sigma(bv)\ne\sigma(cw)$.
  Let $u_1$ be the longest common prefix of $\sigma(ubv)$ and $\sigma(ucw)$.
  Then $u_1$ is again a right-special factor of $x$
  beginning with $\sigma(a)$. Continuing
  in this way, we build a sequence $u_0,u_1,u_2,\ldots$ of
  right-special factors of $x$ beginning with $\sigma^n(a)$.
  Since $a$ is growing, their lengths go to $\infty$
  and this contradicts the fact that $x$ is periodic.
  This proves the first claim.

  We now claim that the length of the factors of $x$ formed
  of non-growing letters is bounded. To prove this, let $a$ be a growing
  letter of $x$. Since $x$ is periodic, we may assume that $a=x_0$.
  We may also assume that $x$ is a fixed point of $\sigma$
  and thus that $x=w^\infty$ with $\sigma(w)=w^n$
  (Lemma~\ref{lemmaPeriodicFixedPoints}). Since $x$
  has some growing letters, we have $n\ge 2$. We
  may assume that all letters occur in $w$.

 Let $\ell(w)$ be the row vector with components $|w|_b$
 for $b\in A$. We have then $M(\sigma)\ell(w)=n\ell(w)$, that is,
  $\ell(w)$ is a non-zero column eigenvector of $M(\sigma)$
  for the eigenvalue $n$.
  Since, for a nonnegative matrix, the dominant eigenvalue
  is the only eigenvalue having a positive eigenvector,
  the eigenvalue $n$ is the dominant eigenvalue
  of $M$ and  the eigenvalue $n$ is simple.
  The coefficients of a corresponding eigenvector can be
  effectively bounded in terms of $\sigma$.
  
  For every eigenvector $z$ of $M(\sigma)$ relative to $n$,
  consider the ratio
  \begin{displaymath}
    r(z)=\sum_{b \in A_f}z_b/\sum_{b\in A_i}z_b,
  \end{displaymath}
  where $A_i$ is the set of growing letters and $A_f$ is its complement.
  Since $r(z)=r(\lambda z)$, this ratio
  is a constant $\rho$ depending only on $\sigma$. Now consider an
  integer $e$ such that $w^e$ has more than $\Card(A)$
  occurrences of growing letters.
  We have a factorization $w^e=u_1u_2\cdots u_ts$ where each $u_i$
  has $\Card(A)+1$ growing letters and $s$ has less than
  $\Card(A)+1$ growing letters. If, for some integer $k$,
  every $u_i$ has a run of
  $k(\Card(A)+1)$ non-growing letters, we have
  \begin{eqnarray*}
    \rho&=&r(\ell(w))=r(\ell(w^e)))=\frac{\sum_{i=1}^t\ell_f(u_i)+\ell_f(s)}{\sum_{i=1}^t\ell_i(u_i)+\ell_i(s)}\\
    &\ge&\frac{\sum_{i=1}^t\ell_f(u_i)}{\sum_{i=1}^t\ell_i(u_i)+(\Card(A)+1)}\\
      &\ge&\frac{k(\Card(A)+1)t}{(\Card(A)+1)(t+1)}\ge k/2
  \end{eqnarray*}
  with $\ell_f(u)=\sum_{b\in A_f}|u|_b$ and $\ell_i(u)=\sum_{b\in A_i}|u|_b$.
  We conclude, taking $k=2\rho+1$, that $w^e$ has a factor with $\Card(A)+1$
  growing letters separated by at most $(2\rho+1)(\Card(A)+1)-1$
  non-growing letters. Such a factor has a factor of the form
  $bub$ where $b$ is a growing letter. By the first claim,
  $ub$ is the only return word to $b$ since otherwise,
  there would be a right-special factor of $x$
  beginning with $b$. As a consequence, $|ub|$ is a period
  of $x$.

   Since $|ub|< (2\rho+1)(\Card(A)+1)\Card(A)$, this shows that the minimal period
   of $x$ is bounded by $(2\rho+1)(\Card(A)+1)\Card(A)$.
\end{proof}
When $\sigma$ is growing,
the above proof shows that the period of $x$ is bounded by $\Card(A)$.
Note the following subtle point. We know, by Theorem~\ref{propositionFixedPointMinimal2}, that the number of orbits of admissible fixed points of $\sigma$,
and thus the number of orbits of its periodic points, is
effectively bounded. However, this does not imply that
the period of such points is effectively bounded
and thus that the problem of deciding whether
a fixed point is periodic is decidable. Lemma~\ref{lemmaPansiotPeriodicity}
implies that it is decidable for an elementary morphism
and Theorem \ref{theoremDecidableAperiodic}
below shows that it is true for every morphism.
\begin{example}
  The morphism $\sigma\colon a\mapsto a, b\mapsto baab$, is elementary.
  Since $\sigma(aab)=(aab)^2$, the sequence $x=(aab)^\infty$
  is a fixed point of $\sigma$.
\end{example}
The following was proved in~\cite{Pansiot1986}
and also \cite{HarjuLinna1986} for a primitive morphism
(the result of \cite{Pansiot1986}  and \cite{HarjuLinna1986}
is formulated in terms
of a fixed point, see the comments after the proof).
\begin{theorem}\label{theoremDecidableAperiodic}
  It is decidable whether a morphism $\sigma\colon A^*\to A^*$ is aperiodic.
  The set of periodic points in $\XS(\sigma)$ is finite and
  their periods are effectively bounded in terms of $\sigma$.
  \end{theorem}
\begin{proof}
  By Proposition~\ref{lemmaNouveau}, it is enough to prove that one
  may decide whether a morphism $\sigma$ has a power with
  a periodic admissible fixed point.  From Proposition \ref{lemmaNonGrowing},
  it is decidable whether a morphism has a periodic
  point with non-growing letters.

  Thus
  it is enough to decide whether $\sigma$ has a power
  with a periodic admissible  fixed point
  with some growing letters.
  
   Assume first that $\sigma$ is elementary. By
  Lemma~\ref{lemmaPansiotPeriodicity},
  the minimal period $p$ of every periodic
  point $x\in X(\sigma)$  is bounded in terms of $\sigma$.
To
  check whether $\sigma$ acts as a permutation
  of a set of points of period $p$,
  we consider the set $F$ of
  primitive words of length $p$ and look for cycles
  in the graph with edges $(u,v)$ with
  $u,v\in F$ and $\sigma(u)$ a power of $v$.
  Since these periodic points are assumed to have at least one
  growing letter, they will be in $\XS(\sigma)$.
  Indeed, if $a$ is such a letter, all factors of
  $x$ are factors of some $\sigma^n(a)$. This proves
  the decidability in this case.
  
Otherwise, write $\sigma=\alpha\circ\beta$ with
$A^*\edge{\beta}B^*\edge{\alpha}A^*$ and $\Card(B)<\Card(A)$.
Using induction on $\Card(A)$, we may assume that we can check
whether $\tau=\beta\circ\alpha$ is aperiodic. Let $x$ be  a fixed
point of $\sigma$ and set $y=\beta(x)$. Then
\begin{displaymath}
  \tau(y)=\beta\circ\sigma(x)=\beta(x)=y
\end{displaymath}
and thus $y$ is a fixed point of $\beta\circ\alpha$.
Moreover, $x$ is periodic if and only if $y$ is periodic
(because $y=\beta(x)$ and $x=\alpha(y)$)
and $x$ has a growing letter for $\sigma$ if and only
if $y$ has a growing letter for $\tau$
(because $\beta\circ\sigma^n=\tau^n\circ\beta$
for all $n\ge 1$). Thus
$x\in \XS(\sigma)$ if and only if $y\in \XS(\tau)$. Since,
by induction hypothesis,
we can test the existence of $y$, the same is true for $x=\alpha(y)$.
Using again induction on $\Card(A)$, we conclude that the
number of periodic points is finite and that their periods are bounded
in terms of $\sigma$.
\end{proof}
\begin{example}
  Let $\sigma\colon a\mapsto ab,b\mapsto ac,c\mapsto ac$. Then
  $\sigma^\omega(c\cdot a)=(abac)^\infty$, which is periodic
  of period $4>\Card(A)$. The morphism $\sigma$
  is not elementary since $\sigma=\alpha\circ\beta$ with
  $\alpha\colon d\mapsto ab,e\mapsto ac$ and $\beta\colon a\mapsto d,
  b\mapsto e, c\mapsto e$. The point $y=\beta((abac)^\infty)$ is
  $(de)^\infty$.
\end{example}
We note the following corollary, concerning the case
of an admissible fixed point.
\begin{corollary}\label{corollaryPeriodicOneSided}
  It is decidable whether a one-sided (resp. two-sided) admissible fixed point
  of a morphism $\sigma\colon A^*\to A^*$ is periodic.
\end{corollary}
\begin{proof}
  This follows from Theorem~\ref{theoremDecidableAperiodic} since
  the period of such a fixed point is bounded in terms of $\sigma$.
\end{proof}
Note that when the substitution is primitive, Corollary~\ref{corollaryPeriodicOneSided} is equivalent to Theorem~\ref{theoremDecidableAperiodic}
since, the shift being in this case minimal, it is periodic
if and only if every point is periodic.
The result of \cite{Pansiot1983} and \cite{HarjuLinna1986} is actually
formulated as the decidability of the eventual periodicity of
an admissible one-sided fixed point $\sigma^\omega(a)$ of a primitive morphism.
Note also that this statement was generalized in
\cite{Durand2013} as the decidability of the eventual periodicity
of the image $\phi(\sigma^\omega(u))$
by a morphism $\phi$ of an admissible one-sided fixed point
of a morphism $\sigma$.

We now prove the following complement to
Theorem~\ref{theoremDecidableAperiodic}.
\begin{theorem}\label{theoremDecidablePeriodic}
  It is decidable whether the shift $\XS(\sigma)$
  generated by a morphism $\sigma$ is periodic.
\end{theorem}
To prove Theorem~\ref{theoremDecidablePeriodic}, we use
the following characterization of periodic substitution shifts.
\begin{lemma}\label{lemmaCaractPeriodicSubstitution}
  The shift $\XS(\sigma)$ is periodic if and only if every admissible
  quasi-fixed point of a power of $\sigma$ is periodic.
\end{lemma}
\begin{proof}
  The condition is clearly necessary. Conversely,
  let $\sigma\colon A^*\to A^*$ be such that $\XS(\sigma)$
  is not periodic. Let $x\in \XS(\sigma)$ be a sequence that
  is not periodic. We distinguish three cases.
  
  Case 1. Suppose that $x$ is formed of non-growing letters.
  Since there is a finite number of orbits of sequences
  formed of non-growing letters, and in particular of non-periodic
  ones, there exists at least one such orbit stable under
  a power of $\sigma$. Thus there is
  a quasi-fixed point of a power of $\sigma$ which is not periodic.

  Case 2. There is exactly one growing letter $a$ in $x$.
  If some $\sigma^n(x)$ has more than one growing letter,
  we change $x$ for $\sigma(x)$ and go to Case 3.
  Otherwise, replacing $\sigma$
  by some power, $a$ by another growing letter,
  and $x$ by its image by this power, we may assume
  that $\sigma(a)=uav$ with $u,v$ non-growing. We may also assume that for every $n\ge 1$,
  $x_{[-n,n]}$ is a factor of some $\sigma^m(b)$ with $b$ a growing letter.
  By choosing $n$ large enough and replacing $x$ by
  some $\sigma^n(x)$, we may assume that $b=a$. This implies
  that $u,v$ are both non-erasable. Then $\sigma^{\omega,i}(a)$
  is, for $i=|ua|$, a non-periodic quasi-fixed point of a power
  of $\sigma$.

  Case 3. There are several growing letters in $x$. Let
  $aub$ be a factor of $x$ such that $a,b$ are growing letters
  and $u$ a non-growing word.  We may assume, replacing $\sigma$
  by some power $x$ by some $\sigma^n(x)$ and $a,b$ by other growing letters,
  that $\sigma(a)=paq$, $\sigma(q)=q$, $\sigma(u)=u$, $\sigma(b)=rbs$
  and $\sigma(r)=r$. Then $\sigma^\omega(aq\cdot urb)$
  is a non-periodic fixed point of a power of $\sigma$.
  \end{proof}
\begin{proofof}{of Theorem~\ref{theoremDecidablePeriodic}}
  By Theorem \ref{propositionAdmissibleQuasiFP}, there
  is a finite and computable number of orbits of admissible
  quasi-fixed points of a power of $\sigma$. Since
  the period of those which are periodic is effectively bounded
  by Theorem~\ref{theoremDecidableAperiodic},
  the result follows from Lemma~\ref{lemmaCaractPeriodicSubstitution}.
  \end{proofof}
\section{Erasable letters}\label{sectionErasable}
Let $\sigma:B^*\to B^*$ 
and  $\phi:B^*\to A^*$ be two morphisms. We say that the pair $(\sigma,\phi)$
is \emph{admissible} if  $\phi(x)$ is
a two-sided infinite sequence for every $x\in \XS(\sigma)$.
When $(\sigma,\phi)$ is admissible, we let $\XS(\sigma,\phi)$ denote
the closure under the shift of $\phi(\XS(\sigma))$.

A shift $\XS(\sigma,\tau)$, for an admissible pair $(\sigma,\tau)$,
is called a \emph{morphic} shift. A substitution shift
$\XS(\sigma)$ is also called a \emph{purely morphic} shift.

As a closely related notion, a \emph{morphic sequence}
is a right-infinite sequence $x=\phi(\sigma^\omega(u))$ where
$\sigma\colon B^*\to B^*$ is a morphism
right-prolongable on $u$ and $\phi\colon B^*\to A^*$
is a morphism.

The following statement  is a 
 result essentially due to Cobham \cite{Cobham1968}
(see also \cite{Pansiot1983}, \cite{CassaigneNicolas2003},
\cite{Durand2013}, \cite{Honkala2009} and
\cite{CharlierLeroyRigo2015}
and \cite[Exercise 7.38]{DurandPerrin2021}).
It is also a variant of a result known as Rauzy Lemma
(see Proposition~\ref{propositionRauzyLemma} below).
It is a normalization result, showing that,
for a shift generated by a morphic sequence,
one can replace a pair
$(\tau,\phi)$  by an equivalent pair $(\zeta,\theta)$
with $\zeta$ non-erasing and $\theta$ letter-to-letter
(see Figure~\ref{figureRauzyProper}).
Our statement is slightly more general than the previous
ones, as the fixed point is of the form $\sigma^\omega(u)$
for $u$ a word and not a letter.

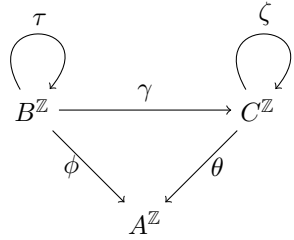
\begin{figure}[hbt]
\centering
\tikzset{node/.style={circle,draw,minimum size=0.1cm,inner sep=0pt}}
\tikzstyle{every loop}=[->,shorten >=.3pt]
\tikzstyle{loop above}=[in=50,out=120,loop]
\begin{tikzpicture}
\node(B)at(0,1.5){$B^\Z$};
\node(C)at(3,1.5){$C^\Z$};
\node(A)at(1.5,0){$A^\Z$};

\draw[above](B)edge[loop above]node{$\tau$}(B);
\draw[above](C)edge[loop above]node{$\zeta$}(C);
\draw[left,->](B)edge node{$\phi$}(A);
\draw[right,->](C)edge node{$\theta$}(A);
\draw[above,->](B)edge node{$\gamma$}(C);
\end{tikzpicture}
\caption{The pairs $(\tau,\phi)$ and  $(\zeta,\theta)$}\label{figureRauzyProper}
\end{figure}

\begin{theorem}\label{theoremNormalization}
  Let $\tau:B^*\to B^*$ 
   be a morphism right-prolongable on $u\in B^+$ and
   let $y=\tau^\omega(u)$. Let $\phi\colon B^*\to A^*$
   be such that $x=\phi(y)$ is an infinite sequence.
  There exist an
  alphabet $C$, a morphism $\gamma:B^*\to C^*$, a non-erasing
  morphism $\zeta:C^*\to C^*$,  
  right-prolongable on $v=\gamma(u)$
  and a letter-to-letter morphism $\theta:C^*\to A^*$
  such that 
  \begin{equation}
    \phi\circ\tau^m=\theta\circ\gamma\label{eqThetaGamma}
  \end{equation}
 and
\begin{equation}
    \zeta\circ\gamma=\gamma\circ\tau^n\label{eqZetaGamma}
\end{equation}
for some $m,n\ge 1$. As a consequence, 
$x=\theta(\zeta^\omega(v))$.
\end{theorem}

For this, we use the following lemma (see~\cite{CassaigneNicolas2003}).
It is not a constructive statement since the sequence $(n_i)$ is
not shown to be computable. This weakness was
overcome in \cite{Durand2013} and \cite{Honkala2009}
independently.

\begin{lemma}\label{lemmaCN}
  Let $B$ be a finite set and for each $b\in B$, let
  $(\ell_n(b))_{n\ge 0}$ be a sequence of natural integers
  with $(\ell_n(a))$ unbounded for some $a\in B$.
  There is an infinite sequence $n_0<n_1<\ldots$ such that
  for each $b\in B$, the sequence $(\ell_{n_i}(b))$ is
  either strictly increasing or constant,
  with  $(\ell_{n_i}(a))$  strictly increasing.
\end{lemma}
\begin{proof}
  The proof is by induction on $\Card(B)$. The statement
  is clear if $B=\{a\}$. Assuming that it holds
  for $B\setminus \{b\}$, either the sequence  $(\ell_{n_i}(b))$
  is bounded and we can find a subsequence $m_i$ of the $n_i$
  such that $(\ell_{m_i}(b))$ is constant, or
  $(\ell_{n_i}(b))$
  is unbounded and we can find a subsequence $m_i$
  of the $n_i$ such that $(\ell_{m_i}(b))$ is stricly increasing.
\end{proof}
\begin{proofof}{of Theorem~\ref{theoremNormalization}}
By Lemma~\ref{lemmaRightProl}, and changing $\tau$ for
some of its powers, we may assume that $u=pra$ with
$\tau(p)=p$ and $\tau(r)=\varepsilon$  and $a\in B$ growing.

We first prove that we can change the pair $(\tau,\phi)$
into a pair $(\tau',\phi')$ such that 
\begin{equation}
|\phi'\circ\tau'(b)|\ge|\phi'(b)|\label{eqphiCircvarphi}
\end{equation}
for every $b\in B$ and with strict inequality when $b=a$.

Since,
by Lemma~\ref{lemmaGrowing},
$\lim_{n\to +\infty}|\tau^n(a)|=+\infty$, we can apply Lemma~\ref{lemmaCN}
to the sequence $\ell_n(b)=|\phi\circ\tau^n(b)|$. Thus,
there are $m,n\ge 1$ such that
$|\phi\circ\tau^m(b)|\le|\phi\circ\tau^{m+n}(b)|$
for every $b\in B$ with strict inequality if $b=a$.
Set  $\tau'=\tau^{n}$, $\phi'=\phi\circ\tau^m$.
Then 
\begin{eqnarray*}
|\phi'\circ\tau'(b)|&=&|\phi\circ\tau^m\circ\tau^{n}(b)|=|\phi\circ\tau^{m+n}(b)|\\
&\ge&|\phi\circ\tau^m(b)|=|\phi'(b)|
\end{eqnarray*}
for every $b\in B$ with strict inequality when $b=a$.
This proves \eqref{eqphiCircvarphi}.

We define an alphabet $C=\{b_p\mid b\in B, 1\le p\le|\phi'(b)|\}$,
a map $\theta:C\to A$ by $\theta(b_p)=(\phi'(b))_p$
where $(\phi'(b))_p$ denotes the $p$-th letter of $\phi'(b)$ and a map
$\gamma:B\to C^*$ by $\gamma(b)=b_1b_2\cdots b_{|\phi'(b)|}$
with $\gamma(b)=\varepsilon$ if $\phi'(b)=\varepsilon$.
In this way, we have $\theta\circ\gamma=\phi'$
and thus \eqref{eqThetaGamma}.

For every $b\in B$, we have 
\begin{displaymath}
|\gamma\circ\tau'(b)|=|\phi'\circ\tau'(b)|\ge|\phi'(b)|
\end{displaymath}
by \eqref{eqphiCircvarphi}. Thus, we can define words $w_1,w_2,\ldots,w_{|\phi(b)|}\in C^+$ such that
\begin{displaymath}
\gamma\circ\tau'(b)=w_1w_2\ldots w_{|\phi(b)|}.
\end{displaymath}
Then we define a non-erasing
morphism $\zeta\colon C^*\to C^*$
by $\zeta(b_p)=w_p$.
We  have by construction $\zeta\circ\gamma=\gamma\circ\tau'$
and thus \eqref{eqZetaGamma}. Since $\zeta\circ\gamma(y)=\gamma\circ\tau'(y)=\gamma(y)$, the sequence $\gamma(y)$ is a fixed point of $\zeta$.
Since  $\zeta(v)=\zeta\circ\gamma(u)=\gamma\circ\tau'(u)$,
the sequence $\zeta(v)$ begins with $v$. Moreover,
since $|\phi'(r)|\le|\phi\circ\tau'(r)|=0$, we have
\begin{eqnarray*}
  |\zeta(v)|&=&|\gamma\circ\tau'(u)|=|\phi'\circ\tau'(pra)|
  =|\phi'(p)|+|\phi'\circ\tau'(a)|\\
  &>&|\phi'(p)|+|\phi'(a)|=|\phi'(u)|=|\gamma(u)|=|v|.
  \end{eqnarray*}
Thus $\zeta$ is right-prolongable on $v$. Finally, we have
\begin{displaymath}
  \theta(\zeta^\omega(v))=\theta(\zeta^\omega(\gamma(u)))
  =\theta\circ\gamma(\tau'^\omega(u))=\phi'(y)=\phi\circ\tau^m(y)=\phi(y)=x.
  \end{displaymath}
\end{proofof}

We illustrate the proof of Theorem~\ref{theoremNormalization}
with the following example.
\begin{example}
  Consider the morphism
  $\tau\colon a\mapsto abc,b\mapsto ac,c\mapsto\varepsilon$
  and the morphism $\phi$ which is the identity on $A=B=\{a,b,c\}$.
  The sequence $\tau^\omega(a)$ is the Fibonacci sequence with
  letters $c$ inserted at the division points.
  
  We first replace the pair $(\tau,\phi)$ by the pair $(\tau,\tau)$
  in such a way that $|\tau\circ\tau(e)|\ge|\tau(e)|$ for
  every letter $e\in B$. We define $C=\{a_1,a_2,a_3,b_1,b_2\}$,
  $\gamma(a)=a_1a_2a_3$, $\gamma(b)=b_1b_2$ and $\gamma(c)=\varepsilon$.
  The map $\theta\colon C\to A$ is defined as shown below.
  \begin{displaymath}
    \begin{array}{|c|c|c|c|c|c|c|}\hline
      c&a_1&a_2&a_3&b_1&b_2\\ \hline
      \theta(c)& a& b  &c & a &c\\ \hline
    \end{array}
  \end{displaymath}
  Next, we set $\zeta(a_1)=\zeta(b_1)=a_1a_2$
  and $\zeta(a_2)=a_3b_1$, $\zeta(b_2)=a_3$, $\zeta(a_3)=b_2$. Then $\zeta\circ\gamma=\gamma\circ\tau$
  as one may verify.

  Note that a more economical solution is $\zeta\colon a\mapsto ab$,
  $b\mapsto ca$ and $c\mapsto c$, since $\tau^\omega(a)=\zeta^\omega(a)$
  as one may verify easily.
\end{example}
\begin{corollary}\label{corollaryCobhamShift}
  For every  morphic shift $X=\XS(\sigma,\phi)$ such that $\XS(\sigma)$
  is minimal, there
  is an admissible pair $(\zeta,\theta)$, with $\zeta$ non-erasing
  and $\theta$ letter-to-letter, such that $X=\XS(\zeta,\theta)$.
\end{corollary}
\begin{proof}
  Let $y$ be an admissible one-sided fixed point of a power of $\sigma$.
  By Proposition~\ref{propositionAdmissibleOneSided}, either
  $y$ is eventually periodic or
  $y=\sigma^{n\omega}(u)$ where $\sigma^n$ is right-prolongable
  on $u$.
  If $y$ is eventually periodic, it has to be periodic
  since $\XS(\sigma)$ is minimal. Then $\phi(y)$ is periodic
  and there exists a non-erasing morphism
  $\zeta$ such that $\XS(\zeta)=X$. Otherwise, 
  By Theorem~\ref{theoremNormalization},
  there exists a pair $(\zeta,\theta)$ with $\zeta$ non-erasing
  and $\theta$ letter-to-letter such that $\phi(y)=\theta(\zeta^\omega(v))$.
  Since $\XS(\sigma)$ is minimal, $X=\XS(\sigma,\phi)$ is also
  minimal and thus $X$ is genererated by $\phi(y)$,
  which implies our conclusion.
\end{proof}
It is possible to prove Corollary \ref{corollaryCobhamShift}
without the hypothesis
that $X(\sigma)$ is minimal (see~\cite{BealDurandPerrin2023}).



The following example is from \cite{CassaigneNicolas2003}.
It shows that, in Theorem \ref{theoremNormalization},
even when the morphism $\phi$ is the identity, one cannot dispense
with the morphism $\theta$.
\begin{example}\label{exampleCassaigneNicolas}
  Let $\sigma\colon a\mapsto abccc,b\mapsto baccc,c\mapsto\varepsilon$
  and let $x=\sigma^\omega(a)$.
  We will show that there is no non-erasing morphism $\sigma'$ such that
  $\XS(\sigma')=\XS(\sigma)$.

  Assume the contrary.
  Let $\chi\colon \{a,b,c\}^*\to\{a,b\}^*$ be the morphism which erases $c$,
  that is, $\chi\colon a\mapsto a,b\mapsto b,c\mapsto \varepsilon$.
  Then $\chi\circ\sigma=\mu\circ\chi$ where $\mu\colon a\mapsto ab,b\mapsto ba$
  is the Thue-Morse morphism.

  For every  $u\in\cL(\sigma)$, the word $\sigma'(u)$ is in $\cL(\sigma)$
  and thus $\chi\circ\sigma'(u)$ is a factor of the Thue-Morse
  sequence. Set $A=\{a,b,c\}$. Since
  $\sigma'(\cL(\sigma))\subset\cL(\sigma)$,
  we have $\sigma'(A)\subset A^*$.

  Let us first show that $\sigma'(c)=c$. Indeed, since $\chi\circ\sigma'(ccc)=(\chi\circ\sigma'(c))^3$ and since $\cL(\mu)$ does not contain cubes (see
  for example~\cite[Corollary2.2.4]{Lothaire1983}), we have $\chi\circ\sigma'(c)=\varepsilon$. Since
  $\cL(\sigma)\cap c^*=\{\varepsilon, c,cc,ccc\}$, this forces $\sigma'(c)=c$.

  Next, $\sigma'(a)$ (resp. $\sigma'(b)$)
  cannot begin or end with $c$. Indeed, if for example, $\sigma'(a)=cw$, then $\sigma'(ccca)$
  begins with $c^4$, which is impossible.

  One of $\sigma'(a),\sigma'(b)$ contains $c$ since otherwise
  $\{\sigma'(ab),\sigma'(ba)\}=\{ab,ba\}$ which contradicts
  the fact that some power of $\sigma'$ has a two-sided
  infinite fixed point (since $\XS(\sigma)=\XS(\sigma')$ is aperiodic, this fixed
  point is aperiodic and thus $\sigma'$ has at least one growing letter).

  If $\sigma'(a)$ does not contain $c$, then $\sigma'(a)=a$ or $\sigma'(a)=b$.
  In the first case, $\sigma'(b)$ begins with $bccc$ and ends with $cccb$
  and thus $\sigma'^2(b)$ has a factor $cccbccc$, which is impossible.
  The second case is similar.

  Similarly, $\sigma'(b)$ contains $c$.

   Thus, either $\sigma'(a)$ and $\sigma'(b)$ are both in
  $\{a,b\}ccc\{a,b\}A^*\cap A^*\{a,b\}ccc\{a,b\}$
  or  both in $\{a,b\}^2cccA^*\cap A^*ccc\{a,b\}^2$. In the first case,
  \begin{eqnarray*}
    \sigma'^2(a)&\in& \sigma'(\{a,b\}ccc\{a,b\}A^*)
    \subset\sigma'(\{a,b\})cccA^*
    \subset A^*\{a,b\}ccc\{a,b\}cccA^*,
    \end{eqnarray*}
  which is impossible since two factors $c^3$ must be separated
  by a word of length $2$. In the second case,
  $\sigma'(ab)\in A^*\{a,b\}ccc\{a,b\}^4ccc\{a,b\}A^*$,
  which is also impossible. Thus, such a $\sigma'$
  cannot exist.
\end{example}

We do not know,  for an endomorphism $\tau$, whether it is decidable if
there is
a non-erasing morphism $\zeta$ such that $\XS(\tau)=\XS(\zeta)$.
\section{Recognizability}\label{sectionRecognizability}
Let $\sigma\colon A^*\to B^*$ be a morphism and let $X$ be a shift
space on $A$. The morphism $\sigma$ is \emph{recognizable}
in $X$ for a point $y\in B^\Z$ if $y$
has at most one centered $\sigma$-representation $(x,k)$ with $x\in X$.
It is \emph{fully recognizable} in $X$ if it is recognizable
in $X$ for every point $y\in B^\Z$.

The following result is proved in~\cite{BealPerrinRestivo2021}.
\begin{theorem}\label{theoremRecognizable}
  Every morphism $\sigma$ is recognizable in $\XS(\sigma)$
  for aperiodic points.
  \end{theorem}

This result has a substantial history.
It was first proved by Moss\'e \cite{Mosse1992,Mosse1996}
that every aperiodic primitive morphism
is recognizable in $\XS(\sigma)$. Later, this important result was generalized
in \cite{BezuglyiKwiatkowskiMedynets2009} where
it is shown that every aperiodic non-erasing morphism $\sigma$ is
recognizable in $\XS(\sigma)$. As a further extension, it is
proved in \cite{BertheSteinerThuswaldnerYassawi2019} that
a non-erasing morphism $\sigma$ is recognizable in $\XS(\sigma)$
at aperiodic points. Finally, the last result was extended
in \cite{BealPerrinRestivo2021} to arbitrary morphisms
(see also the more recent \cite{BealPerrinRestivoSteiner2023} for a further extension).

We prove the following result.
\begin{theorem}\label{theoremDecidabilityRecognizable}
  A morphism $\sigma$ is fully recognizable in $\XS(\sigma)$
  if and only if every periodic point in $\XS(\sigma)$
  is formed of non-growing letters. Consequently,
it is decidable whether $\sigma$ is fully recognizable in $\XS(\sigma)$.
\end{theorem}
We first prove the following lemma, which does not seem to
have been proved before.
\begin{lemma}\label{lemma6.4}
  Let $\sigma\colon A^*\to A^*$ be a morphism. For every
  aperiodic point $x\in \XS(\sigma)$, the point $\sigma(x)$
  is aperiodic.
\end{lemma}
\begin{proof}
  Assume that for some $x\in \XS(\sigma)$, the point
  $y=\sigma(x)$ is periodic. By
  Proposition~\ref{lemmaNouveau},  $y$ is
  a fixed point of a power of $\sigma$. We may assume that
  it is a fixed point of $\sigma$.
  
  Set $x=x^{(1)}$. For every $n\ge 1$, let $(k_n,x^{(n+1)})$ be a $\sigma$-representation of $x^{(n)}$ (its existence follows from \cite[Proposition 5.1]{BealPerrinRestivo2021}).

  Since $\sigma^n(x^{(n)})$ is a shift of $y$, we have for every $k,n\ge 1$
  \begin{equation}
    \sigma^n(\cL_k(x^{(n)}))\subset\cL(y).\label{eqsigman}
  \end{equation}
  But since $\sigma(y)=y$, we have also $\sigma(\cL(y))\subset\cL(y)$.
  Thus, we have
  \begin{equation}
    \sigma^m(\cL_k(x^{(n)}))\subset\cL(y).\label{eqsigmanm}
  \end{equation}
  for all $m\ge n$. The set
  \begin{equation}
    R_k=\cup_{n\ge 1}\cL_k(x^{(n)})\label{eqDefRk}
  \end{equation}
  is finite. Thus, for every $k$, there is an integer  $m(k)$ such that
  \begin{equation}
    \sigma^n(R_k)\subset \cL(y)\label{eqRk}
  \end{equation}
  for every $n\ge m(k)$.
  
  Now for every $w\in \cL_k(x)$ and every $n\ge 1$, there is some
  $w_n\in\cL(x^{(n+1)})$ such that $w$ is a factor of $\sigma^n(w_n)$.
  We may assume that each $w_n$ is chosen of minimal length.

  The
  lengths of the words $w_n$ are bounded. Indeed, if $|w_n|\ge k$,
  then, since $|w|=k$ and since $w_n$ is chosen of minimal
  length, the word  $w_n$ has at least $|w_n|-k$ erasable letters. But a word of
  $\cL(\sigma)$ has a bounded number of consecutive erasable letters
  because otherwise there would exist an infinite sequence of erasable
  letters in $\XS(\sigma)$, which is impossible.
  Thus we obtain that $|w_n|\le K$ for some constant $K\ge 1$.

  For $n\ge m(K)$, we have by \eqref{eqRk}, $\sigma^n(w_n)\subset \cL(y)$.
  This shows that $w$ is in $\cL(y)$. Since $\cL_k(x)\subset\cL(y)$
  for all $k$, we conclude that $x$ is periodic.
\end{proof}

\begin{proofof}{of Theorem~\ref{theoremDecidabilityRecognizable}}
  By Theorem~\ref{theoremRecognizable}, $\sigma$ is fully
  recognizable in $\XS(\sigma)$
  if and only if it is recognizable in $\XS(\sigma)$ for every periodic
  point of $\XS(\sigma)$. By Proposition~\ref{lemmaNouveau}, every
  periodic point in $\XS(\sigma)$ is a fixed point of a power of
  $\sigma$. By Theorem~\ref{propositionFixedPointMinimal2}
  there is a finite number of orbits of fixed points of a power of $\sigma$
  and these orbits form a computable set.
  Thus, there is a finite number of periodic points in $\XS(\sigma)$,
  which form a computable set. Changing
  $\sigma$ for some of its powers, we may assume that every periodic
  point is a fixed point of $\sigma$ (the exponent is bounded
  in terms of the size of $\sigma$ by Lemma~\ref{lemmaGrowing}). Let $x$ be one of them. Set
  $x=w^\infty$ with $w$ a primitive word (the length of $w$
  is bounded by Theorem~\ref{theoremDecidableAperiodic}).
  Then $\sigma(w)=w^n$
  for some $n=n(x)\ge 1$ by Lemma~\ref{lemmaPeriodicFixedPoints}.  Let $(y,k)$ be a $\sigma$-representation
  of $x$ with $y\in \XS(\sigma)$. By Lemma~\ref{lemma6.4}, $y$ is periodic and thus $\sigma(y)=y$.
  This forces $k$ to be a multiple of $|w|$. Thus
  the only possible centered $\sigma$-representations $(y,k)$ of $x$
  with $y\in \XS(\sigma)$ are of the form $(x,k)$
  with $k=0,|w|,\ldots,(n-1)|w|$. 
   Thus,
  $\sigma$ is recognizable in $\XS(\sigma)$ if and
   only if $n(x)=1$ for every periodic point $x\in \XS(\sigma)$.
   Since $n(x)=1$ if and only if $x$ is formed of non-growing
   letters, this proves the statement.
\end{proofof}
\section{Irreducible substitution shifts}\label{sectionIrreducible}
Let us say that a morphism $\sigma$ is \emph{irreducible}
if $\XS(\sigma)$ is irreducible.

The following result shows that it is decidable whether a morphism such that
$\cL(\sigma)=\cL(\XS(\sigma))$ is irreducible.
\begin{theorem}\label{theoremIrreducible}
  A morphism $\sigma\colon A^*\to A^*$ such that $\cL(\sigma)=\cL(\XS(\sigma))$
  is irreducible if and only if there is a letter $a\in A$ such that
  \begin{enumerate}
    \item[\rm(i)] its strongly
  connected component in $G(\sigma)$ is
  a non-trivial aperiodic graph not reduced to a cycle.
  \item[\rm(ii)] There is an $n\ge 1$
  such that $|\sigma^n(a)|_b\ge 1$
  for every letter $b\in A$.
  \end{enumerate}
\end{theorem}
\begin{proof} Set $X=\XS(\sigma)$.
  If the condition  is satisfied, there are
  words in $\cL(X)$ with an arbitrary large number of
  occurrences of $a$ and thus there
  is some $t\in \cL(X)$ such that $a$ appears twice in $t$.
  If $u,v\in\cL(X)$ there are $b,c\in A$ such
  that $u,v$ appear in $\sigma^m(b),\sigma^p(c)$.
  By condition (ii),
  $b$ and $c$ appear in $\sigma^n(a)$ for all $n$ large enough.
  Then a word of the form $uwv$ appears in $\sigma^{n+q}(t)$
  for $q\ge m,p$
  and thus there is a word $w$ such that $uwv$ is in $\cL(X)$.
  We conclude that $\sigma$ is irreducible.

  Conversely, assume $\sigma$ irreducible. 
  
 Let $a, b$ be two letters of $\cL(\sigma)$. Since $\sigma$ is
irreducible, there is a point in $\XS(\sigma)$ containing a factor $w$
containing both the letters $a$ and $b$. There is an integer $n$ such
that $w$ is a factor of some $\sigma^n(c)$ for some letter $c$.
Thus $a$ and $b$ are connected in $G(\sigma)$.
Thus we may assume that $G(\sigma)$ is connected.

   Let $a$ be a letter such that the strongly connected
  component $C(a)$ of $a$ in $G(\sigma)$ is minimal
  (that is, $a$ is only accessible from an element of $C(a)$).

  The component $C(a)$ cannot be trivial since otherwise $a$
  is not in $\cL(X)$. It cannot either be reduced
  to a cycle. Indeed, in this case, the words of $\cL(\sigma)$
  have at most one occurrence of $a$. But since $\cL(X)=\cL(\sigma)$,
  $a$ is in $\cL(X)$. Since $X$ is irreducible,
  $\cL(X)$ contains some $aua$, a contradiction.

   Thus we can assume that $C(a)$ is not reduced to a cycle.
  The component $C(a)$ is aperiodic. Indeed, suppose that
  $C(a)=C_0\cup C_1\cup\ldots\cup C_{p-1}$ with all edges
  going from $C_i$ to $C_{i+1}$ with the indices taken modulo $p$.
  For $b,c\in C(a)$, there is a word $u$ such that
  $buc\in\cL(X)$. This implies that there is some $d\in C(a)$
  such that $buc$ appears in $\sigma^n(d)$ for some $n\ge 1$.
  But then $b,c$ are in the same $C_i$, which shows that $p=1$.

  If $b$ is another letter such that its strongly connected
  component $C(b)$ is also minimal,
  since $aub$ is also in $\cL(X)$ for some word $u$,
  this forces $C(b)=C(a)$. Thus
  every letter is accessible from $a$. Since  $C(a)$
  is aperiodic, there is  a path from $a$ to a given letter
  of all lengths large enough. Thus conditions (i) and (ii)
  are satisfied.
\end{proof}
Note that all conditions of Theorem~\ref{theoremIrreducible}
are decidable. Indeed, the condition $\cL(\sigma)=\cL(\XS(\sigma))$ is
decidable by Lemma~\ref{lemmaDecidability}. Conditions
(i) and (ii) are easily verified on the graph $G(\sigma)$.

The following example shows that Theorem~\ref{theoremIrreducible}
is false when $\cL(\sigma)\ne\cL(\XS(\sigma))$. We have
no proof that irreducibility is decidable in this
more general case.
\begin{example}
  The morphism $\sigma\colon a\mapsto ab,b\mapsto b$ is irreducible
  and even minimal since $\XS(\sigma)=b^\infty$ but the strongly connected components
  of $a,b$ are both reduced to a cycle.
\end{example}

\
\section{Minimal substitution shifts}\label{sectionMinimal}

The  following statement was proved in~\cite{DamanikLenz2006}
(see also~\cite[Proposition 6.3]{DurandPerrin2021}).
 
 \begin{theorem}\label{theoremCaractMinimal}
   A morphism $\sigma\colon A^*\to A^*$ such that $\cL(X(\sigma))=\cL(\sigma)$
   is  minimal if and only if 
   there is a growing 
     letter $a\in \cL(\sigma)$  which appears
     with bounded gaps in $\cL(\sigma)$ and for every $b\in\cL(\sigma)$
     there is $n\ge 1$ such that $|\sigma^n(a)|_b\ge 1$.
 \end{theorem}
\begin{proof}
  The condition is sufficient.
  Consider indeed $w\in\cL(\sigma)$.
Let $b\in A$ be such that $w$ appears in $\sigma^m(b)$.

Now there is $n\ge 1$ such that
$b$ appears in $\sigma^n(a)$. Thus $w$ appears in $\sigma^{n+m}(a)$.
Since $a$ occurs with bounded gaps in $\cL(\sigma)$,
the same is true for $w$.
    This shows that
   $\XS(\sigma)$ is uniformly recurrent and thus minimal.

   Conversely, assume that $\sigma$ is minimal. Let $a\in\cL(\sigma)$
   be a growing letter. Since $\XS(\sigma)$ is minimal, it
   is uniformly recurrent and thus every $b\in A$
   appears with bounded gaps in $\cL(\sigma)$. Thus, the condition
    is satisfied.
 \end{proof}
 We give an example illustrating
  Theorem~\ref{theoremCaractMinimal}.

 \begin{example}\label{exampleChacon}
   The morphism $\sigma\colon 0\mapsto 0010,1\mapsto 1$ is known as the
   binary \emph{Chacon morphism}. The condition
    in Theorem \ref{theoremCaractMinimal} is clearly satisfied by
   the letter $0$ and thus $\sigma$ is minimal.
   \end{example}
 Theorem~\ref{theoremCaractMinimal} allows us to prove the
 following result. It is actually related to the result
 of \cite{Durand2013b} which states that it is decidable
 whether the image by a morphism $\phi$
 of an admissible one-sided fixed point of
 a morphism $\sigma$ is uniformly recurrent.
 \begin{theorem}\label{theoremDecidableMinimal}
It is decidable
whether a morphism $\sigma$ such that $\cL(\sigma)=\cL(\XS(\sigma))$
is minimal.
 \end{theorem}
 \begin{proof}
   For each growing letter $a\in\cL(\sigma)$, condition  is decidable.
  Indeed, one has $|\sigma^n(a)|_b\ge 1$ for some $n$ if and
  only if $b$ is accessible from $a$ in $G(\sigma)$ and
  $a$ appears with bounded gaps in $\cL(\sigma)$ if and
  only if for every $b\in A$, the language $\cL(\sigma)\cap (A\setminus\{b\})^*$
  is finite, which is decidable by Lemma~\ref{lemmaDecidabilityL(sigma)}.
 \end{proof}
  Let us analyze the time complexity of the above algorithm.
 
  For each growing letter $a \in \cL(\sigma)$, the condition
  of Theorem~\ref{theoremCaractMinimal} is decidable. Indeed, one
has $|\sigma^n(a)|_b \geq 1$ for some $n$ if and only if $b$ is
accessible from $a$ in $G(\sigma)$. A letter $a \in \cL(\sigma)$ does not appear with bounded
gaps in $\cL(\sigma)$ if and only if $\cL(\sigma) \cap (A \setminus
\{a\})^*$ is infinite. This is the case if and only if there is a growing letter
$b$ such that there is no path from $b$ to $a$ in $G(\sigma)$.
Thus, the condition can be checked in linear time in the size of $\sigma$.

 \paragraph{Primitive morphisms}

The following was proved in \cite[Theorem 2.1]{MaloneyRust2018}
for non-erasing morphisms. We show that the proof can be adapted
to arbitrary morphisms. The result below is actually also a particular
case of a more general result from \cite[Theorem 3.1]{DurandLeroy2020}
which states that every shift $\XS(\sigma,\phi)$ is
conjugate to a primitive substitution shift.
We give nevertheless the proof in our case, which is
substantially simpler (and already known to Fabien Durand).
\begin{theorem}\label{theoremMaloneyRust}
  If $\sigma$ is minimal, then $\XS(\sigma)$ is conjugate to $\XS(\tau)$
  with $\tau$ primitive.
\end{theorem}
The proof uses the following variant of Theorem~\ref{theoremNormalization}
which is called Rauzy Lemma in \cite[Proposition 7.2.10]{DurandPerrin2021}
(see also~\cite[Proposition 23]{DurandHos&Skau1999}).
For the sake of completeness, we reproduce the proof.
A non-erasing morphism $\phi\colon B^*\to A^*$ is said to be
\emph{circular} if the morphism $\phi$ is fully recognizable in $B^\Z$
(see~\cite{BealPerrinRestivo2021} for an alternative definition).
\begin{proposition}\label{propositionRauzyLemma}
  Let $\sigma\colon B^*\to B^*$ be a primitive morphism
  and let $\phi\colon B^*\to A^*$ be a circular morphism.
  Let $X=\XS(\sigma,\phi)$ be the closure of $\phi(\XS(\sigma))$
  under the shift.
    There exist an
  alphabet $C$, a primitive
  morphism $\zeta\colon C^*\to C^*$
  and a letter-to-letter morphism $\theta\colon C^*\to A^*$
  such that $\theta$ is a conjugacy from $\XS(\zeta)$ onto~$X$.
\end{proposition}
\begin{proof}
  We may assume that we are not in the trivial case where $B=\{b\}$
  and $\sigma(b)=b$.
  Then, since $\sigma$ is primitive, it is growing and, by substituting $\sigma$ by a power of itself,
  we may assume that $|\sigma(b)|\ge|\phi(b)|$ for all $b\in B$.
  As in the proof of Theorem~\ref{theoremNormalization}
  we define
  \begin{displaymath}
    C=\{b_p\mid b\in B,\ 1\le p\le|\phi(b)|\}.
  \end{displaymath}
  Next
  $\gamma(b)=b_1b_2\cdots b_{|\phi(b)|}$ and $\theta(b_p)=\phi(b)_p$,
  where $\phi(b)_p\in A$ denotes the $p$-th letter of $\phi(b)$.
  Since $\phi$ is non-erasing, we have for every $b\in B$,
  $|\gamma\circ\sigma(b)|\ge|\sigma(b)|\ge|\phi(b)|$.

  For $b\in B$ and $1\le p\le |\phi(b)|$, we define
  \begin{displaymath}
    \zeta(b_p)=\begin{cases}
    \gamma(\sigma(b)_p)&\mbox{if $1\le p<|\phi(b)|$}\\
    \gamma(\sigma(b)_{[|\phi(b)|,|\sigma(b)|]})&\mbox{ if $p=|\phi(b)|$}
    \end{cases}
  \end{displaymath}

   We have then by definition for every $b\in B$,
  \begin{eqnarray*}
    \zeta\circ\gamma(b)&=&\zeta(b_1\cdots b_{|\phi(b)|})\\
    &=&\gamma(\sigma(b)_1)\gamma(\sigma(b)_2)\cdots \gamma(\sigma(b)_{|\phi(b)|-1})
    \gamma(\sigma(b)_{[|\phi(b)|,|\sigma(b)|]})
    =\gamma\circ\sigma(b).
  \end{eqnarray*}
  Thus $\zeta\circ\gamma=\gamma\circ\sigma$. 

  We claim that $\zeta$ is primitive. 
Let $n$ be an
integer such that $b$ occurs in $\sigma^n(a)$ for all $a,b\in B$. Let
$b_p$ and $c_q$ be in $C$. 
By construction, $\zeta(b_p)$
contains $\gamma\left( \sigma(b)_p\right)$ as a factor, thus
$\zeta^{n+1}(b_p)$ contains  $\zeta^n\left( \gamma\left( \sigma(b)_p\right)\right)=
\gamma\left(\sigma^n\left(\sigma(b)_p\right)\right)$ as a factor. 
By the
choice of $n$, the letter $c$ occurs in $\sigma^n\left(\sigma(b)_p\right)$, thus
$\gamma(c)$ is a factor of
$\gamma\left(\sigma^n\left(\sigma(b)_p\right)\right)$, and also of 
$\zeta^{n+1}(b_p)$. 
Since $c_q$ is a letter of $\gamma(c)$, $c_q$
occurs in  $\zeta^{n+1}(b_p)$ and our claim is proved.

\begin{figure}[hbt]
\centering
\tikzset{node/.style={circle,draw,minimum size=0.1cm,inner sep=0pt}}
\tikzstyle{every loop}=[->,shorten >=.3pt]
\tikzstyle{loop above}=[in=50,out=120,loop]
\begin{tikzpicture}
\node(B)at(0,1.5){$\XS(\sigma)$};
\node(C)at(3,1.5){$\XS(\zeta)$};
\node(A)at(1.5,0){$\XS(\tau,\phi)$};

\draw[left,->](B)edge node{$\phi$}(A);
\draw[right,->](C)edge node{$\theta$}(A);
\draw[above,->](B)edge node{$\gamma$}(C);
\end{tikzpicture}
\caption{The pairs $(\sigma,\phi)$ and  $(\zeta,\theta)$}\label{figureRauzyLemma}
\end{figure}
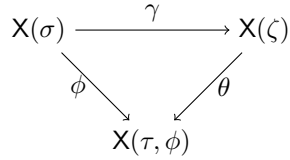

Changing $\sigma$ for some of its
powers, we may assume that some
$x\in \XS(\sigma)$ is an admissible fixed point of $\sigma$.
Since $\zeta\circ\gamma(x)=\gamma\circ\sigma(x)=\gamma(x)$,
$y=\gamma(x)$ is a fixed point of $\zeta$. Since $\zeta$ is primitive,
$\XS(\zeta)$ is minimal and thus $\XS(\zeta)$ is generated by $y$.
Since $\theta(y)=\theta\circ\gamma(x)=\phi(x)$, we have $\theta(\XS(\zeta))\subset \XS(\sigma,\phi)$.
Since $\XS( \sigma,\phi)$ is minimal, this implies $\theta(\XS(\zeta))=\XS(\sigma,\phi)$
(see Figure~\ref{figureRauzyLemma}).

There remains to show that $\theta$ is injective.
Consider $x\in X$ and $z\in \XS(\zeta)$ such that $x=\theta(z)$. Set $z_0=b_{p+1}$
with $b\in B$ and $0\le p<|\phi(b)|$. Then $b_0=z_{-p-1}$ and thus
we have $z=S^p\gamma(y)$ for some $y\in B^\Z$ and
\begin{displaymath}
  x=\theta(z)=\theta(S^p(\gamma(y))=S^p(\theta(\gamma(y)))=S^p(\phi(y)).
  \end{displaymath}Since
$\phi$ is circular, the pair $(y,p)$ is unique and thus $\theta$
is injective.
  \end{proof}
\begin{proofof}{of Theorem~\ref{theoremMaloneyRust}}
If $\XS(\sigma)$ is periodic, we have $\XS(\sigma)=\XS(\tau)$ where $\tau$
sends every letter to the primitive word $w$ such that $\XS(\sigma)=w^\infty$.
Since $\tau$ is primitive, the result is true.

Otherwise, since $\XS(\sigma)$ is minimal, $\XS(\sigma)$ is aperiodic.
By Theorem~\ref{propositionFixedPointMinimal2}, there
is an admissible fixed point of the form $x=\sigma^\omega(r\cdot \ell)$
where $r,\ell$ are non-empty and $r\ell\in\cL(\sigma)$.
Let  $\RR(r\cdot\ell)$ be the set of non-empty words $u\in\cL(\sigma)$
such that the word $r u \ell$
\begin{enumerate}
  \item[(i)] begins and ends with $r\ell$,
  \item[(ii)] has only two occurrences of $r\ell$ and
    \item[(iii)] is in $\cL(\sigma)$.
\end{enumerate}
Note that every word  $u$ such that $ru\ell$ begins and ends with $r\ell$
and is in $\cL(\sigma)$
belongs to $\RR(r\cdot\ell)^*$, that is, is a concatenation
of words of the set $U=\RR(r\cdot\ell)$.

Since $\XS(\sigma)$ is minimal, it is uniformly recurrent. It
implies that the set $U$
is finite. Moreover, it is a circular code. This means
(see \cite{DurandPerrin2021})
that every two-sided infinite sequence $x$ has at most
one factorization in terms of words of $U$.
Let $\phi\colon B\to U$ be a bijection from
an alphabet $B$ onto $U$. We extend $\phi$ to a morphism
from $B^*$ into $A^*$. Since $\XS(\sigma)$ is minimal,
there is an infinity of occurrences of $r\ell$ in $x$
at positive and at negative indices.
Since $U$ is a circular
code, there is a unique $y\in B^\Z$ such that $\phi(y)=x$.

For every $b\in B$, the word $r(\sigma\circ\phi(b))\ell$
begins and ends with $r\ell$. This implies that $\sigma\circ\phi(b)$
is in $U^*$ and thus there is a unique $w\in B^*$ such that
$\sigma\circ\phi(b)=\phi(w)$. Set $\tau(b)=w$. This defines a morphism
$\tau\colon B^*\to B^*$ characterized by
\begin{displaymath}
  \phi\circ\tau=\sigma\circ\phi
\end{displaymath}
Let $k>0$ be an occurrence of $r\cdot \ell$
(in the sense that $x_{[0,k)}\ell$ ends with $r\ell$)
  large enough so that   every word
$w\in\mathcal{R} (r\cdot\ell)$   appears  in the decomposition of $x_{[0,k)}$, that is,   every $b\in
B$   occurs in the finite word $u\in B^+$ defined by $\phi(u)=x_{[0,k)}$.
Let $n$ be  large enough so that $|\sigma^n(\ell)|>k$. Let $b,c\in B$. As above,
$x_{[0,k)}$ is a prefix of $\sigma^n(\ell)$, which is a prefix of
$\sigma^n(\phi(b))=\phi(\tau^n(b))$. Thus $u$ is a prefix of $\tau^n(b)$, and
  $c$ occurs in $\tau^n(b)$, hence $\tau$ is primitive.

  Clearly, $\XS(\sigma)$ is the closure under the shift of $\phi(\XS(\tau))$.
  Thus, by Proposition~\ref{propositionRauzyLemma}, $\XS(\sigma)$
  is conjugate to a primitive substitution shift. 
\end{proofof}
A shift space $X$ is called \emph{linearly recurrent} if
 there is a constant $K$ such that every word of $\cL_n(\sigma)$
 appears in every word of $\cL_{Kn}(\sigma)$.
 It is well known that a primitive substitution shift is linearly
 recurrent (see \cite[Proposition 2.4.24]{DurandPerrin2021}).
 
 A linearly recurrent shift is clearly
 minimal. The converse is true for substitution
 shifts by a result  from \cite{DamanikLenz2006} (see also \cite[Proposition 7.3.2]{DurandPerrin2021}) in the case of a substitution $\sigma$
 such that $\cL(\XS(\sigma))=\cL(\sigma)$. A proof in the
 general case was given by \cite{Shimomura2018}. We show
 how this can be derived from Theorem~\ref{theoremMaloneyRust}.
 \begin{theorem}\label{theoremDamanikLentz}
   A non-empty substitution shift is minimal if and only if it is
   linearly recurrent.
 \end{theorem}
 \begin{proof}
   Let $\sigma\colon A^*\to A^*$ be a morphism.
   If $\XS(\sigma)$ is minimal, it is conjugate to a primitive substitution
   shift $\XS(\tau)$ by Theorem~\ref{theoremMaloneyRust}. Thus
   $\XS(\sigma)$ is linearly recurrent.

   Conversely, if $\XS(\sigma)$ is linearly recurrent then
   it is clearly uniformly recurrent and thus minimal.
 \end{proof}
 
 \paragraph{Quasi-Minimal shifts}

 We now investigate the question of whether a substitution shift 
 has a finite number of subshifts.
 Following the terminology introduced in \cite{Salo2017}
 these shifts are called \emph{quasi-minimal} .

 First of all, observe
 that for every  shift space $X$ on the alphabet $A$
 the following properties are equivalent.
 \begin{enumerate}
 \item[(i)] $X$ is quasi-minimal.
   \item[(ii)] The number of distinct
     languages $\cL(x)$ for $x\in X$ is finite.
     \end{enumerate}
 Indeed, for every subshift $Y$ of $X$, one has $\cL(Y)=\cup_{y\in Y}\cL(y)$
 and thus (ii) implies (i). Conversely, for every $x\in X$, let $\alpha(x)
 =\{y\in A^\Z\mid\cL(y)\subset\cL(x)\}$. Then $\alpha(x)$
 is a subshift of $X$ and $\alpha(x)=\alpha(x')$ implies
 $\cL(x)=\cL(x')$. Thus (i) implies (ii).

 Next, if $X$ is quasi-minimal, it has a finite number of minimal
 subshifts. 

 The following example shows that a  shift
 may have a finite number of minimal subshifts without being quasi-minimal.

 \begin{example}
   Let, for $n\ge 1$,
   \begin{displaymath}
     x_n=\cdots ba^{n+2}ba^{n+1}b\cdot a^nba^{n+1}ba^{n+2}b\cdots
     \end{displaymath}
 and let $O(x_n)$ denote the orbit of $x_n$.
 Let $X=(\cup_{n\ge 1}O(x_n))\cup a^\infty$. Then
 $a^\infty$ is the  only minimal subshift of $X$
 but each set $O(x_n)\cup a^\infty$ is a subshift of $X$.
 \end{example}
 The shifts with only one minimal subshift, as in the above example, have been
 called \emph{essentially minimal} in \cite{HermanPutnamSkau1992}.
 
 The following result is proved in \cite[Proposition 5.14]{BertheSteinerThuswaldnerYassawi2019}
 for non-erasing morphisms and was proved before in
 \cite[Proposition 5.6]{BezuglyiKwiatkowskiMedynets2009} for the
 case of aperiodic non-erasing morphisms. It is interesting
 to note that, in this case, the number of minimal
 subshifts is at most $\Card(A)$.

\begin{proposition}\label{lemmaLx}
  Let $\sigma\colon A^*\to A^*$ be a morphism. The shift $\XS(\sigma)$
  is quasi-minimal.
\end{proposition}
We use the following lemma (see~\cite{BertheSteinerThuswaldnerYassawi2019}
for the non-erasing case).
\begin{lemma} \label{lemmaLimit}
  Let $\sigma\colon A^*\to A^*$ be a morphism and $x$ a point in
  $\XS(\sigma)$ such that $x$ has a $\sigma^{n}$-representation
$(x^{(n)}, 0)$ with $x^{(n)}\in \XS(\sigma)$ for all $n \geq 0$. Then $x$ is a fixed
point of some power $r$ of $\sigma$, where $r$ is bounded by a value
depending only on $\sigma$.
\end{lemma}
\begin{proof}
Let us assume that $x$ contains a growing letter $a$ at some
index $i_0 \geq 0$. 
By \cite[Lemma 2]{BealPerrinRestivo2021}, $\cL(\sigma)$ contains a finite
number of  erasable words. Thus there is an integer $K$
such that $x^{(n)}_{[0,K]}$ contains  at least $i_0+1$ non-erasable letters for
all $n$.
It follows that $x^{(n)}_{[0,K]}$ contains a growing letter for all $n$.
As a consequence, there is an index $0 \leq j \leq K$ and an infinite
set of integers $E$ such that, for all $n \in E$, $x^{(n)}_j= b$, where $b$ is a growing
letter, and $x^{(n)}_{[0,j)}=u$, where $u$ is a non-growing word.
Let $i, p$ be the constants of Lemma
\ref{propositionNonGrowingLetters}. We can choose $p'$ a multiple of
$p$ larger than $i$. Hence for each non-growing
word $z$,  $\sigma^{kp'}(z) = \sigma^{p'}(z)$ for every $k \geq 1$.

Hence we have for any $k \geq 1$,
$\sigma^{kp'}(u) = \sigma^{p'}(u)$,
$\sigma^{p'}(b) = vcw$, $\sigma^{p'}(c) =v'cw'$,
$\sigma^{kp'}(v) = \sigma^{p'}(v)$, $\sigma^{kp'}(v') = \sigma^{p'}(v')$,
where $c$ is a growing letter, $v, v'$ are non-growing words.

It follows that $v'$ is erasing and thus $w'$
is non-erasing. Hence $\sigma^{p'}(v') = \varepsilon$.
Let $u' = \sigma^{p'}(u)$, $v'' = \sigma^{p'}(v)$.

Considering an infinite number of integers $n$ of $E$ such that
$n = kp' + r$ for some fixed integer $0 \leq r < p'$, we get that
$x_{[0,\infty)} = \sigma^r(u'v''v')\sigma^r(\sigma^{p'}(c)^\omega)$
and thus $x_{[0,\infty)}$ is a fixed point of $\sigma^{p'}$.
Similarly, if $x$ contains a growing letter at some
negative index $x_{(\infty,-0)}$ is a fixed point of $\sigma^{p'}$.

Let us now assume that $x$ contains only non-growing letters at
non-negative indices. Then $x^{(n)}$ contains only non-growing letters at
non-negative indices for all $n$. Thus by Lemma
\ref{lemmaNonGrowingBis} there are finite non-growing words $z,
t^{(n)}$ such that for an infinite number of $n$,
$x^{(n)}_{[0, \infty)} = t^{(n)}z^\omega$.
  
We have
$\sigma^{p'}(z) = z'$, $\sigma^{p'}(z') = z'$, $\sigma^{p'}(t^{(p')})
= t'^{(p')}$ and $\sigma^{p'}(t'^{(p')}) = t'^{(p')}$.
Then $\sigma^{p'}(x_{[0, \infty)}) = x_{[0, \infty)}$.
Similarly if $x$ contains only non-growing letters at negative
indices, $\sigma^{p'}(x_{(-\infty, 0)}) = x_{(-\infty, 0)}$.
  \end{proof}
\begin{proofof}{of Proposition~\ref{lemmaLx}}
  Assume $\XS(\sigma)\ne\emptyset$.
We will define recursively two finite sequences of words
$w^{(k)} \in \cL(\sigma)$ and shifts $Y_k$  for $0\le k\le K$,
and for each such $k$ and all $n\ge 0$, an infinite sequence of alphabets $A_n^{(k)}$  as
follows. We set $w^{(0)} = \varepsilon$. For $k \geq 0$, when $w^{(k)}$ is defined, we define
\begin{align*}
  A_n^{(k)} &= \{ a \in A \mid w^{(k)} \text{ is not a factor of }
              \sigma^n(a)\}\\
  Y_k & = \{ x \in \XS(\sigma) \mid \text{ $x$ has a
          $\sigma^n$-representation in } \XS(\sigma) \cap (A_n^{(k)})^\Z \text{ for all } n
          \geq 0 \}.
\end{align*}
Thus  $A_n^{(0)}= \emptyset$, $Y_0 = \emptyset$.
Assume that $w^{(k)}$, $A_n^{(k)}$ and $Y_k$ are defined, we construct
$w^{(k+1)}$ when $\Card\{\cL(x) \mid x \in \XS(\sigma) \setminus Y_k\} \geq 2$ as follows.

We choose
$x,y \in \XS(\sigma) \setminus Y_k$ such that $\cL(x) \neq \cL(y)$.
Up to an exchange of $x,y$,
we may assume that $\cL(y)$ is not contained in $\cL(x)$.
We choose a word $w^{(k+1)} \in \cL(y) \setminus \cL(x)$.
Since $y \in \XS(\sigma) \setminus Y_k$, there is an integer $n \geq 0$
such that $y$ has a $\sigma^n$-representation $(z,i)$ where $z$ is not in
$(A_n^{(k)})^\Z$. Indeed every point in $\XS(\sigma)$ has for every $n\ge 0$, a
$\sigma^n$-representation in $\XS(\sigma)$ by \cite[Proposition 5.1]{BealPerrinRestivo2021}.
Hence, for some $n$, $y$ has a $\sigma^n$-representation $(z,i)$
such that for some letter  $a$ of $z$, the word $w^{(k)}$
is a factor of $\sigma^n(a)$ and thus $y$ contains the factor $w^{(k)}$.
Thus we may choose $w^{(k+1)} \in \cL(y) \setminus \cL(x)$ such that
$w^{(k)}$ is a factor of $w^{(k+1)}$. As a consequence, we obtain
$A_n^{(k)} \subset
A_n^{(k+1)}$. 


Since $x \in \XS(\sigma) \setminus Y_k$, there is an integer $m \geq 0$
such that it has a $\sigma^m$-representation in $\XS(\sigma)$
which is not in $(A_m^{(k)})^\Z$. Then there exists a letter $a \in A \setminus A_m^{(k)}$
such that $\sigma^m(a) \in \cL(x)$. Hence, $\sigma^m(a)$ does not
contain $w^{(k+1)}$ as a factor and $a \in A_m^{(k+1)}$.
Thus $A_m^{(k)} \subsetneq A_m^{(k+1)}$. 

Now for any point $z \in \XS(\sigma) \setminus Y_k$,
there is an $n \geq 0$ such that
none of the $\sigma^n$-representation of $z$  in $\XS(\sigma)$  is in $(A_n^{(k)})^\Z$ . This holds for all sufficiently large $n$ since
$\sigma(A_n^{(k)})^\Z \subset (A_{n-1}^{(k)})^\Z$.

It follows that for sufficiently large $n$, we have $A_n^{(k)} \subsetneq
A_n^{(k+1)}$. Moreover, since $w^{(k+1)} \in \cL(y)$, we have
$A_n^{(k+1)} \subsetneq A$ for all large $n$. Hence for all  $n$ large enough,
we have 
\[
A_n^{(0)} \subsetneq A_n^{(1)} \subsetneq \cdots  \subsetneq
A_n^{(k-1)} \subsetneq A_n^{(k)} \subsetneq A.
\]
As a consequence there is a  value $k=K \geq 0$ such that
$\Card\{\cL(x) \mid x \in \XS(\sigma) \setminus Y_K\} \le 1$. 

Let us show that the sequence $A_n^{(K)}$ is eventually periodic.
Following the construction of Proposition \ref{lemmaDecidabilityL(sigma)},
we define a finite automaton $\A=(Q,I,T)$ recognizing $A^*w^{(K)}A^*$
and the binary  relation $\varphi(w)$ on $Q$ by
  \begin{displaymath}
    \varphi(w)=\{(p,q)\in Q\times Q\mid p\longedge{w}q\}.
  \end{displaymath}
Let $\psi_n=\varphi\circ \sigma^n$. There are positive integers $m,p$
such $\psi_m=\psi_{m+p}$. Then $\psi_{m+r}=\psi_{m+p+r}$ for all $r\geq 0$.
Now a letter $a$ belongs to $A_n^{(K)}$ if and only if $\psi_n(a)$
contains no pair $(i,t)$ with $i \in I, t \in T$. Thus
the sequence $A_n^{(K)}$ is eventually periodic.

Let $m,p\ge 1$ be such that
$A_{m}^{(K)} = A_{m+p}^{(K)}$. We may choose $m$
multiple of $p$.
We define $\sigma' \colon A_{m}^{(K)} \rightarrow (A_{m}^{(K)})^*$ as the
restriction of $\sigma^p$ to $A_{m+p}^{(K)} = A_{m}^{(K)}$.

Let $x \in Y_K\setminus \XS(\sigma')$.
The point $x$ has a $\sigma^{pn}$-representation
$(x^{(pn)}, k_{pn})$ in
$\XS(\sigma) \cap (A_m^{(K)})^\Z$ for every $n$ such that $pn \geq m$.
Since $x \notin \XS(\sigma')$, 
there is an integer $\ell$ such that $x_{[-\ell, \ell]}$ is not a factor
of $\sigma'^r(a)$ for all $r \geq 0$ and all $a \in A_{m}^{(k)}$.
Thus there is a shift $S^h(x)$ with $-\ell \leq h \leq \ell$
such that $S^h(x)$ has a $\sigma^{pn}$-representation
$(x^{(pn)}, 0)$ in $\XS(\sigma) \cap (A_m^{(K)})^\Z$ for $n$ in some
infinite set of positive integers $E$.
Thus $S^h(x)$ has a $\sigma^{pn}$-representation
$(x^{(pn)}, 0)$ in $\XS(\sigma)$ for all $n \geq 0$.
By Lemma \ref{lemmaLimit} $x' = S^h(x)$ is a fixed point of
$\sigma^{pr}$ for some $r$ bounded by a value depending on $\sigma$.

By Theorem \ref{propositionFixedPointMinimal2} there is a finite number
of orbits of points $x$ such that $\sigma^{pr}(x) =x$. Thus there is a
finite number of languages $\cL(x)$ corresponding to such points.

As a consequence there is a finite number $N_1$ of sets $\cL(x)$ for
 $x \in Y_K \setminus X(\sigma')$.
We thus have
\begin{align*}
  \Card\{\cL(x) \mid x \in \XS(\sigma)\} &\leq  \Card\{\cL(x) \mid x \in Y_K\} +
                                     \Card\{\cL(x) \mid x \in \XS(\sigma)
                                     \setminus Y_K\} \\
                                     &\leq  \Card\{\cL(x) \mid x \in \XS(\sigma')\}\\
                                     &\quad + \Card\{\cL(x) \mid x \in Y_K
    \setminus \XS(\sigma')\}  + 1.
\end{align*}
Assume by induction hypothesis that $\Card\{\cL(x) \mid x \in
\XS(\sigma')\}$ is a finite number $N_2$, then
$\Card\{\cL(x) \mid x \in \XS(\sigma)\} \leq N_1 + N_2 +1$.
Thus $\Card\{\cL(x) \mid x \in \XS(\sigma)\}$ is finite.
\end{proofof}

\bibliographystyle{plain}
\bibliography{periodicity}

\begin{thebibliography}{10}

\bibitem{AlloucheShallit2003}
Jean-Paul Allouche and Jeffrey Shallit.
\newblock {\em Automatic Sequences}.
\newblock Cambridge University Press, Cambridge, 2003.

\bibitem{BealDurandPerrin2023}
Marie-Pierre B\'eal, Fabien Durand, and Dominique Perrin.
\newblock {\em Substitution shifts}.
\newblock 2023.
\newblock in preparation.

\bibitem{BealPerrinRestivo2021}
Marie-Pierre B\'eal, Dominique Perrin, and Antonio Restivo.
\newblock Recognizability of morphisms.
\newblock {\em Ergodic Theory Dynam. Systems}, 43(11):3578–3602, 2023.

\bibitem{BertheSteinerThuswaldnerYassawi2019}
Val\'{e}rie Berth\'{e}, Wolfgang Steiner, J\"{o}rg~M. Thuswaldner, and Reem
  Yassawi.
\newblock Recognizability for sequences of morphisms.
\newblock {\em Ergodic Theory Dynam. Systems}, 39(11):2896--2931, 2019.

\bibitem{BezuglyiKwiatkowskiMedynets2009}
Sergey Bezuglyi, Jan Kwiatkowski, and Konstantin Medynets.
\newblock Aperiodic substitution systems and their {B}ratteli diagrams.
\newblock {\em Ergodic Theory Dynam. Systems}, 29(1):37--72, 2009.

\bibitem{BruyereHanselMichauxVillemaire1994}
V\'eronique Bruy\`ere, Georges Hansel, Christian Michaux, and Roger Villemaire.
\newblock {Logic and {$p$}-recognizable sets of integers}.
\newblock {\em Bulletin of the Belgian Mathematical Society - Simon Stevin},
  1(2):191 -- 238, 1994.
\newblock Corrigendum, {\it Bull. Belg. Math. Soc.} {\bf 1} (1994), 577.

\bibitem{BealPerrinRestivoSteiner2023}
Marie-Pierre Béal, Dominique Perrin, Antonio Restivo, and Wolfgang Steiner.
\newblock Recognizability in {S}-adic shifts, 2023.
\newblock arXiv:2302.06258, to appear in Israel. J. Math.

\bibitem{CartonThomas2002}
Olivier Carton and Wolfgang Thomas.
\newblock The monadic theory of morphic infinite words and generalizations.
\newblock {\em Inf. Comput.}, 176(1):51--65, 2002.

\bibitem{CassaigneNicolas2003}
Julien Cassaigne and Fran\c{c}ois Nicolas.
\newblock Quelques propri\'et\'es des mots substitutifs.
\newblock {\em Bull. Belg. Math. Soc. Simon Stevin}, 10:661--676, 2003.

\bibitem{CharlierLeroyRigo2015}
{\'E}milie Charlier, Julien Leroy, and Michel Rigo.
\newblock Asymptotic properties of free monoid morphisms.
\newblock {\em Linear Algebra Appl.}, 500:119--148, 2016.

\bibitem{Cobham1968}
Alan Cobham.
\newblock On the {H}artmanis-{S}tearns problem for a class of tag machines.
\newblock In {\em IEEE Conference Record of Seventh Annual Symposium on
  Switching and Automata Theory}, 1968.

\bibitem{DamanikLenz2006}
David Damanik and Daniel Lenz.
\newblock Substitution dynamical systems: characterization of linear
  repetitivity and applications.
\newblock {\em J. Math. Anal. Appl.}, 321(2):766--780, 2006.

\bibitem{DiekertKrieger2009}
Volker Diekert and Dalia Krieger.
\newblock Some remarks about stabilizers.
\newblock {\em Theoretical Computer Science}, 410(30):2935--2946, 2009.
\newblock A bird’s eye view of theory.

\bibitem{Durand2013}
Fabien Durand.
\newblock Decidability of the {HD}0{L} ultimate periodicity problem.
\newblock {\em RAIRO Theor. Inform. Appl.}, 47(2):201--214, 2013.

\bibitem{Durand2013b}
Fabien Durand.
\newblock Decidability of uniform recurrence of morphic sequences.
\newblock {\em Int. J. Found. Comput. Sci.}, 24(1):123--146, 2013.

\bibitem{DurandHos&Skau1999}
Fabien Durand, Bernard Host, and Christian Skau.
\newblock Substitutional dynamical systems, {B}ratteli diagrams and dimension
  groups.
\newblock {\em Ergodic Theory Dynam. Systems}, 19(4):953--993, 1999.

\bibitem{DurandLeroy2020}
Fabien Durand and Julien Leroy.
\newblock {Decidability of the isomorphism and the factorization between
  minimal substitution subshifts}.
\newblock {\em {Discrete Analysis}}, 2022.

\bibitem{DurandPerrin2021}
Fabien Durand and Dominique Perrin.
\newblock {\em Dimension Groups and Dynamical Systems}.
\newblock Cambridge University Press, 2022.

\bibitem{EhrenfeuchtRozenberg1978}
Andrew Ehrenfeucht and Gregorz Rozenberg.
\newblock Elementary homomorphisms and a solution of the {${\rm DOL}$} sequence
  equivalence problem.
\newblock {\em Theoret. Comput. Sci.}, 7(2):169--183, 1978.

\bibitem{Shallit1999}
David Hamm and Jeffrey Shallit.
\newblock Characterization of finite and onesided infinite fixed points of
  morphisms on free monoids.
\newblock Technical report, School of Computer science, Univerisity of
  Waterloo, 1999.
\newblock Technical Report CS-99-17.

\bibitem{HarjuLinna1986}
Tero Harju and Matti Linna.
\newblock On the periodicity of morphisms on free monoids.
\newblock {\em RAIRO Inform. Th\'{e}or. Appl.}, 20(1):47--54, 1986.

\bibitem{HeadLando1986}
Tom Head and Barbara Lando.
\newblock Fixed and stationary $\omega$-words and $\omega$-languages.
\newblock In Gregorz Rozenberg and Arto Salomaa, editors, {\em The Book of L},
  pages 147--156. Springer-Verlag, 1986.

\bibitem{HermanPutnamSkau1992}
Richard~H. Herman, Ian~F. Putnam, and Christian~F. Skau.
\newblock Ordered {B}ratteli diagrams, dimension groups and topological
  dynamics.
\newblock {\em Internat. J. Math.}, 3(6):827--864, 1992.

\bibitem{Honkala2009}
Juha Honkala.
\newblock On the simplification of infinite morphic words.
\newblock {\em Theoretical Computer Science}, 410(8):997 -- 1000, 2009.

\bibitem{LindMarcus1995}
Douglas Lind and Brian Marcus.
\newblock {\em An Introduction to Symbolic Dynamics and Coding}.
\newblock Cambridge University Press, Cambridge, 1995.
\newblock Second edition, 2021.

\bibitem{Lothaire1983}
M.~Lothaire.
\newblock {\em Combinatorics on {W}ords}.
\newblock Cambridge University Press, second edition, 1997.
\newblock (First edition 1983).

\bibitem{MaloneyRust2018}
Gregory~R. Maloney and Dan Rust.
\newblock Beyond primitivity for one-dimensional substitution subshifts and
  tiling spaces.
\newblock {\em Ergodic Theory Dynam. Systems}, 38:1086--1117, 2018.

\bibitem{Mosse1992}
Brigitte Moss{\'e}.
\newblock Puissances de mots et reconnaissabilit\'e des points fixes d'une
  substitution.
\newblock {\em Theoret. Comput. Sci.}, 99(2):327--334, 1992.

\bibitem{Mosse1996}
Brigitte Moss\'{e}.
\newblock Reconnaissabilit\'{e} des substitutions et complexit\'{e} des suites
  automatiques.
\newblock {\em Bull. Soc. Math. France}, 124(2):329--346, 1996.

\bibitem{mousavi2016automatic}
Hamoon Mousavi.
\newblock Automatic theorem proving in {W}alnut, 2016.

\bibitem{Pansiot1983}
Jean-Jacques Pansiot.
\newblock Hi\'erarchie et fermeture de certaines classes de tag-syst\`emes.
\newblock {\em Acta Informatica}, 20:179--196, 1983.

\bibitem{Pansiot1986}
Jean-Jacques Pansiot.
\newblock Decidability of periodicity for infinite words.
\newblock {\em RAIRO Inform. Th\'{e}or. Appl.}, 20(1):43--46, 1986.

\bibitem{Queffelec2010}
Martine Queff{\'e}lec.
\newblock {\em Substitution Dynamical Systems---Spectral Analysis}, volume 1294
  of {\em Lecture Notes in Mathematics}.
\newblock Springer-Verlag, Berlin, second edition, 2010.

\bibitem{RozenbergSalomaa1980}
Grzegorz Rozenberg and Arto Salomaa.
\newblock {\em The Mathematical Theory of $L$-systems}.
\newblock Academic Press, 1980.

\bibitem{Salo2017}
Ville Salo.
\newblock Decidability and universality of quasiminimal subshifts.
\newblock {\em Journal of Computer and System Sciences}, 89:288--314, 2017.

\bibitem{Shallit2021}
Jeffrey Shallit.
\newblock {\em The Logical Approach To Automatic Sequences: Exploring
  Combinatorics on Words with Walnut}, volume 482 of {\em London Mathematical
  Society Lecture Notes}.
\newblock Cambridge, 2023.

\bibitem{ShallitWang2002}
Jeffrey~O. Shallit and Ming{-}wei Wang.
\newblock On two-sided infinite fixed points of morphisms.
\newblock {\em Theor. Comput. Sci.}, 270(1-2):659--675, 2002.

\bibitem{Shimomura2018}
Takashi Shimomura.
\newblock A simple approach to substitution minimal subshifts.
\newblock {\em Topology and its Applications}, 260:203--214, 2019.

\bibitem{Vitanyi1974}
Paul M.~B. Vit{\'a}nyi.
\newblock On the size of {D0L} languages.
\newblock In Grzegorz Rozenberg and Arto Salomaa, editors, {\em L Systems},
  pages 78--92. Springer Berlin Heidelberg, Berlin, Heidelberg, 1974.

\bibitem{Wielandt1950}
Helmut Wielandt.
\newblock Unzerlegbare, nicht negative {M}atrizen.
\newblock {\em Math. Z.}, 52:642--648, 1950.

\bibitem{Yuasa2007}
Hisatoshi Yuasa.
\newblock Invariant measures for the subshifts arising from non-primitive
  substitutions.
\newblock {\em Journal d'Analyse Mathématique}, 102(1):143--180, 2007.

\end{thebibliography}

\end{document}